\documentclass{amsart}

\usepackage{amssymb}
\usepackage{graphicx}
\usepackage{tikz}
\usetikzlibrary{calc,shapes.geometric} 
\usepackage{tikz-cd}
\usepackage{verbatim}
\usepackage{dsfont}
\usepackage{comment}
\usepackage{epstopdf}
\usepackage{enumitem}

\theoremstyle{plain}
\newtheorem{thm}{Theorem}[section]
\newtheorem{cor}[thm]{Corollary}
\newtheorem{lem}[thm]{Lemma}
\newtheorem{prop}[thm]{Proposition}
\newtheorem*{quotepfthm}{Theorem A} 
\newtheorem*{quoteappbound}{Theorem B} 
\newtheorem*{quoteappasympt}{Theorem C}

\theoremstyle{definition}
\newtheorem{defn}[thm]{Definition}
\newtheorem{claim*}{Claim}
\newtheorem{eg}[thm]{Example}
\newtheorem*{eg*}{Example}
\newtheorem{rem}[thm]{Remark}

\theoremstyle{remark}

\newcommand{\R}{\mathbb{R}}

\newcommand{\C}{\mathbb{C}}

\newcommand{\N}{\mathbb{N}}
\newcommand{\Z}{\mathbb{Z}}

\newcommand{\cA}{\mathcal{A}}
\newcommand{\cB}{\mathcal{B}}
\newcommand{\cC}{\mathcal{C}}

\newcommand{\cG}{\mathcal{G}}

\newcommand{\cJ}{\mathcal{J}}

\newcommand{\cM}{\mathcal{M}}
\newcommand{\cP}{\mathcal{P}}
\newcommand{\cQ}{\mathcal{Q}}

\newcommand{\abs}[1]{\left\lvert #1 \right\rvert}
\newcommand{\norm}[1]{\left\lVert #1 \right\rVert}

\newcommand{\set}[2]{\left\{#1 \ : \ #2\right\}}

\newcommand{\conv}[1]{\underset{#1}\longrightarrow}
\DeclareMathOperator{\id}{id}
\DeclareMathOperator{\subspan}{span}

\DeclareMathOperator{\diam}{diam}
\DeclareMathOperator{\Var}{Var}

\DeclareMathOperator*{\esssup}{ess\,sup}
\DeclareMathOperator*{\essinf}{ess\,inf}

\newcommand\restr[2]{{
  \left.\kern-\nulldelimiterspace 
  #1 
  \vphantom{\big|} 
  \right|_{#2} 
  }}

\begin{document}

\title[Asymptotics for some second-largest Lyapunov exponents]{Asymptotics for the second-largest Lyapunov exponent for some Perron-Frobenius operator cocycles.} 

\author{Joseph Horan}
\address{Department of Mathematics and Statistics, University of Victoria, 3800 Finnerty Road, Victoria BC, V8P 5C2, Canada}
\email{jahoran@uvic.ca}
\urladdr{http://www.math.uvic.ca/~jahoran/}
\thanks{The author was partially funded by NSERC during the completion of this research.} 

\date{\today}


\begin{abstract}
Given a discrete-time random dynamical system represented by a cocycle of non-singular measurable maps, we may obtain information on dynamical quantities by studying the cocycle of Perron-Frobenius operators associated to the maps. Of particular interest is the second-largest Lyapunov exponent, $\lambda_2$, which can tell us about mixing rates and decay of correlations in the system. We prove a generalized Perron-Frobenius theorem for cocycles of bounded linear operators on Banach spaces that preserve and occasionally contract a cone; this theorem shows that the top Oseledets space for the cocycle is one-dimensional, and there is an readily computed lower bound for the gap between the largest Lyapunov exponent $\lambda_1$ and $\lambda_2$ (that is, an upper bound for $\lambda_2$ which is strictly less than $\lambda_1$). We then apply this theorem to the case of cocycles of Perron-Frobenius operators arising from a parametrized family of maps to obtain an upper bound on $\lambda_2$; to the best of our knowledge, this is the first time $\lambda_2$ has been upper-bounded for a family of maps. To do this, we utilize a new balanced Lasota-Yorke inequality. We also examine random perturbations of a fixed map with two invariant densities and show that as the perturbation is scaled back down to the unperturbed map, $\lambda_2$ is asymptotically linear in the scale parameter. Our estimates are sharp, in the sense that there is a sequence of scaled perturbations of the fixed map that are all Markov, such that $\lambda_2$ is asymptotic to $-2$ times the scale parameter.
\end{abstract}

\maketitle

\section{Introduction}
\label{sect:intro}

Let $(\Omega,\cB,\mu,\sigma)$ be an ergodic invertible probability-preserving transformation. Our focus will be discrete-time random dynamical systems governed by the cocycle of maps $T_{\omega}^{(n)} = T_{\sigma^{n-1}(\omega)}\circ \dots \circ T_{\omega}$ over $(\Omega,\cB,\mu,\sigma)$. Assuming that each map $T_{\omega}$ is measurable and non-singular over the measure space $(Y,\cA,\lambda)$, then each $T_{\omega}$ gives rise to a bounded linear operator on $L^1(\lambda)$ called the Perron-Frobenius operator associated to $T_{\omega}$, which for now we will denote by $L_{\omega}$. The Perron-Frobenius operator describes what happens to densities on $Y$ under the action of $T_{\omega}$, and the Perron-Frobenius operators associated to the cocycle $T_{\omega}^{(n)}$ are themselves a cocycle, $L_{\omega}^{(n)} = L_{\sigma^{n-1}(\omega)}\circ \dots \circ L_{\omega}$, acting on $L^1(\lambda)$.

When studying random dynamical systems, one might ask a number of questions: does there exist some sort of random invariant density? If so, is it unique? If so, do initial states of the system converge to it at some rate, and is that rate exponential? Can that rate be estimated, in some way? In the case of a single primitive matrix representing a Markov chain (thinking of the transition matrix as changing the probability density on the state space), the answer to all of the above questions is ``yes'', thanks to the classical Perron-Frobenius theorem \cite{frobenius,perron}. The theorem is applied to the (powers of the) single transition matrix, and the spectral data of the matrix yields the desired properties of the invariant density and the exponential rate of convergence to that density. For a single map on a space where the Perron-Frobenius operator acts on an infinite-dimensional space of densities, there are similar general theorems, often called Krein-Rutman theorems, after the generalization of the Perron-Frobenius theorem for compact positive operators by Krein and Rutman \cite{krein-rutman}. 

The key aspect in all of these theorems is the notion of \emph{positivity}. Roughly, the positive direction is indicated by a cone, and operators which preserve and contract the cone have very nice structure: a vector in the cone playing the role of the direction of largest growth, and a complementary subspace of non-positive vectors which grow more slowly than anything in the positive direction. This insight was notably abstracted by Garrett Birkhoff in the 1950s \cite{birkhoff-jentzsch}, and the study of abstract cones has played a role in the study of general topological vector spaces (for example, \cite{peressini,schaefer}).

In 1984, Keller \cite{keller-rate} discussed the link between the spectral theory of Perron-Frobenius operators acting on bounded variation functions and the rate of convergence of the underlying dynamical systems to an equilibrium. In 1994, Arnold et al.\ \cite{arnold-random-pf-matrices} proved a cocycle Perron-Frobenius theorem for a class of positive matrix cocycles arising in evolutionary biology, based on the Birkhoff cone technique, for the purpose of obtaining a random invariant density. In 1995, Liverani \cite{liverani-annals-doc,liverani-bv-doc} applied Birkhoff's technique to the study of piecewise-expanding $C^2$ dynamical systems, obtaining invariant densities and explicit decay of correlations for powers of a single map. In 1999, Buzzi \cite{buzzi-doc} extended Liverani's application to cocycles of maps, instead of a single map, to obtain decay of correlations for certain random dynamical systems where not every Perron-Frobenius operator is required to preserve a cone. In this case, the largest Lyapunov exponent is equal to $0$ and corresponds to an equivariant family of densities, which is the random generalization of an invariant density. Then, the decay of correlations is related to the second-largest Lyapunov exponent for the cocycle of operators, and its magnitude gives the logarithm of the rate of decay of correlations. However, the constants involved in the proof of Buzzi's result are not easily identifiable.

For random dynamical systems that are quasi-compact, the Multiplicative Ergodic Theorem (for example, \cite{flq-2,gt-quas-semi-inv}) yields equivariant families of subspaces, called Oseledets spaces, on which the cocycle of Perron-Frobenius operators has well-defined growth rates, given by the Lyapunov exponents. In the above situations, the Oseledets space corresponding to the zero Lyapunov exponent turns out to be one-dimensional, and there is a complementary equivariant family of subspaces on which the Perron-Frobenius operator cocycle, restricted to the spaces, has the second-largest Lyapunov exponent. Here, the Oseledets spaces can be interpreted as indicating ``coherent structures'' in the system, as described in \cite{flq-1}. When the spaces correspond to large Lyapunov exponents of the cocycle (when compared to local expansion and dispersion), the Lyapunov exponents describe how these parts of the system are ``slowly exponentially mixing'', in the sense that they are not invariant sets but they mix with the rest of the space more slowly than one would expect from local expansion.

We are therefore interested in finding an upper bound for the second Lyapunov exponent for the Perron-Frobenius cocycle corresponding to a random dynamical system $T_{\omega}^{(n)}$, as that would give us a minimal mixing rate for the system (or decay of correlations). The approach we take is to prove a generalized Perron-Frobenius theorem for a fairly general class of cocycles of operators on a Banach space preserving a cone. In particular, no compactness is required (as in \cite{mierczynski-pos}); however, we do require that almost every operator preserves the cone (which is a stronger restriction than what Buzzi has in the specific cases in \cite{buzzi-doc}). The setup we use allows us to obtain a measurable equivariant decomposition of the Banach space into an equivariant positive direction with the largest growth rate and an equivariant family of subspaces of non-positive vectors growing at the next largest growth rate. An important outcome of the theorem is a rigorous quantitative bound on the second-largest Lyapunov exponent for such cocycles that is computable without tracing constants along in the proof of the theorem itself. Another important aspect of the theorem is that the proof is completely independent of any Multiplicative Ergodic Theorem; hence, the hypotheses provide a checkable condition for quasi-compactness and thus a full Oseledets decomposition for the cocycle after applying an MET in the appropriate setting. A summarized version of the theorem follows; see Section \ref{sect:pf-thm} for more details, including required definitions and what is meant by measurable in our context. Moreover, Corollary \ref{cor:suff-cond-primitive} provides even simpler sufficient conditions for the existence of the quantities listed here in the hypotheses.

\begin{quotepfthm}
Let $(\Omega,\cB,\mu,\sigma,X,\norm{\cdot},L)$ be either a strongly measurable random dynamical system or a $\mu$-continuous random dynamical system, such that the function $\log^+\norm{L(1,\cdot)}$ is integrable. Let $\cC \subset X$ be a nice cone such that $L(1,\omega)\cC \subset \cC$ for all $\omega$. Suppose that there exists a positive measure subset $G_P$ of $\Omega$, a positive integer $k_P$, and a positive real number $D_P$ such that for all $\omega \in G_P$, $\diam_\theta\big(L(k_P,\omega)\cC\big) \leq D_P.$ Then there exists a $\sigma$-invariant set of full measure $\tilde{\Omega} \subset \Omega$ on which the following statements are true:
\begin{enumerate}
\item There exist measurable functions $v(\omega) \in X$ and $\eta(\omega,\cdot) \in X^*$ such that \[ X = \subspan_{\R}\{v(\omega)\} \oplus \ker(\eta(\omega,\cdot)) \] is a measurable equivariant decomposition, the Lyapunov exponent for $v(\omega)$ is $\lambda_1$, and all vectors in $\ker(\eta(\omega,\cdot))$ have Lyapunov exponent strictly less than $\lambda_1$, unless $\lambda_1 = -\infty$.
\item When $\lambda_1 > -\infty$, we have $\displaystyle \lambda_2 \leq \lambda_1 - \frac{\mu(G_P)}{k_P}\log\left( \tanh(\tfrac{1}{4}D_P)^{-1} \right).$
\end{enumerate}
\end{quotepfthm}

We remark specifically on the quantitative bound on the second-largest Lyapunov exponent. The set $G_P$ and the quantities $k_P$ and $D_P$ are dependent on the cone and the cocycle and can therefore be computed outside of the proof of the theorem. The theorem is therefore something of a black box for the bound and, subsequently, a minimal mixing rate or decay of correlations. Outside of cocycles of positive operators, this problem is quite difficult.

To demonstrate the use of the theorem, we apply it to the situation of a cocycle of piecewise expanding maps with a specific form; all of the maps $T_{\omega}$ are `` paired tent maps'' that act on $[-1,1]$ and leave $[-1,0]$, $[0,1]$ mostly invariant, except for ``leaking'' mass of size $\epsilon_1(\omega) \geq 0$ from $[-1,0]$ to $[0,1]$ and mass of size $\epsilon_2(\omega) \geq 0$ in the other direction (see Section \ref{sect:app-pf-ops} for the precise definitions, and Figure \ref{fig:leaking} for a simple schematic diagram). Maps like these have been considered by Gonz\'alez-Tokman et al.\ in \cite{gt-metastable}; in that work, the authors fix a single perturbation of a map that leaves two sets invariant and investigate properties of the invariant density and the eigenvector for the second-largest eigenvalue for the perturbed map, in terms of the two invariant densities for the unperturbed map. They find that the eigenvector corresponding to the second-largest eigenvalue is asymptotically a scalar multiple of the difference of the two invariant densities for the unperturbed map, which indicates a coherent structure related to the transfer of mass between the two parts of the space. In \cite{dolgopyat-wright}, Dolgopyat and Wright take a similar situation but analyze the restrictions of the map to parts of the space, where the ``leaking'' of mass is seen as holes in the system. Looking at these open systems, the largest eigenvalues have a particular form related directly to the sizes of the mass transfer/leaking (which are framed as transition probabilities of a related Markov chain).In our case, we are interested in generalizing these ideas to the non-autonomous setting, to see how random mass transfer impacts the value of the second-largest Lyapunov exponent for the cocycle of Perron-Frobenius operators (instead of just a single map).
\begin{figure}[tb]
\begin{tikzpicture}
\node[ellipse,minimum height = 1cm,draw,fill=black!20] (left) at (-2,0) [text centered] {$[-1,0]$};
\node[ellipse,minimum height = 1cm,draw,fill=gray!20] (right) at (2,0) [text centered] {$[0,1]$};
\draw (left) edge [->,bend left=30,"$\epsilon_1(\omega)$"] (right);
\draw (right) edge [->,bend left=30,"$\epsilon_2(\omega)$"] (left);
\end{tikzpicture}
\caption{Schematic of ``leaking'' behaviour, where $\epsilon_1(\omega)$ and $\epsilon_2(\omega)$ are generally small.}
\label{fig:leaking}
\end{figure}
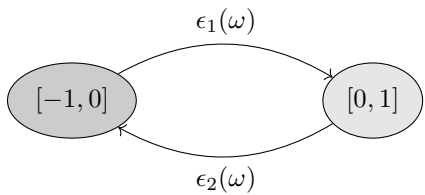

In the setting of these cocycles of paired tent maps, we are able to show that the hypotheses of Theorem A are true, taking the Banach space to be ($L^{\infty}$ equivalence classes of) bounded variation functions and finding a suitable cone that is preserved by all of the associated Perron-Frobenius operators. Thus we obtain an equivariant density for the cocycle, and an upper bound for the second-largest Lyapunov exponent in terms of $\epsilon_1$ and $\epsilon_2$ and quantities related to both the map and the cone. Next, we study the response of the system upon scaling $\epsilon_1$ and $\epsilon_2$ by some parameter $\kappa$ and taking $\kappa$ to $0$, which simulates shrinking a perturbation of the map $T_{0,0}$ back towards the original map. In this way, we can see how the second Lyapunov exponent behaves under perturbations; one might hope that it shrinks as a nice function of $\kappa$ (linear, say), until at $\kappa = 0$ the top Oseledets space becomes two-dimensional (spanned by the two invariant densities of $T_{0,0}$) and the zero Lyapunov exponent obtains multiplicity two. Our results are outlined in the following theorems, with precise statements to follow in the body of the paper. 

\begin{quoteappbound}
Let $(\Omega,\cB,\mu,\sigma)$ be an ergodic, invertible, probability-preserving transformation, and let $\epsilon_1, \epsilon_2 : \Omega \to [0,1]$ be measurable functions which are both not $\mu$-a.e.\ equal to $0$ and which both have countable range. Let $T_{\omega} = T_{\epsilon_1(\omega),\epsilon_2(\omega)}$ be defined as above. Then there exists a readily computed number $C$ such that \[ \lambda_2 \leq C < 0 = \lambda_1, \] where $\lambda_1$ and $\lambda_2$ are the largest and second-largest Lyapunov exponents for the cocycle of Perron-Frobenius operators associated to $T_{\omega}^{(n)}$.
\end{quoteappbound}

\begin{quoteappasympt}
Let $(\Omega,\mu,\sigma)$, $\epsilon_1$, and $\epsilon_2$ be as in Theorem B. Let $\kappa \in (0,1]$, and consider the cocycle of maps $T^{(n)}_{\kappa\epsilon_1(\omega),\kappa\epsilon_2(\omega)}$. Then there exists $c > 0$ such that for sufficiently small $\kappa$, the second-largest Lyapunov exponent $\lambda_2(\kappa)$ for the cocycle of Perron-Frobenius operators satisfies \[ \lambda_2(\kappa) \leq -c\kappa. \] This estimate is sharp, in the following sense. Set $\epsilon_1 = \epsilon_2 = 1$ for all $\omega$. Then there is a sequence $(\kappa_n)_{n=1}^{\infty} \subset (0,1/2)$ such that $\kappa_n \to 0$, each $T_{\kappa_n,\kappa_n}$ is Markov, and $\lambda_2(\kappa_n)$ is asymptotically equivalent to $-2\kappa_n$.
\end{quoteappasympt}

We emphasize that these results apply to an entire parametrized family of maps, and thus they give a general statement on the asymptotic properties of the second-largest Lyapunov exponent for these maps; to the best of our knowledge, this is the first time $\lambda_2$ has been upper-bounded for a family of maps, with an asymptotic estimate on the order of the bound in the scaling parameter.. Note also that Theorems B and C are consequences of Theorem A, applied to different quantities $k_P,G_P,D_P$. The primary work done, outside of showing that the hypotheses of Theorem A are satisfied, is to obtain expressions for each of those quantities.

In the process of applying Theorem A to the cocycle of Perron-Frobenius operators associated to the paired tent maps, we happen to require a new Lasota-Yorke-type inequality for Perron-Frobenius operators acting on bounded variation functions. Its utility comes from being sufficiently strong to force small coefficients of the variation terms, but balanced in such a way as to provide uniform bounds on both terms over a family of maps, not just one map individually. The inequality is based on Rychlik's work \cite{rychlik}; we prove the inequality in a similar level of generality, to provide a tool for future work. For details, see Section \ref{sect:ly-ineq}.

The remainder of the paper is as follows. In Section \ref{sect:pf-thm}, we give some required background on cones and measurability before stating and proving our cocycle Perron-Frobenius theorem. In Section \ref{sect:ly-ineq}, we briefly set up, state, and prove a new balanced Lasota-Yorke-type inequality. In Section \ref{sect:app-pf-ops}, we use that new Lasota-Yorke inequality to apply our cocycle Perron-Frobenius theorem to cocycles of paired tent maps as described above, to prove the aforementioned bound in Theorem B on the second-largest Lyapunov exponents for the Perron-Frobenius operators, and then find the perturbation estimate in Theorem C. For the proof of the sharpness of that estimate, we outline an example that provides a partial answer to a related but possibly harder question: what is a \emph{lower} bound for the second-largest Lyapunov exponent? The example is a specific class of Markov paired tent maps that turn out to be very amenable to analysis via standard finite-dimensional linear algebra techniques, and allow for explicit computation of the second-largest Lyapunov exponents (through eigenvalues). There is an appendix containing miscellaneous technical results that are used in various places in the paper.

\section{Cocycle Perron-Frobenius Theorem}
\label{sect:pf-thm}

In this section, our goal is to generalize the classical Perron-Frobenius theorem to the setting of measurable cocycles of bounded linear operators on a Banach space that preserve and contract a cone. The classical theorem is often stated for primitive matrices (those with non-negative entries and such that some power of the matrix has all positive entries) and without explicit mention of a cone; the cone being preserved and contracted is, in this special case, those vectors with non-negative entries, and this cone has many nice properties. Before we state and prove the theorem, we will briefly describe the types of cones that we use and recall some related quantities and results. In addition, because we are in a measurable dynamics setting, we will describe the choices of topologies and $\sigma$-algebras on the spaces being considered and for the related spaces of linear operators and functionals.

\subsection{Preliminaries on Cones and Measurability}
\label{subsect:prelims-cones}

\begin{defn}
\label{defn:cones}
Let $(X,\norm{\cdot})$ be a real Banach space. A \emph{cone} is a set $\cC \subset X$ that is closed under scalar multiplication by positive numbers, i.e.\ $\lambda\cC \subset \cC$ for all $\lambda > 0$. In this paper, a \emph{nice} cone is a cone $\cC \subset X$ that has the following properties, mimicking the positive orthant in $\R^d$:
\begin{itemize}
\item $\cC$ is convex (equivalently, closed under addition);
\item $\cC$ is \emph{blunt}, i.e.\ $0\notin \cC$;
\item $\cC$ is \emph{salient}, i.e.\ $\cC\cap (-\cC) = \emptyset$ (more generally $\subset \{0\}$);
\item $\cC\cup\{0\}$ is closed;
\item $\cC$ is \emph{generating} (or \emph{generates} $X$), i.e.\ $\cC - \cC = X$;
\item $\cC$ is $D$-\emph{adapted}, i.e.\ there exists $D \in \R_{\geq 1}$ such that for $x\in X$ and $y\in \cC$, if $y\pm x \in \cC\cup\{0\}$, then $\norm{x} \leq D\norm{y}$.
\end{itemize}
Note that a convex cone $\cC$ with $\cC\cup\{0\}$ closed and salient induces a partial order on $X$, denoted $\preceq_{\cC}$ (or $\preceq$ when the choice of $\cC$ is clear), by $x \preceq y$ if and only if $y-x\in \cC\cup\{0\}$. Then the $D$-adapted condition is rephrased as saying if $-y \preceq x \preceq y$, then $\norm{x} \leq D\norm{y}$ (note that this inequality forces $y\in \cC$).
\end{defn}

\begin{rem}
\label{rem:cones}
Note that if $\cC$ is any cone in $X$ with an interior point $x$, then $\cC$ generates $X$ by taking $z = \epsilon^{-1}(x+\epsilon z)- \epsilon^{-1}x$ for sufficiently small positive $\epsilon$. If $\cC$ induces a lattice order on $X$, then $\cC$ also generates $X$, by taking $z = z^+ - z^-$.

In the literature, if there is some $D\geq 1$ such that $\cC$ is $D$-adapted in $(X,\norm{\cdot})$, then $\cC$ is \emph{normal} and there exists an equivalent norm to $\norm{\cdot}$ such that $\cC$ is $1$-adapted with respect to that norm \cite[Section V.3]{schaefer}. We will not use this feature, opting instead to run the proofs using the existing norm (the equivalent norm is a Minkowski functional for an appropriate saturated convex set).
\end{rem}

\begin{eg}
\label{eg:nice-cone-fin-dim}
For $X = \R^d$ equipped with $\norm{x}_1 = \sum_i \abs{x_i}$, the positive orthant $\cC = \R_{\geq 0}^d\setminus \{0\}$ is a nice cone. It is easy to see that $\cC$ is $1$-adapted: if $-y \preceq x \preceq y$, then $-y_i \leq x_i \leq y_i$ for all $i$. Thus $\abs{x_i} \leq y_i$ and summing over $i$ yields $\norm{x}_1 \leq \norm{y}_1$.
\end{eg}

\begin{eg}
\label{eg:nice-cone-bv}
Let $X = BV(I) \subset L^{\infty}(I)$ be the space of (almost-everywhere equivalence classes of) bounded variation functions on some totally ordered, order-complete set $I$ equipped with a probability measure; for example, $I = [c,d] \subset \R$ with normalized Lebesgue measure. Equip $X$ with the norm $\norm{f}_{BV} = \norm{f}_1 + \tilde{\Var}(g)$, where $\tilde{\Var}(f) = \inf_{g=f\ a.e.} \Var(g)$, and for $a>0$ let \[ \cC_a = \set{f\in BV(I)}{f\geq 0,\ \tilde{\Var}(f) \leq a\norm{f}_1}\setminus\{0\}. \] Then $\cC_a$ is a nice cone, and is $D$-adapted with $D = 2a+1$. The only non-trivial properties to see are that $\cC_a$ generates $X$ and that $\cC_a$ is $(2a+1)$-adapted. For the former, observe that $\mathds{1}$ is an interior point for $\cC_a$. For the latter, note that $\tilde{\Var}$ satisfies the triangle inequality on $BV(I)$ and that if $-f \preceq g \preceq f$, then $\norm{g}_1 \leq \norm{f}_1$ and $f-g \in \cC_a$.
\end{eg}

\begin{defn}
\label{defn:proj-metric}
Let $(X,\norm{\cdot},\cC)$ be a real Banach space with a nice cone. For $v,w\in \cC$, define:
\begin{gather*}
\alpha(v,w) = \sup\set{\lambda \geq 0}{\lambda v \preceq w}; \\
\beta(v,w) = \inf\set{\mu \geq 0}{w \preceq \mu v}; \\
\theta(v,w) = \log\left(\frac{\beta}{\alpha}\right).
\end{gather*}
The quantity $\theta(v,w)$ is the \emph{projective metric} on $\cC$, and is sometimes called a \emph{Hilbert} metric on $\cC$. If $f,g\in \cC$ have $\theta(f,g) < \infty$, then we say that $f$ and $g$ are \emph{comparable}.
\end{defn}

The quantity $\theta$ is symmetric and satisfies the triangle inequality, but $\theta(v,w) = 0$ if and only if $v$ and $w$ are collinear. Moreover, $\theta$ can be infinite: if $v,w\in \cC$, then $\theta(v,w-\alpha(v,w)v) = \infty$. Thus $\theta$ is a pseudo-metric that can be infinite, so on $\theta$-connected components of $\cC$ it restricts to a metric, but in general $\diam_\theta(\cC) = \infty$. We summarize the facts that we will use about $\alpha$, $\beta$, and $\theta$ in the next lemma; the facts about $\alpha$ and $\beta$ are straightforward to prove from the definitions, and proof of the contraction inequality for $\theta$ can be found in \cite[Section 1]{liverani-annals-doc}.

\begin{lem}
\label{lem:proj-metric}
Let $(X,\norm{\cdot},\cC)$ be a real Banach space with a nice $D$-adapted cone. Then in the second component, $\alpha$ is super-additive, $\beta$ is sub-additive, and both are positive-scalar-homogeneous, i.e.\ for all $v,w,z\in \cC$ and $c > 0$:
\begin{gather*}
\alpha(v,w+z) \geq \alpha(v,w) + \alpha(v,z), \\
\beta(v,w+z) \leq \beta(v,w) + \beta(v,z), \\
\alpha(v,cw) = c\alpha(v,w), \quad \beta(v,cw) = c\beta(v,w).
\end{gather*}
We also have symmetry properties of $\alpha,\beta,\theta$; for all $v,w\in\cC$:
\begin{gather*}
\alpha(v,w) = \beta(w,v)^{-1}, \\
\theta(v,w) = \theta(w,v).
\end{gather*}
If $v,w\in \cC$ with $\theta(v,w) = 0$, then $v$ and $w$ are collinear. For all $v,w\in \cC$ we have \[ \alpha(v,w) \leq D\frac{\norm{w}}{\norm{v}}, \quad \beta(v,w) \geq \frac{1}{D}\frac{\norm{v}}{\norm{w}}. \] Now, let $L \in \cB(X)$ be a bounded linear operator such that $L\cC \subset \cC$. Then for all $v,w\in \cC$:
\begin{gather*}
\alpha(L(v),L(w)) \geq \alpha(v,w), \\
\beta(L(v),L(w)) \leq \beta(v,w), \\
\theta(L(v),L(w)) \leq \tanh\left(\tfrac{1}{4}\diam_\theta\left(L\cC\right)\right)\theta(v,w),
\end{gather*}
where if the $\theta$-diameter of $L\cC$ is infinite the scale factor is $1$. If $v,w\in\cC$ are comparable, then $\alpha(v,w)$ and $\alpha(w,v) > 0$, and $\beta(v,w)$ and $\beta(w,v) < \infty$. Finally, considering $\cC\times\cC$ as a subset of $X\times X$ with the norm $\norm{(x,y)}_1 = \norm{x}+\norm{y}$ and equipping $\cC\times\cC$ with the restriction of the Borel $\sigma$-algebra on $X\times X$, we see that $\alpha$ is upper-semi-continuous, $\beta$ and $\theta$ are lower-semi-continuous, and all three are Borel measurable, into $([0,\infty],\sigma(\tau_{\abs{\cdot}}))$.
\end{lem}

Because we will be considering functions from a measure space to spaces which may be equipped with multiple topologies, we must decide on the topologies and $\sigma$-algebras on the relevant spaces. 
\begin{itemize}
\item For a Banach space $(X,\norm{\cdot})$, we use the norm topology and the associated Borel $\sigma$-algebra.
\item For the dual of a Banach space, $X^*$, the two main options for topologies are the weak-* topology $\tau_{w^*}$ and the norm topology, where the norm is denoted by $\norm{\cdot}_{X^*}$ if clarity is required.
\item For the bounded linear operators on a Banach space, $\cB(X)$, we will use either the strong operator topology $\tau_{\mathrm{SOT}}$ or the norm topology, $\norm{\cdot}_{\mathrm{op}}$.
\end{itemize}
In general, we will write $\sigma(\tau)$ to denote the Borel $\sigma$-algebra generated by a topology $\tau$. Observe that the weak-* topology and the strong operator topology are both topologies of pointwise convergence.

We will need a standard extension lemma like the following.

\begin{lem}
\label{lem:ext-pos}
Let $\cC \subset X$ be a closed convex cone in a real Banach space, and let $\eta : \cC \to \R$ be a positive, positive-scalar-homogeneous, additive function on $\cC$. Then $\eta$ extends uniquely to a linear functional on $\cC - \cC \subset X$, by setting $\eta(g_1-g_2) = \eta(g_1) - \eta(g_2)$.
\end{lem}

The next Proposition provides the key relations between $\theta$ and $\norm{\cdot}$ on $\cC$. The first part is an inequality that will be used repeatedly, and the second is the generalization of the same statement but for $D=1$, found in \cite{liverani-annals-doc} as Lemma 1.3. 

\begin{prop}
\label{prop:D-adapt-Cauchy}
Let $(X,\norm{\cdot},\cC)$ be a real Banach space with a nice $D$-adapted cone. 
\begin{enumerate}
\item If $f,g\in\cC$ are comparable, then \[ \norm{f-\left( \frac{\alpha(g,f)+\beta(g,f)}{2} \right)g} \leq \frac{D}{2}\norm{g}\alpha(g,f)\left( e^{\theta(f,g)}-1 \right). \] 
\item If $f,g\in\cC$ with $\norm{f} = \norm{g} = r$, then $\displaystyle \norm{f-g} \leq D^2r\left( e^{\theta(f,g)}-1 \right).$ 
\end{enumerate}
Thus if $(f_n)_n \subset \cC$ is a $\theta$-Cauchy sequence of elements with the same norm, then $(f_n)_n$ is Cauchy in norm, hence convergent.
\end{prop}

\begin{proof}
First, suppose that $f,g\in\cC$ are comparable. Subtract $\tfrac{1}{2}(\alpha(g,f)+\beta(g,f))g$ from the inequality \[ \alpha(g,f)g \preceq f \preceq \beta(g,f)g \] to get \[ -\frac{\beta(g,f)-\alpha(g,f)}{2} g \preceq f- \frac{\alpha(g,f)+\beta(g,f)}{2} g \preceq \frac{\beta(g,f)-\alpha(g,f)}{2} g. \] Use the $D$-adapted condition to obtain:
\begin{align*}
\norm{f-\left( \frac{\alpha(g,f)+\beta(g,f)}{2} \right)g} & \leq D\norm{\frac{\beta(g,f)-\alpha(g,f)}{2} g} \\
& = \frac{D}{2}\norm{g}\alpha(g,f)\left( e^{\theta(g,f)}-1 \right).
\end{align*}

Now, suppose that $f,g\in\cC$ with $\norm{f} = \norm{g} = r$. If $f$ and $g$ are not comparable, then the norm bound in the Proposition is trivial, since the right-hand-side is infinite. Thus, assume that $f$ and $g$ are comparable. Using the reverse triangle inequality, we have:
\begin{align*} 
r\cdot\abs{1-\frac{\alpha(g,f)+\beta(g,f)}{2}} & = \abs{\norm{f} - \norm{\left( \frac{\alpha(g,f)+\beta(g,f)}{2} \right)g}} \\
& \leq \norm{ f-\left( \frac{\alpha(g,f)+\beta(g,f)}{2} \right)g }.
\end{align*} 
We apply the first part of the Proposition, the upper bound for $\alpha$ in Lemma \ref{lem:proj-metric}, and the triangle inequality to obtain the desired inequality:
\begin{align*}
\norm{f-g} & \leq \norm{f-\left( \frac{\alpha(g,f)+\beta(g,f)}{2} \right)g} + \norm{\left( \frac{\alpha(g,f)+\beta(g,f)}{2} \right)g-g} \\
& = \norm{f-\left( \frac{\alpha(g,f)+\beta(g,f)}{2} \right)g} + r\cdot\abs{1-\frac{\alpha(g,f)+\beta(g,f)}{2}} \\
& \leq 2\norm{f-\left( \frac{\alpha(g,f)+\beta(g,f)}{2} \right)g} \\
& \leq D\norm{g}\alpha(g,f)\left( e^{\theta(g,f)}-1 \right) \leq D^2r\left( e^{\theta(g,f)}-1 \right).
\end{align*}

From this inequality, it is clear that if $(f_n)_n$ is sequence of cone elements that is Cauchy in $\theta$, then it is Cauchy in norm.
\end{proof}

We will need to consider measurability and continuity of maps from a probability space (with and without a topology) into the dual space of $X$. Because we have a cone in $X$ that generates $X$, we can obtain these properties by looking at a space of functions on the cone that contains $X^*$ and has an appropriate norm. This space is essentially bounded functions on $S(0,1)\cap \cC$.

\begin{defn}
\label{defn:bpsh}
Let $(X,\norm{\cdot},\cC)$ be a real Banach space with a nice cone. The set of norm-bounded, positive-scalar-homogeneous functions $f : \cC \to \R$ is denoted by $\mathrm{BPSH}(\R^{\cC})$, and precisely given by \[ \mathrm{BPSH}(\R^{\cC}) := \set{f : \cC \to \R}{f(cx) = cf(x) \text{ for all } x\in \cC, c>0; \norm{f}_{\R^\cC} < \infty} \] where $\norm{f}_{\R^{\cC}} = \sup\set{\abs{f(x)}}{x\in\cC,\norm{x} \leq 1}$.
\end{defn}

The following lemma, by And\^o \cite[Lemma 1]{ando}, indicates that when a cone generates the Banach space, every vector has a bounded (not necessarily continuous or even measurable) decomposition into cone vectors. This fact will be used multiple times throughout the paper, and is the main tool in the easy proof of the immediately subsequent lemma.

\begin{lem}[And\^{o}]
\label{lem:ando}
Let $X$ be a Banach space, and let $\cC$ be a closed convex cone in $X$. Then the following are equivalent:
\begin{enumerate}
\item $X = \cC - \cC$, i.e.\ $\cC$ generates $X$; 
\item there exists $K > 0$ such that for all $x\in X$, there are $x^+,x^- \in \cC$ such that $x = x^+ - x^-$ and $\norm{x^{\pm}} \leq K\norm{x}$.
\end{enumerate}
\end{lem}

\begin{lem}
\label{lem:bpsh-dual}
Let $(X,\norm{\cdot},\cC)$ be a real Banach space with a nice cone. Then $(\mathrm{BPSH}(\R^{\cC}),\norm{\cdot}_{\R^{\cC}})$ is a normed vector space that contains $(X^*,\norm{\cdot}_{X^*})$. If $\phi \in X^*$, then \[ \norm{\phi}_{\R^\cC} \leq \norm{\phi}_{X^*} \leq 2K\norm{\phi}_{\R^\cC}, \] where $K$ is the constant from And\^o's Lemma. The norm topology on $X^*$ is the same as the restriction of the norm topology on $\mathrm{BPSH}(\R^{\cC})$. 
\end{lem}

We will be using two measurability hypotheses for our main theorem: strong measurability in the case where our Banach space $(X,\norm{\cdot})$ is separable, and $\mu$-continuity in the case where $(X,\norm{\cdot})$ is not.

\begin{defn}
\label{defn:strong-meas}
Let $(\Omega,\cB,\mu)$ be a probability space, and let $(X,\norm{\cdot}_X)$, $(Y,\norm{\cdot}_Y)$ be normed linear spaces. A family of bounded linear operators $L : \Omega \to \cB(X,Y)$ is \emph{strongly measurable} when it is measurable with respect to the Borel $\sigma$-algebra generated by the strong operator topology on $\cB(X,Y)$ (that is, pointwise convergence). When $(X,\norm{\cdot}_X)$ and $(Y,\norm{\cdot}_Y)$ are separable, then $L$ is strongly measurable if and only if $L_{\omega}(x)$ is Borel measurable for all $x\in X$.
\end{defn}

The definition of strong measurability specifically applies to two main cases: $Y = \R$, so that $\cB(X,Y) = X^*$, and $Y = X$, so that $\cB(X,X) = \cB(X)$.

\begin{defn}
\label{defn:mu-cts}
Let $(\Omega,\cB,\mu)$ be a Borel probability space over a Borel subset of a Polish space, let $(Y,\tau)$ be a topological space, and let $f: \Omega \to Y$ be a function. We say that $f$ is $\mu$-\emph{continuous} when there exists an increasing sequence of compact sets $K_n$ such that $\mu\left(\bigcup_n K_n\right) = 1$ and on each $K_n$, $f$ is continuous. 
\end{defn}

For Borel probability spaces over a Borel subset of a Polish space, this definition is equivalent to requesting that there are only measurable sets $A_n$ such that $\mu\left(\bigcup_n A_n\right) = 1$ and on which $f$ is continuous, because $\mu$ is tight (measurable sets can be approximated from inside by compact sets) and $\Omega$ is normal. This weaker condition is usually taken to be the definition of $\mu$-continuity.

The following lemma is a reworking of the well-known fact that a limit of measurable functions into a metric space is measurable. The proof is essentially the same as for Egoroff's theorem.

\begin{lem}
\label{lem:limit-mu-cts}
Let $(\Omega,\cB,\mu)$ be a Borel probability space over a Borel subset of a Polish space, let $(X,d)$ be a metric space, and let $f_n : \Omega \to X$ be a sequence of $\mu$-continuous functions. Suppose that $f_n$ converges pointwise to $f : \Omega \to X$. Then $f$ is also $\mu$-continuous.
\end{lem}

Lastly, we need the notion of tempered functions. The name comes from tempered distributions, which have subexponential growth.

\begin{defn}
\label{defn:tempered}
Let $(\Omega,\cB,\mu,\sigma)$ be a invertible probability-preserving transformation, and let $f: \Omega\to \R$. We say that $f$ is \emph{tempered} when \[ \lim_{n\to\pm\infty} \frac{1}{n}\log\abs{f(\sigma^{n}(\omega))} = 0. \]
\end{defn}

\subsection{Statement of the Main Theorem}
\label{subsect:pf-thm-statement}

The main theorem is actually two theorems with the same conclusion outside of measurability, and drastically different measurability assumptions. Due to this fact, we will describe the two sets of measurability hypotheses first, and then state the theorem.

\begin{defn}[Measurability Assumptions on Cocycles]
\label{defn:meas-assump}
Let $(\Omega,\cB,\mu,\sigma)$ be an ergodic, invertible, probability-preserving transformation.
\begin{enumerate}
\item Let $(X,\norm{\cdot})$ be a real separable Banach space, with bounded linear operators $\cB(X)$, and $L : \Omega \to \cB(X)$ a strongly measurable map. Then the cocycle $L(n,\omega) = L(1,\sigma^{n-1}(\omega))\cdots L(1,\omega)$ is called a \emph{strongly measurable} cocycle over $(\Omega,\cB,\mu,\sigma)$. We will call the tuple $(\Omega,\cB,\mu,\sigma,X,\norm{\cdot},L)$ a \emph{strongly measurable random dynamical system}.
\item Suppose that $(\Omega,\tau)$ is homeomorphic to a Borel subset of a Polish space, $\mu$ is a Borel measure on $(\Omega,\sigma(\tau))$, and $\sigma : \Omega \to \Omega$ is a homeomorphism. Let $(X,\norm{\cdot})$ be a real Banach space (not necessarily separable), and let $L : \Omega \to \cB(X)$ be $\mu$-continuous with respect to the norm topology on $\cB(X)$. Then the cocycle $L(n,\omega) = L(1,\sigma^{n-1}(\omega))\cdots L(1,\omega)$ is called a $\mu$-\emph{continuous cocycle} over $(\Omega,\sigma(\tau),\mu,\sigma)$. We will call the tuple $(\Omega,\cB,\mu,\sigma,X,\norm{\cdot},L)$ a $\mu$-\emph{continuous random dynamical system}.
\end{enumerate}
\end{defn}

\begin{thm} 
\label{thm:pf}
Let $(\Omega,\cB,\mu,\sigma,X,\norm{\cdot},L)$ be either a strongly measurable random dynamical system or a $\mu$-continuous random dynamical system, such that $\log^+\norm{L(1,\cdot)} \in L^{1}(\mu)$. Let $\cC \subset X$ be a nice cone such that $L(1,\omega)\cC \subset \cC$ for all $\omega$. Suppose that there exists a positive measure subset $G_P$ of $\Omega$, a positive integer $k_P$, and a positive real number $D_P$ such that for all $\omega \in G_P$, $\diam_\theta\big(L(k_P,\omega)\cC\big) \leq D_P.$ Then there exists a $\sigma$-invariant set of full measure $\tilde{\Omega} \subset \Omega$ on which the following statements are true:

\begin{enumerate}
\item \label{thm:pf-v} There exists a unique function $v : \tilde{\Omega} \to \cC$, and a positive measurable function $\phi : \tilde{\Omega} \to \R_{> 0}$ such that
\begin{itemize}
\item $\norm{v(\omega)} = 1$,
\item $\log^+(\phi)\in L^1(\mu)$, and
\item $L(1,\omega)v(\omega) = \phi(\omega)v(\sigma(\omega)).$
\end{itemize} 

\item \label{thm:pf-eta} There exists a family of bounded linear functionals $\omega \mapsto \eta(\omega,\cdot) : \tilde{\Omega} \to X^*$ such that
\begin{itemize}
\item $\eta(\omega,\cdot)$ is strictly positive with respect to $\cC$,
\item $\eta(\sigma(\omega),L(1,\omega)x) = \phi(\omega)\eta(\omega,x)$ for all $x\in X$ and $\omega\in \tilde{\Omega}$,and
\item $\eta(\omega,v(\omega)) = 1$.
\end{itemize} 
Thus $X = \subspan_{\R}\{v(\omega)\} \oplus \ker(\eta(\omega,\cdot))$ is an equivariant decomposition of $X$ with respect to $L(n,\omega)$, and $\pi_{\omega} : X\to X$ given by $\pi_{\omega}(x) = \eta(\omega,x)v(\omega)$ is the continuous projection onto $\subspan_{\R}\{v(\omega)\}$.

\item \label{thm:pf-meas} If $(\Omega,\cB,\mu,\sigma,X,\norm{\cdot},L)$ is strongly measurable, then:
\begin{itemize}
\item $v$ is $\cB$-$\sigma(\norm{\cdot})$ measurable,
\item $\eta(\omega,\cdot)$ is strongly measurable,
\item the projection operators $\pi_\omega$, $I-\pi_\omega$ are strongly measurable, and
\item $\subspan_{\R}\{v(\omega)\}$ is $\cB$-$\sigma(d)$ measurable.
\end{itemize}
If $(\Omega,\cB,\mu,\sigma,X,\norm{\cdot},L)$ is $\mu$-continuous, then:
\begin{itemize}
\item $v$ is $\mu$-continuous with respect to $\norm{\cdot} = \norm{\cdot}_{X}$,
\item $\eta(\omega,\cdot)$ is $\mu$-continuous with respect to $\norm{\cdot}_{X^*}$,
\item the projection operators $\pi_\omega$ and $I-\pi_\omega$ are $\mu$-continuous with respect to $\norm{\cdot}_{\text{op}}$, and
\item $\subspan_{\R}\{v(\omega)\}$ and $\ker(\eta(\omega,\cdot))$ are $\mu$-continuous with respect to the distance on the Grassmannian.
\end{itemize}

\item \label{thm:pf-le} The Lyapunov exponent for $v(\omega)$ is 
\begin{align*} 
\lim_{n\to\infty} \frac{1}{n}\log\norm{L(n,\omega)v(\omega)} & = \lim_{n\to\infty} \frac{1}{n}\sum_{i=0}^{n-1} \log(\phi(\sigma^i(\omega))) \\ 
& = \lambda_1 = \lim_{n\to\infty} \frac{1}{n}\log\norm{L(n,\omega)},
\end{align*} 
where $\lambda_1$ is the maximal Lyapunov exponent for $L(n,\omega)$ (possibly $-\infty$), and we have 
\begin{equation*}
\begin{split}
\limsup_{n\to\infty} \frac{1}{n}&\log\norm{\restr{L(n,\omega)}{\ker(\eta(\omega,\cdot))}} - \frac{1}{n}\sum_{i=0}^{n-1} \log(\phi(\sigma^i(\omega))) \\
& \leq -\frac{\mu(G_P)}{k_P}\log\left(\tanh\left(\frac{1}{4}D_P\right)^{-1}\right) < 0.
\end{split}
\end{equation*} If $\lambda_1 > -\infty$, then $\lambda_1$ has multiplicity $1$ with $\displaystyle \lambda_1 = \int_\Omega \log\phi\ d\mu$, and the projection operators $\pi_\omega$ and $I-\pi_\omega$ are norm-tempered in $\omega$.
\end{enumerate}
\end{thm}

\begin{rem}
\label{rem:pf-met-link}
Theorem \ref{thm:pf} will be proved without any reference to any version of the Multiplicative Ergodic Theorem. It therefore provides a way to show quasi-compactness of certain dynamical systems and verify the hypotheses of the MET in order to use it. Moreover, in this case, the ``top'' or ``fast'' space, the equivariant space along which all non-zero vectors grow at the fastest rate $\lambda^*$, is the one-dimensional span of $v(\omega)$, and the ``slow'' space, the equivariant space along which all vectors grow at a rate slower than $\lambda^*$, is the kernel of $\eta(\omega,\cdot)$.
\end{rem}

In addition, we state a corollary that provides a sufficient condition for the existence of the set $G_P$ and quantities $k_P$ and $D_P$ in the hypotheses of Theorem \ref{thm:pf}. This condition is a generalization of the primitivity condition in the classical Perron-Frobenius theorem, where a power of a matrix with non-negative entries has all positive entries. Instead, we require that over a positive measure set of $\omega$, $L(n,\omega)$ eventually strictly contracts a cone.

\begin{cor}
\label{cor:suff-cond-primitive}
Let $(\Omega,\cB,\mu,\sigma,X,\norm{\cdot},L)$ be either a strongly measurable random dynamical system or a $\mu$-continuous random dynamical system. Let $\cC \subset X$ be a nice cone such that $L(1,\omega)\cC \subset \cC$ for all $\omega$. Suppose that \[ n_P(\omega) = \inf\set{k \geq 1}{\diam_\theta(L(k,\omega)\cC) < \infty} \] is finite on a set of positive measure. Then $n_P$ is finite $\mu$-almost everywhere, and there exists a positive measure subset $G_P$ of $\Omega$, a positive integer $k_P$, and a positive real number $D_P$ such that for all $\omega \in G_P$, \[ \diam_\theta\big(L(k_P,\omega)\cC\big) \leq D_P. \] Thus, if in addition $\log^+\norm{L(1,\cdot)} \in L^{1}(\mu)$, then Theorem \ref{thm:pf} applies.
\end{cor}

The proof of the corollary is simply to observe that $\omega \mapsto \diam_\theta(L(k,\omega)\cC)$ is measurable for each $k$ (allowing the value $\infty$). Then, thanks to the assumption on $n_P$, for some $k_P \geq 1$ it is bounded above by $D_P$ on a set $G_P$ of positive measure, and this choice of constants satisfies the hypotheses of Theorem \ref{thm:pf}.

\subsection{Proof of Theorem \ref{thm:pf}} 
\label{subsect:proof-pf-thm}

The key ingredient of most Perron-Frobenius-type theorems is contraction of a cone, or a family of cones. The first lemma shows that the $\theta$-diameter of the image of the cone under an iterate of the cocycle is a measurable function. The following proposition gives a quantitative estimate on the contraction of the cone in terms of its $\theta$-diameter along both forwards and backwards orbits of $\sigma$. The lemma afterwards establishes a minimum rate of contraction. The first proof relies on the hypotheses placed on the random dynamical system. The latter two proofs are entirely combinatorial and ergodic theoretic, in the sense that we only use the ergodic properties of $\sigma$ and the algebraic and order-theoretic properties of the cone. 

\begin{lem}
\label{lem:diam-meas}
Let $(\Omega,\cB,\mu,\sigma,X,\norm{\cdot},L)$ be either a strongly measurable random dynamical system or a $\mu$-continuous random dynamical system. Then the function $\omega \mapsto \diam_\theta(L(k,\omega)\cC)$ is measurable into $\R_{\geq 1}\cup\{\infty\}$ for all $k\geq 1$.
\end{lem}

\begin{proof}
Fix $k\geq 1$. First assume that we are in the strongly measurable case. Observe that $\cC$ is separable, with some countable dense subset $\{x_n\}_{n=1}^{\infty}$. For any $M > 1$, we have that \[ \set{\omega\in \Omega}{ \diam_\theta(L(k,\omega)\cC) \leq M } = \bigcap_{i,j\geq 1} \set{\omega\in\Omega}{\theta(L(k,\omega)x_i,L(k,\omega)x_j) \leq M}, \] because $L(k,\omega)$ is a continuous map on $X$ and $\theta$ is lower-semi-continuous on $\cC\times\cC$. Since the map $\omega \mapsto L(k,\omega)$ is strongly measurable, we see that $\omega \mapsto \theta(L(k,\omega)x_i,L(k,\omega)x_j)$ is measurable for each $i,j$. Thus $\omega \mapsto \diam_\theta(L(k,\omega)\cC)$ is measurable.

Next, assume that we are in the $\mu$-continuous case. Find a sequence of disjoint compact sets $K_n \subset \Omega$ on which $L(k,\omega)$ is continuous. Consider the set of operators $\cP \subset \cB(X)$ which preserve $\cC$, equipped with the subspace norm topology, and define $D: \cP \to [0,\infty]$ by $D(L) = \diam_\theta(L\cC)$. We claim that $D$ is lower-semi-continuous. To see this, suppose that $L_n \conv{n\to\infty} L$ in $\cP$, let $M = \liminf_{n\to\infty} D(L_n)$, and find a subsequence $L_{n_k}$ such that $M = \lim_{k\to\infty} D(L_{n_k})$. By the lower-semi-continuity of $\theta$, for any $x,y\in \cC$ we have:
\begin{align*}
\theta(Lx,Ly) & \leq \liminf_{n\to\infty} \theta(L_n x, L_n y) \\
& \leq \liminf_{k\to\infty} \theta(L_{n_k} x, L_{n_k} y) \\
& \leq \liminf_{k\to\infty} D(L_{n_k}) = M.
\end{align*}
Taking a supremum over all $x$ and $y$ yields lower-semi-continuity of $D$. Then $\diam_\theta(L(k,\omega)\cC) = D(L(k,\omega))$ is the composition of a continuous function and a lower-semi-continuous function, which is lower-semi-continuous and thus measurable on each compact $K_n$, therefore measurable on $\Omega$.
\end{proof}

\begin{prop}
\label{prop:exp-bound}
Let $(\Omega,\cB,\mu,\sigma,X,\norm{\cdot},L)$ be either a strongly measurable random dynamical system or a $\mu$-continuous random dynamical system, such that $\log^+\norm{L(1,\cdot)} \in L^{1}(\mu)$. Let $\cC \subset X$ be a nice cone such that $L(1,\omega)\cC \subset \cC$ for all $\omega$. Suppose that there exists a positive measure subset $G_P$ of $\Omega$, a positive integer $k_P$, and a positive real number $D_P$ such that for all $\omega \in G_P$, \[ \diam_\theta\big(L(k_P,\omega)\cC\big) \leq D_P. \] Then there exists a $\sigma$-invariant set of full measure $\tilde{\Omega} \subset \Omega$ and measurable functions $j^{\pm} : \tilde{\Omega}\times\Z_{\geq 0} \to \Z_{\geq 0}$ that are non-decreasing and tend to infinity in $n$ for all fixed $\omega$, such that for any $\omega\in \tilde{\Omega}$, and $n \geq k_P + \min\set{n\geq 0}{j^{\pm}(\omega,n) \geq 1}$, we have: 
\begin{gather*}
\diam_\theta(L(n,\omega)\cC) \leq \tanh\left( \frac{D_P}{4} \right)^{j^+(\omega,n)-1}\cdot D_P, \\
\diam_\theta(L(n,\sigma^{-n}(\omega))\cC) \leq \tanh\left( \frac{D_P}{4} \right)^{j^-(\omega,n)-1}\cdot D_P.
\end{gather*}
\end{prop}

\begin{proof}
By Poincar\'e Recurrence applied to both $\sigma$ and $\sigma^{-1}$, $\mu$-almost every point in $G_P$ returns infinitely often to $G_P$ both forward and backward in time; call the set of these points $G$, and let $\tilde{\Omega} = \bigcup_{n=-\infty}^{\infty} \sigma^{-n}(G)$. We have that $\tilde{\Omega}$ is $\sigma$-invariant, with measure $1$ by ergodicity of $(\mu,\sigma)$.

If $\omega \in \tilde{\Omega}$, then $\omega$ is in the orbit of a point in $G$, which means that $\sigma^{n}(\omega)$ enters $G_P$ infinitely often in both the forward and backward directions. Let $\{n_l^+(\omega)\}_{l\geq 1}$ be the sequence of non-negative indices such that $\sigma^{n_l^+(\omega)}(\omega) \in G_P$, and similarly we let $\{n_l^-(\omega)\}_{l\geq 1}$ be the sequence of positive indices such that $\sigma^{-n_l^-(\omega)}(\omega) \in G_P$. For notational purposes, we set
\begin{gather*} 
l^+(\omega,n) = \abs{\set{n_l^+(\omega)}{n_l^+(\omega) \leq n}}, \\
l^-(\omega,n) = \abs{\set{n_l^-(\omega)}{n_l^-(\omega)\leq n}}
\end{gather*} 
to denote the numbers of these indices.

When $\sigma^n(\omega) \in G_P$, we know that $L(k_P,\sigma^n(\omega))$ contracts the cone $\cC$ by a minimum amount, by the assumption on $G_P,k_P,D_P$. We want to count the number of times this event happens, going both forwards and backwards. We therefore define two new sequences $\{m_j^+(\omega)\}_{j\geq 1}$ and $\{m_j^-(\omega)\}_{j\geq 1}$ by taking subsequences of $n_l^+$ and $n_l^-$ where consecutive terms are at least $k_P$ apart. Specifically, we set
\begin{gather*}
m_1^+(\omega) = n_1^+(\omega), \\
m_j^+(\omega) = \min\set{n_l^+(\omega)}{n_l^+(\omega)\geq m_{j-1}^+(\omega)+k_P}\ (j>1); \\
m_1^-(\omega) = \min\set{n_l^-(\omega)}{n_l^-(\omega)\geq k_P}, \\
m_j^-(\omega) = \min\set{n_l^-(\omega)}{n_l^-(\omega) \geq m_{j-1}^-(\omega)+k_P}\ (j>1).
\end{gather*}
In the forward direction, we count the number of cone contractions in $n$ steps by counting the number of terms of $m_j^+$ that are at most $n-k_P$, to allow for the $k_P$ steps afterwards. In the backward direction, we count the number of cone contractions in $n$ steps by simply counting the number of $m_k^-$ terms that are at most $n$, because we have already accounted for the $k_P$ steps in the definition of $m_1^-$. In notation:
\begin{gather*}
j^+(\omega,n) = \abs{\set{m_j^+(\omega)}{m_j^+(\omega)\leq n-k_P}}, \\
j^-(\omega,n) = \abs{\set{m_j^-(\omega)}{m_j^+(\omega)\leq n}}.
\end{gather*}
It is straightforward to see that for fixed $\omega\in \tilde{\Omega}$, $j^{\pm}(\omega,n)$ is non-decreasing in $n$ and tends to infinity as $n$ grows arbitrarily large.

By the definition of $m_j^+(\omega)$, we have that \[ \diam\left(L(k_P,\sigma^{m_j^+(\omega)}(\omega))\cC\right) \leq D_P. \] We then let $C_j = L(k_P,\sigma^{m_j^+(\omega)}(\omega))$. Then, using the cocycle property we can write, for $n \geq m_1^++k_P$, \[ L(n,\omega) = B_{j^+(\omega,n)+1} C_{j^+(\omega,n)} B_{j^+(\omega,n)}\cdots C_2 B_2 C_1 B_1, \] where the $B$ terms do not necessarily contract distances in $\cC$ but still preserve them (using Lemma \ref{lem:proj-metric}). We now repeatedly utilize Lemma \ref{lem:proj-metric}, to see that for any $v,w\in \cC$:
\begin{align*}
\theta(L(n,\omega)v, L(n,\omega)w) & = \theta\left( B_{j^+(\omega,n)+1} C_{j^+(\omega,n)} B_{j^+(\omega,n)}\cdots C_2 B_2 C_1 B_1 v, \right. \\
& \hspace{30pt} \left. B_{j^+(\omega,n)+1} C_{j^+(\omega,n)} B_{j^+(\omega,n)}\cdots C_2 B_2 C_1 B_1 w \right) \\
& \leq \tanh\left( \frac{D_P}{4} \right)^{j^+(\omega,n)-1} \theta\left( C_1 B_1v, C_1 B_1w \right) \\
& \leq \tanh\left( \frac{D_P}{4} \right)^{j^+(\omega,n)-1} \cdot D_P.
\end{align*}
The multiple powers of $\tanh\left( \frac{D_P}{4} \right)$ arise from each block $C_i$ in the product, since that block contracts the $\cC$ to $\theta$-diameter at most $D_P$. Taking a supremum over all $v,w\in \cC$ yields the statement for the forward direction. The proof for the backward direction is completely analogous, instead using the backward direction indices $m_j^-$ and the counting function $j^-$.
\end{proof}

\begin{lem}
\label{lem:j-estimate}
For $\omega\in\tilde{\Omega}$, we have \[ j^+(\omega,n)+1 \geq \frac{l^+(\omega,n)}{k_P}, \quad j^-(\omega,n)+1 \geq \frac{l^-(\omega,n)+1}{k_P}. \]
\end{lem}

\begin{proof}
By the definition of $m_j^+(\omega)$, we have: \[ \set{n_l^+(\omega)}{n_l^+(\omega)\leq n} \subset \bigcup_{j=1}^{j^+(\omega,n)+1} \{ m_j^+(\omega),m_j^+(\omega)+1,\cdots,m_j^+(\omega)+k_P-1 \}. \] The first inequality follows by taking cardinalities.

By the definition of $m_j^-(\omega)$, we have \[ \set{n_i^-(\omega)}{ k_P \leq n_i^- \leq n} \subset \bigcup_{j=1}^{j^-(\omega,n)} \{ m_j^-(\omega),m_j^-(\omega)+1,\cdots,m_j^-(\omega)+k_P-1 \}, \] and $l^-(\omega,n)-k_P+1$ is at most the cardinality of the smaller set. The second inequality follows by taking inequalities and rearranging.
\end{proof}

By Birkhoff's Theorem, we know that $n^{-1}l^+(\omega,n)$ and $n^{-1}l^+(\omega,n)$ both converge $\mu$-almost everywhere to $\mu(G) = \mu(G_P)$. This fact provides the exponential rate of contraction of the cone by $L(n,\omega)$ in $n$, in both directions.

We now give the proof of the theorem. Parts \ref{thm:pf-v}, \ref{thm:pf-eta}, and \ref{thm:pf-le} will be proven in order, and statements in part \ref{thm:pf-meas} will be proven throughout as appropriate.

\begin{proof}[Proof of Theorem \ref{thm:pf}(\ref{thm:pf-v}).] 
We construct $v(\omega) \in X$ by constructing a Cauchy sequence and proving it converges to something with the correct properties; we use similar ideas to those found in \cite{arnold-random-pf-matrices,flq-2,gt-quas-semi-inv}. Given $\omega \in \tilde{\Omega}$, choose some $g\in \cC$ and define $v_n(\omega) \in X$ for $n\geq 0$ by \[ v_n(\omega) = \frac{L(n,\sigma^{-n}(\omega))g}{\norm{L(n,\sigma^{-n}(\omega))g}}. \] Each $v_n(\omega)$ has unit norm. By Proposition \ref{prop:exp-bound} and the scale-invariance of $\theta$, we see that for $n\geq m$:
\begin{align*} 
\theta(v_m(\omega),v_n(\omega)) & = \theta(L(m,\sigma^{-m}(\omega))g,L(n,\sigma^{-n}(\omega))g) \\
& = \theta(L(m,\sigma^{-m}(\omega))g,L(m,\sigma^{-m}(\omega))L(n-m,\sigma^{-(n-m)}(\omega))g) \\
& \leq \tanh\left( \frac{D_P}{4} \right)^{j^-(\omega,m)-1}\cdot D_P.
\end{align*}
Since $j^-$ tends to infinity in $n$ and $\theta$ is symmetric, we see that $\{v_n(\omega)\}_n$ is a Cauchy sequence in $\theta$, and hence Cauchy in norm by Proposition \ref{prop:D-adapt-Cauchy}. Let $v(\omega)$ be the limit of this sequence. Then $\norm{v(\omega)} = 1$ and $v(\omega) \in \cC$. Moreover, observe that $L(n,\omega)\cC$ is a decreasing chain of sets with $\theta$-diameter decreasing to zero. Since $\theta$ is lower-semi-continuous, we see that \[ \diam_\theta\left( \overline{L(n,\sigma^{-n}(\omega))\cC}^{\norm{\cdot}} \right) = \diam_\theta(L(n,\sigma^{-n}(\omega))\cC). \] Thus we see that $\bigcap_{n=0}^{\infty} \overline{L(n,\sigma^{-n}(\omega))\cC}^{\norm{\cdot}}$ is a norm-closed set with $\theta$-diameter zero; this set contains $v(\omega)$, as $v_n(\omega) \in L(n,\omega)\cC$ for each $n$, and so it is the positive ray containing $v(\omega)$. If we had used a different initial vector $g'$ and obtained the vector $v'(\omega)$, we would have $v'(\omega)$ collinear with $v(\omega)$ and having the same norm, which shows they are equal. Hence $v(\omega)$ is independent of the initial choice of $g$.

Suppose that $(\Omega,\cB,\mu,\sigma,X,\norm{\cdot},L)$ is strongly measurable. As per Appendix A in \cite{gt-quas-semi-inv}, we see that each $L(n,\sigma^{-n}(\omega))$ is strongly measurable, which implies that $L(n,\sigma^{-n}(\omega))g$ is measurable with respect to the norm on $X$. By measurability of the norm, we see that each $v_n(\omega)$ is measurable, and so the limit function is also measurable, as $(X,\norm{\cdot})$ is a metric space. 

Now, suppose that $(\Omega,\cB,\mu,\sigma,X,\norm{\cdot},L)$ is $\mu$-continuous. By continuity of the Banach algebra multiplication on $\cB(X)$ and the operator norm, each of the functions $L(n,\sigma^{-n}(\omega))$, $L(n,\sigma^{-n}(\omega))g$, and $\norm{L(n,\sigma^{-n}(\omega))}$ are also $\mu$-continuous. In addition, $\norm{L(n,\sigma^{-n}(\omega))}$ is bounded away from zero on compact sets where it is continuous, so that each $v_n(\omega)$ is $\mu$-continuous, and thus by Lemma \ref{lem:limit-mu-cts}, $v(\omega)$ is $\mu$-continuous.

To see that $v(\omega)$ is equivariant, we have (since $L(1,\omega)$ is continuous):
\begin{align*}
L(1,\omega)v(\omega) & \in L(1,\omega) \left( \bigcap_{n=0} \overline{L(n,\sigma^{-n}(\omega))\cC}^{\norm{\cdot}} \right) \\
& \subset \bigcap_{n=0} L(1,\omega)\left( \overline{L(n,\sigma^{-n}(\omega))\cC}^{\norm{\cdot}} \right) \\
& \subset \bigcap_{n=0} \overline{L(n+1,\sigma^{-{n+1}}(\sigma(\omega)))\cC}^{\norm{\cdot}}
\end{align*}
This last set has $\theta$-diameter equal to $0$ and contains $v(\sigma(\omega))$, as shown above. Thus we see that \[ L(1,\omega)v(\omega) = \norm{L(1,\omega)v(\omega)}v(\sigma(\omega)); \] set $\phi(\omega) = \norm{L(1,\omega)v(\omega)}$, so that $L(1,\omega)v(\omega) = \phi(\omega)v(\omega)$. It is clear that $\phi(\omega) \leq \norm{L(1,\omega)}$ and that $\phi(\omega)$ is measurable (in either set of hypotheses), so $\log^+\phi \in L^1(\mu)$. 
\end{proof}

\begin{proof}[Proof of Theorem \ref{thm:pf}(\ref{thm:pf-eta})]
For $\omega\in \tilde{\Omega}$, let $\tilde{L}(1,\omega) = \phi(\omega)^{-1}L(1,\omega)$. Then $\tilde{L}(1,\omega)$ preserves $\cC$ and satisfies $\tilde{L}(1,\omega)v(\omega) = v(\sigma(\omega))$. We will construct, for each $\omega$, a linear functional on $X$. So fix $\omega$, and for any $g\in \cC$ and $n\geq 0$, we apply Lemma \ref{lem:proj-metric} to see that
\begin{align*} 
\alpha(v(\sigma^n(\omega)),\tilde{L}(n,\omega)g) & \leq \alpha(v(\sigma^{n+1}(\omega)),\tilde{L}(n+1,\omega)g) \\
& \leq \beta(v(\sigma^{n+1}(\omega)),\tilde{L}(n+1,\omega)g) \leq \beta(v(\sigma^n(\omega)),\tilde{L}(n,\omega)g).
\end{align*} 
By Proposition \ref{prop:exp-bound}, we know that there is some $N = N(\omega)$ such that the $\theta$-diameter of $L(N,\omega)\cC$ is finite, which implies that \[ 0 < \alpha(v(\sigma^N(\omega)),\tilde{L}(N,\omega)g) \leq \beta(v(\sigma^N(\omega)),\tilde{L}(N,\omega)g) < \infty, \] and so the two sequences are monotonic and bounded, thus convergent. Moreover, for $n\geq N$ we have \[ 1 \leq \frac{\beta(v(\sigma^n(\omega)),\tilde{L}(n,\omega)g)}{\alpha(v(\sigma^n(\omega)),\tilde{L}(n,\omega)g)} = e^{\theta(\tilde{L}(n,\omega)v(\omega),\tilde{L}(n,\omega)g)} = e^{\diam_\theta(L(n,\omega)\cC)}, \] and Proposition \ref{prop:exp-bound} shows that the right side of the equation converges to $1$. Let $\eta(\omega,g)$ be the shared limit of $\alpha(v(\sigma^n(\omega)),\tilde{L}(n,\omega)g)$ and $\beta(v(\sigma^n(\omega)),\tilde{L}(n,\omega)g)$.

We now show that $\eta(\omega,\cdot)$ extends to a bounded linear functional on $X$. By Lemma \ref{lem:proj-metric} and linearity of $\tilde{L}(n,\omega)$, we see that $\alpha(v(\sigma^n(\omega)),\tilde{L}(n,\omega)g)$ is positive-scalar-homogeneous in $g$, so that $\eta(\omega,g)$ is also, and $\eta(\omega,g)$ is positive because it is larger than $\alpha(v(\sigma^n(\omega)),\tilde{L}(n,\omega)g)$. We also see that $\eta(\omega,\cdot)$ is additive on $\cC$, by using super-additivity of $\alpha$ and sub-additivity of $\beta$ and taking limits. We then appeal to Lemma \ref{lem:ext-pos} to extend $\eta(\omega,\cdot)$ uniquely to a linear functional on $X$.

To see that $\eta(\omega,\cdot)$ is bounded, let $g\in \cC$, and let $N$ be such that $\diam_\theta(L(N,\omega)\cC)$ is finite. We have:
\begin{align*}
\eta(\omega,g) & \leq \beta(v(\sigma^N(\omega)),\tilde{L}(N,\omega)g) \\
& = \beta\left(v(\sigma^N(\omega)),\frac{\tilde{L}(N,\omega)g}{\norm{\tilde{L}(N,\omega)g}}\right) \cdot \norm{\tilde{L}(N,\omega)g} \\
& \leq \alpha\left(v(\sigma^N(\omega)),\frac{\tilde{L}(N,\omega)g}{\norm{\tilde{L}(N,\omega)g}}\right) e^{\diam_\theta(\tilde{L}(N,\omega)\cC)} \norm{\tilde{L}(N,\omega)}\norm{g} \\
& \leq D e^{\diam_\theta(\tilde{L}(N,\omega)\cC)} \norm{\tilde{L}(N,\omega)}\norm{g}.
\end{align*}
We now use the fact that $\cC$ generates $X$. If $x \in X$, by And\^o's Lemma we may write $x = g_1 - g_2$, such that $\norm{g_i} \leq K\norm{x}$. Then we have:
\begin{align*}
\abs{\eta(\omega,x)} & \leq \eta(\omega,g_1) + \eta(\omega,g_2) \\
& \leq D e^{\diam_\theta(\tilde{L}(N,\omega)\cC)} \norm{\tilde{L}(N,\omega)}(\norm{g_1} + \norm{g_2}) \\
& \leq 2KD e^{\diam_\theta(\tilde{L}(N,\omega)\cC)} \norm{\tilde{L}(N,\omega)}\norm{x}.
\end{align*}
Thus $\eta(\omega,\cdot)$ is bounded.

Directly by the definition, we have \[ \eta(\omega,v(\omega)) = \lim_{n\to\infty} \alpha(v(\sigma^n(\omega)),\tilde{L}(n,\omega)v(\omega)) = \lim_{n\to\infty} \alpha(v(\sigma^n(\omega)),v(\sigma^n(\omega))) = 1. \] For $g\in \cC$, we have
\begin{align*} 
\eta(\sigma(\omega),L(1,\omega)g) & = \phi(\omega)\lim_{n\to\infty} \alpha(v(\sigma^n(\sigma(\omega))),\tilde{L}(n,\sigma(\omega))\tilde{L}(1,\omega)g) \\
& = \phi(\omega)\lim_{n\to\infty} \alpha(v(\sigma^{n+1}(\omega)),\tilde{L}(n+1,\omega)g) \\
& = \phi(\omega)\eta(\omega,g),
\end{align*} 
and by linearity this equality extends to $g\in X$. We have already seen that $\eta(\omega,\cdot)$ is strictly positive on $\cC$. 

Equivariance of the decomposition $X = \subspan_{\R}\{v(\omega)\} \oplus \ker(\eta(\omega,\cdot))$ follows from the equivariance properties of $v(\omega)$ and $\eta(\omega,\cdot)$. It is clear to see that $\pi_\omega(\cdot) = \eta(\omega,\cdot)v(\omega)$ is the projection with range $\subspan_{\R}\{v(\omega)\}$ and kernel $\ker(\eta(\omega,\cdot))$, and it is continuous because $\eta(\omega,\cdot)$ is.
\end{proof}

We now prove two technical lemmas that allow us to prove the two remaining parts of the theorem.

\begin{lem}
\label{lem:order-bound}
In the setting of Theorem \ref{thm:pf}, if $\omega\in \tilde{\Omega}$ then there exists $N = N(\omega)$ such that for all $g\in \cC$ and $n\geq N$, \[ \tilde{L}(n,\omega)g \preceq \beta(v(\sigma^N(\omega)),\tilde{L}(N,\omega)g)v(\sigma^n(\omega)). \]
\end{lem}

\begin{proof}
Find $N(\omega)$ such that $\diam_\theta(L(N,\omega)\cC)$ is finite. Then $\beta(v(\sigma^N(\omega)),\tilde{L}(N,\omega)g)$ is finite, and \[ \tilde{L}(N,\omega)g \preceq \beta(v(\sigma^N(\omega)),\tilde{L}(N,\omega)g)v(\sigma^N(\omega)). \] Apply $\tilde{L}(n-N,\sigma^N(\omega))$ to both sides to obtain the conclusion of the lemma.
\end{proof}

\begin{lem}
\label{lem:norm-eta-conv}
In the setting of Theorem \ref{thm:pf}, if $\omega\in \tilde{\Omega}$, then there exists $N = N(\omega)$ and $C_{\omega} > 0$ such that for all $g\in \cC$ and $n\geq N$, $\diam_\theta(L(n,\omega)\cC)$ is finite and \[ \norm{\tilde{L}(n,\omega)g-\eta(\omega,g)v(\sigma^n(\omega))} \leq C_{\omega} \norm{g}\left( e^{\diam_\theta(L(n,\omega)\cC)}-1 \right). \] The constant is $C_{\omega} = \frac{(D+1)D^2}{2}\ e^{\diam_\theta(L(N,\omega)\cC)}\norm{\tilde{L}(N,\omega)}.$
\end{lem}

\begin{proof}
Let $\omega\in \tilde{\Omega}$, and let $g \in\cC$. Let $a(n,\omega) = \alpha(v(\sigma^n(\omega)),\tilde{L}(n,\omega)g)$, $b(n,\omega) = \beta(v(\sigma^n(\omega)),\tilde{L}(n,\omega)g)$. Since $a(n,\omega) \leq \eta(\omega,g) \leq b(n,\omega)$ for all $n$, the distance from $\eta(\omega,g)$ to the midpoint
of $[a(n,\omega), b(n,\omega)]$ is at most half the length of the interval. In addition, by Lemma \ref{lem:proj-metric}, we have $a(n,\omega) = \leq D\norm{\tilde{L}(n,\omega)g}$. By the first part of Proposition \ref{prop:D-adapt-Cauchy}, we have:
\begin{align*}
\lVert \tilde{L}&(n,\omega)g-\eta(\omega,g)v(\sigma^n(\omega)) \rVert \\
& \leq \norm{\tilde{L}(n,\omega)g- \left(\frac{1}{2}(a(n,\omega)+b(n,\omega))\right)v(\sigma^n(\omega))} \\
& \hspace{30pt} + \norm{\left(\frac{1}{2}(a(n,\omega)+b(n,\omega))\right)v(\sigma^n(\omega))-\eta(\omega,g)v(\sigma^n(\omega))} \\
& \leq \frac{D}{2}\ a(n,\omega)\left( e^{\theta(v(\sigma^n(\omega)),\tilde{L}(n,\omega)g)}-1 \right) + \abs{\frac{a(n,\omega)+b(n,\omega)}{2}-\eta(\omega,g)} \\
& \leq \frac{D+1}{2} a(n,\omega)\left( e^{\theta(v(\sigma^n(\omega)),\tilde{L}(n,\omega)g)}-1 \right) \\
& \leq \frac{(D+1)D}{2}\norm{\tilde{L}(n,\omega)g}\left( e^{\diam_\theta(L(n,\omega)\cC)}-1 \right).
\end{align*}
By Lemma \ref{lem:order-bound} and the $D$-adapted condition, we have that $\norm{\tilde{L}(n,\omega)g} \leq D\cdot \beta(v(\sigma^N(\omega)),\tilde{L}(N,\omega)g)$. By definition, $\beta = \alpha e^{\theta}$ on $\cC\times\cC$, and so we get 
\begin{align*} 
\norm{\tilde{L}(n,\omega)g} & \leq \alpha(v(\sigma^N(\omega)),\tilde{L}(N,\omega))e^{\diam_\theta(L(N,\omega)\cC)} \\
& \leq De^{\diam_\theta(L(N,\omega)\cC)}\norm{\tilde{L}(N,\omega)}\norm{g}.
\end{align*}
The proof is complete upon substitution into the above inequality.
\end{proof}

We now prove part \ref{thm:pf-meas} of the main theorem. The primary difficulty lies in the $\mu$-continuous case, because the definition of $\eta(\omega,\cdot)$ is in terms of $\alpha$ and $\beta$ terms, and these two functions are not continuous on the entirety of $\cC\times\cC$, only on $\mathrm{int}(\cC)\times\cC$. Instead, we prove that $\eta(\omega,\cdot)$ is a well-behaved limit of a much nicer function when restricted to $\cC$. The strongly measurable case is much simpler.

\begin{proof}[Proof of Theorem \ref{thm:pf}(\ref{thm:pf-meas})]
First, assume that $(\Omega,\cB,\mu,\sigma,X,\norm{\cdot},L)$ is strongly measurable. We have already seen that $v(\omega)$ is measurable into $(X,\norm{\cdot})$. To see that $\eta(\omega,\cdot)$ is strongly measurable (measurable with respect to the weak\textsuperscript{*} $\sigma$-algebra), let $g\in \cC$. Then the map $\omega \mapsto \alpha(v(\sigma^n(\omega)),\tilde{L}(n,\omega)g)$ is measurable into $(\R,\abs{\cdot})$, by Lemma \ref{lem:proj-metric} and the strong measurability of $L$. Then $\omega \mapsto \eta(\omega,g)$ is the pointwise limit of these functions, taking values in a metric space, and hence measurable. If $x\in X$, then write $\eta(\omega,x) = \eta(\omega,g_1) - \eta(\omega,g_2)$ for some $g_1,g_2 \in \cC$, and observe that $\eta(\omega,x)$ is the difference of two measurable functions, hence measurable. Thus $\eta(\omega,\cdot)$ is strongly measurable. The operators $\pi_\omega$ and $I-\pi_\omega$ are strongly measurable (that is, measurable with respect to the strong operator $\sigma$-algebra) because $\eta(\omega,\cdot)$ is strongly measurable. The map from $S(0,1) \subset X$ to the Grassmannian $\cG(X)$ given by $v\mapsto \subspan_{\R}\{v\}$ is continuous, so that $\omega \mapsto \subspan_{\R}\{v(\omega)\}$ is measurable.

Next, assume that $(\Omega,\cB,\mu,\sigma,X,\norm{\cdot},L)$ is $\mu$-continuous. By Lemma \ref{lem:limit-mu-cts}, we see that $v(\omega)$ is $\mu$-continuous, since each $v_n(\omega)$ is $\mu$-continuous. As just mentioned, taking the span of $v(\omega)$ is continuous, so that $\subspan_{\R}\{v(\omega)\}$ is $\mu$-continuous as well.

To see that $\omega \mapsto \eta(\omega,\cdot)$ is $\mu$-continuous, we use the machinery developed to relate functions on the cone to linear functionals. Since $L(1,\omega)$ is $\mu$-continuous, for fixed $n$ we see that $L(n,\omega)$ is continuous on compact subsets of $\Omega$ with arbitrarily large measure. Applying operators to norm-bounded vectors is a norm-continuous operations on $\cB(X)$, and so because $L(n,\omega)v(\omega)$ is never zero, we see that \[ \omega \mapsto \tilde{L}(n,\omega)= \frac{L(n,\omega)}{\norm{L(n,\omega)v(\omega)}} \] is a $\mu$-continuous map into $(\cB(X),\norm{\cdot})$. Suppose that $\tilde{L}(n,\omega)$ is continuous on some large compact $K\subset \tilde{\Omega}$. For $\omega_1,\omega_2 \in K$ and $g\in \cC$ with $\norm{g} \leq 1$, we then have:
\begin{align*}
\abs{\ \norm{\tilde{L}(n,\omega_1)g} - \norm{\tilde{L}(n,\omega_2)g}\ } & \leq \norm{\tilde{L}(n,\omega_1)g-\tilde{L}(n,\omega_2)g} \\
& \leq \norm{\tilde{L}(n,\omega_1)-\tilde{L}(n,\omega_2)},
\end{align*}
which can be made as small as desired by taking $\omega_2$ arbitrarily close to $\omega_1$. Thus $\omega \mapsto \norm{\tilde{L}(n,\omega)\cdot}$ is a $\mu$-continuous map into $(\mathrm{BPSH}(\R^{\cC}),\norm{\cdot}_{\R^{\cC}})$.

Then, by the reverse triangle inequality and Lemma \ref{lem:norm-eta-conv}, we see that for $\omega \in \tilde{\Omega}$, there exists $N = N(\omega)$ such that for all $g\in \cC$ with $\norm{g} \leq 1$ and $n\geq N$, we have:
\begin{align*}
\abs{\ \norm{\tilde{L}(n,\omega)g}-\eta(\omega,g)} & \leq \norm{\tilde{L}(n,\omega)g-\eta(\omega,g)v(\sigma^n(\omega))} \\
& \leq C_{\omega} \norm{g}\left( e^{\diam_\theta(L(n,\omega)\cC)}-1 \right),
\end{align*}
where $C_{\omega} = \frac{1}{2}(D+1)D^2 e^{\diam_\theta(L(N,\omega)\cC)} \norm{\tilde{L}(N,\omega)}$. As $n$ tends to infinity, we have that $\left( e^{\diam_\theta(L(n,\omega)\cC)}-1 \right)$ tends to $0$ by Proposition \ref{prop:exp-bound}. Therefore $\norm{\tilde{L}(n,\omega)\cdot}$ converges pointwise in $\omega$ to $\eta(\omega,\cdot)$ in $(\mathrm{BPSH}(\R^{\cC}),\norm{\cdot}_{\R^{\cC}})$. By Lemma \ref{lem:limit-mu-cts}, $\omega \mapsto \eta(\omega,\cdot)$ is $\mu$-continuous into the space $(\mathrm{BPSH}(\R^{\cC}),\norm{\cdot}_{\R^{\cC}})$; since the topology on $(X^*,\norm{\cdot}_{C^*})$ is the restriction of the topology from $(\mathrm{BPSH}(\R^{\cC}),\norm{\cdot}_{\R^{\cC}})$ by Lemma \ref{lem:bpsh-dual}, we see that $\omega \mapsto \eta(\omega,\cdot)$ is $\mu$-continuous into $X^*$.

The operators $\pi_\omega$ and $I-\pi_\omega$ are $\mu$-continuous because $\eta(\omega,\cdot)$ is. By Proposition B.3.2 in \cite{thieullen-mu-cts}, the subspaces $\subspan_{\R}\{v(\omega)\} = \ker(I-\pi_\omega)$ and $\ker(\eta(\omega,\cdot)) = \ker(\pi_\omega)$ are both $\mu$-continuous with respect to the Grassmannian distance, since the kernel map is norm-continuous on projections.
\end{proof}

\begin{proof}[Proof of Theorem \ref{thm:pf}(\ref{thm:pf-le})]
First, we prove that for $\omega \in \tilde{\Omega}$, $v(\omega)$ has the largest Lyapunov exponent of any vector in $X$ for $L(n,\omega)$. Let $x\in X\setminus\{0\}$, and by And\^o's Lemma (\ref{lem:ando}), find $g_1,g_2\in \cC$ such that $x = g_1 - g_2$ and $\norm{g_i} \leq K\norm{x}$. By Lemma \ref{lem:order-bound} and multiplying by $\norm{L(n,\omega)v(\omega)}$, there exists $N = N(\omega)$ such that for all $n\geq N$ and $i=1,2$,
\begin{align*} 
L(n,\omega)g_i & \preceq \beta(v(\sigma^N(\omega)),\tilde{L}(N,\omega)g_i)\norm{L(n,\omega)v(\omega)}v(\sigma^n(\omega)) \\
& = \beta(v(\sigma^N(\omega)),\tilde{L}(N,\omega)g_i)L(n,\omega)v(\omega).
\end{align*} 
By the $D$-adapted condition, we see that \[ \norm{L(n,\omega)g_i} \leq D\cdot \beta(v(\sigma^N(\omega)),\tilde{L}(N,\omega)g_i)\norm{L(n,\omega)v(\omega)}. \] We may then bound above the Lyapunov exponent for $x$ over $\omega$:
\begin{align*}
\limsup_{n\to\infty} & \frac{1}{n}\log\norm{L(n,\omega)x} \leq \limsup_{n\to\infty} \frac{1}{n}\log\left( \norm{L(n,\omega)g_1} + \norm{L(n,\omega)g_2} \right) \\
& \leq \limsup_{n\to\infty} \left( \frac{1}{n}\log\left( \beta(v(\sigma^N(\omega)),\tilde{L}(N,\omega)g_1) + \beta(v(\sigma^N(\omega)),\tilde{L}(N,\omega)g_2) \right) \right. \\
& \hspace{50pt} \left. + \frac{1}{n}\log\norm{L(n,\omega)v(\omega)} \right) \\
& = \limsup_{n\to\infty} \frac{1}{n}\log\norm{L(n,\omega)v(\omega)}.
\end{align*}

We now show that $\tilde{L}(n,\omega)$ restricted to the kernel of $\eta(\omega,\cdot)$ has an exponential growth rate strictly less than $0$. By Lemma \ref{lem:norm-eta-conv}, for any $\omega\in\tilde{\Omega}$ there exists $N = N(\omega)$ and $C_{\omega}$ such that for all $n\geq N$ and $g\in\cC$, we have \[ \norm{\tilde{L}(n,\omega)g-\eta(\omega,g)v(\sigma^n(\omega))} \leq C_{\omega} \norm{g}\left( e^{\diam_\theta(L(n,\omega)\cC)}-1 \right). \] Since $X = \cC-\cC$, we apply And\^o's Lemma to get $K \geq 1$ such that for any $x\in X$, there exist $g_1,g_2\in \cC$ such that $x=g_1-g_2$ and $\norm{g_i} \leq K\norm{x}$. Suppose that $x\in\ker(\eta(\omega,\cdot))$. Then, by the triangle inequality we obtain:
\begin{align*} 
\norm{\tilde{L}(n,\omega)x} & = \norm{\tilde{L}(n,\omega)x-\eta(\omega,x)v(\sigma^n(\omega))} \\ 
& \leq 2KC_{\omega}\norm{x}\left( e^{\diam_\theta(L(n,\omega)\cC)}-1 \right).
\end{align*}
Only the diameter term depends on $n$. By Proposition \ref{prop:exp-bound} and Lemma \ref{lem:j-estimate} applied in order, we have (because $\tanh\left( \frac{D_P}{4} \right) \in (0,1)$):
\begin{align*}
\norm{\tilde{L}(n,\omega)x} & \leq 2KC_{\omega} \left( e^{\tanh\left( \frac{D_P}{4} \right)^{j^+(\omega,n)-1}\cdot D_P} - 1 \right) \\
& \leq 2KC_{\omega} \left( e^{\tanh\left( \frac{D_P}{4} \right)^{l^+(\omega,n)/k_P-2} \cdot D_P} - 1 \right).
\end{align*}
Note that $\log(e^x-1)$ is asymptotically equivalent to $\log(x)$ as $x$ tends to $0$ and that $n^{-1}l^+(\omega,n)$ tends to $\mu(G_P)$ for all $\omega\in \tilde{G}$, by Birkhoff's theorem. Taking logarithms, dividing by $n$, and taking a $\limsup$, we thus have:
\begin{align*}
\limsup_{n\to\infty} \frac{1}{n}\log\norm{\tilde{L}(n,\omega)x} & \leq \limsup_{n\to\infty} \frac{1}{n}\log\left( \frac{2KC_{\omega}D_P}{\tanh\left( \frac{D_P}{4} \right)^2} \right) \\
& \hspace{50pt} + \frac{1}{n}\log\left( \tanh\left( \frac{D_P}{4} \right)^{l^+(\omega,n)/k_P} \right) \\
& = \limsup_{n\to\infty} \frac{1}{k_P}\cdot \frac{l^+(\omega,n)}{n}\log \left( \tanh\left( \frac{D_P}{4} \right) \right) \\
& = -\frac{\mu(G_P)}{k_P}\log \left( \tanh\left( \frac{D_P}{4} \right)^{-1} \right),
\end{align*}
where we use an explicit negative sign to indicate the sign of the quantity. The inequality in the theorem statement follows by rewriting $\norm{\tilde{L}(n,\omega)}$.

Next, assume that $\lambda^* = \lim_{n\to\infty} \frac{1}{n}\log\norm{L(n,\omega)}$ is finite. Then by Proposition 14 in \cite{flq-2}, we see that the Lyapunov exponent $\lambda_1(\omega)$ for $v(\omega)$ is equal to $\lambda^*$ for all $\omega \in \tilde{\Omega}$. Moreover, Lemma \ref{lem:birk-sums-conv-L1} tells us that because the Birkhoff sums converge $\mu$-almost everywhere, $\log(\phi) \in L^{1}(\mu)$, and so $\lambda^* = \int_{\Omega} \log(\phi)\ d\mu$. As well, the bound on the Lyapunov exponent for $\tilde{L}(n,\omega)$ on $\ker(\eta(\omega,\cdot))$ and the equivariance of the decomposition show that the top exponent $\lambda_1$ has multiplicity one (corresponding to $\subspan_{\R}\{v(\omega)\}$).

Finally, we want to show that the projections are norm-tempered, that is, $\norm{\pi_{\omega}}$ and $\norm{I-\pi_{\omega}}$ are tempered functions. To do this, we observe that \[ 1 \leq \norm{\pi_\omega} = \norm{\eta(\omega,\cdot)} \leq 2KD\ e^{\diam_\theta(\tilde{L}(N,\omega)\cC)} \norm{\tilde{L}(N,\omega)}, \] where we may choose $N = N(\omega)$ to be the first time that $\sigma^n(\omega)$ enters $G_P$ (for $n\geq 1$). Then the exponential term is at most $e^{D_P}$ for every $\omega$, and it remains to deal with the $\norm{\tilde{L}(N(\omega),\omega)}$ term. We have:
\begin{align*}
0 & \leq \frac{1}{n}\log\left( \frac{\norm{L(N(\sigma^n(\omega)),\sigma^n(\omega))}}{\norm{L(N(\sigma^n(\omega)),\sigma^n(\omega))v(\sigma^n(\omega))}} \right) \\
& \leq \frac{1}{n} \sum_{i=0}^{N(\sigma^n(\omega))-1} \left( \log\norm{L(1,\sigma^{n+i}(\omega))}-\log(\phi(\sigma^{n+i}(\omega))) \right).
\end{align*}
By Lemma \ref{lem:weird-birk-sum}, this last sum converges to $0$ for $\mu$-almost every $\omega$, because the Birkhoff sums for the two individual sums converge to the same finite quantity and $N$ satisfies the hypothesis of the lemma, by Lemma \ref{lem:first-entry-conv}. Thus we have that $n^{-1}\log\norm{\tilde{L}(N(\sigma^n(\omega)),\sigma^n(\omega))}$ converges to $0$, so that the norm of $\pi_{\omega}$ is tempered. Clearly, $\norm{I-\pi_{\omega}} \leq 1 + \norm{\pi_{\omega}}$, so the norm of $I-\pi_{\omega}$ is also tempered. The proof of Theorem \ref{thm:pf} is complete. 
\end{proof}

\section{A Balanced Lasota-Yorke-type Inequality}
\label{sect:ly-ineq}

\subsection{Setting and Bounded Variation}
\label{subsect:bv-setting}

The Banach space setting for the application in Section \ref{sect:app-pf-ops} is $BV[-1,1] \subset L^{\infty}([-1,1],\lambda)$, where $\lambda$ is the normalized Lebesgue measure on $[-1,1]$. For the purposes of providing a more transferable result, we will state and prove the result for more general spaces than closed subintervals of the real line, along the same lines as Rychlik \cite{rychlik}. We will also use an idea from Eslami and G\'ora, who considered points they informally called ``hanging'' \cite{eslami-gora}. 

Let $(X,\leq) = ([a_X,b_X],\leq)$ be a totally ordered order-complete set equipped with its order topology, and equip $X$ with a complete regular Borel probability measure $\lambda$. Let $\{I_n\}_{n\in N}$ be a countable cover of $X$ by closed intervals $I_n = [a_n,b_n]$, with $a_n < b_n$; $N$ may be finite or countably infinite. Denote by $I_n^o$ the interval $(a_n,b_n)$. (In general, $I_n^o$ is not the topological interior, because there could exist isolated points in the order topology, but we will assume no isolated points exist.) Suppose that we have: 
\begin{itemize}
\item $I_n^o \cap I_m^o = \emptyset$ for all $n\ne m$;
\item $\bigcup_n I_n^o$ is dense and has measure $1$.
\end{itemize}
Let $T : X\to X$ be a map such that both of the following hold:
\begin{itemize}
\item $\restr{T}{I_n^o}$ is continuous and extends to a homeomorphism $T_n : I_n \to T_n(I_n) \subset X$;
\item There exists a bounded measurable function $g : X\to [0,\infty)$ such that $L(f)(x) := \sum_{y\in T^{-1}(x)} g(y)f(y)$ defines an operator that preserves $\lambda$ (that is, $\lambda(L(f)) = \lambda(f)$ for all $f\in L^1(X)$), $g$ has one-sided limits at the endpoints of each $I_n$, and $g$ is $0$ at the endpoints of each $I_n$.
\end{itemize}
We make two further assumptions:
\begin{itemize}
\item The intervals $I_n$ are as large as possible (so that the endpoints are places where $T$ is not continuous or not monotone);
\item There are no isolated points in $X$.
\end{itemize}
Denote the collection of assumptions about $X$ and $T$ by $(\cM)$. The intervals $I_n$ are sometimes called intervals of monotonicity. Let $K_n = T_n(I_n)$ be the image of the homeomorphism $T_n$. Setting $g$ to be equal to $0$ at endpoints of intervals of monotonicity simplifies the computations with $L$ while also not actually affecting any of the calculations to be done later; note that there is no requirement that $g$ is continuous at endpoints; in most cases the one-sided limits are non-zero.

The following lemma is mostly Rychlik's Remark 2 in \cite{rychlik}; it is necessary because the assumptions made on $L$ do not explicitly give these properties. Stating the situation in this manner also allows us to collect easily-referenced results, as well as do some basic computations regarding the operator $L$; in addition, we record the important fact that our assumptions rule out the case of $\lambda$ giving positive measure to a singleton.

\begin{lem}
\label{lem:L-props}
Let $X$ and $T$ satisfy the assumptions in $(\cM)$. The operator $L$ has the following properties:
\begin{enumerate}
\item $L(f\circ T\cdot h) = f\cdot L(h)$ for $f\in L^{\infty}(X)$ and $h\in L^{1}(X)$;
\item $L(\mathds{1}_{A}) = g\circ \bar{T}_n^{-1}\cdot \mathds{1}_{T_n(A)}$, for $n\in N$ and $A\subset I_n$ measurable.
\end{enumerate} Moreover, we have:
\begin{enumerate}
\setcounter{enumi}{2}
\item $T$ is non-singular with respect to $\lambda$;
\item $g$ is almost everywhere non-zero, $\mathds{1}_{I_n}\frac{1}{g} \in L^{1}(X)$ for each $n\in N$, and $L\left(\mathds{1}_{I_n}\frac{1}{g}\right) = \mathds{1}_{K_n}$ for each $n\in N$;
\item $\frac{1}{g}$ is the Jacobian of $T$, in the sense that for each $n$ and $f\in L^{\infty}(X)$, \[ \int_{K_n} f\ d\lambda = \int_{I_n} f\circ T\cdot \frac{1}{g}\ d\lambda; \]
\item the image of a measure zero set under $T$ is a measure zero set;
\item $\lambda\{x\} = 0$ for all $x\in X$.
\end{enumerate} Finally, $L$ is the Perron-Frobenius operator on $L^{1}(X)$ corresponding to $T$.
\end{lem}

To rigorously describe the notion of ``hanging'' points, we will use another definition.

\begin{defn}
\label{defn:one-tailed}
Let $(X,\leq)$ be a totally ordered set equipped with its order topology; for $x\in X$, denote $U_x = \set{y\in X}{y > x}$ and $L_x = \set{y\in X}{y < x}$. For a point $x\in X$ and a closed interval $I\subset X$, we say that $(x,+) \in I$ when $x\in I$ and $x$ is in the closure of $I\cap U_x$. Similarly, we say that $(x,-) \in I$ when $x\in I$ and $x$ is in the closure of $I\cap L_x$. For $x\in X$ and $x\in \{+,1\}$, $(x,s)$ is called a \emph{one-tailed point}.
\end{defn}

One may also say that $(x,+) \in I$ when $x\in I$ and $x$ is adherent to $I\cap U_x$ (to use different terminology), and similarly for $(x,-)$. The broad interpretation is that $(x,s)$ is in some closed interval $I$ when it is possible to approach $x$ from the $s$ direction from within $I$; this approaching phenomenon is like the point $x$ having a tail in $I$. Observe that for two closed intervals $[c,x]$ and $[x,d]$, $(x,s)$ can be in at most one of the two intervals, so each one-tailed point is in at most one interval in a partition of $X$ via closed intervals.

For the next definition, recall the notation $f(x^s) := \lim_{y \to x^s} f(t)$, for $s\in \{+,-\}$ and functions $f$ where the limit exists.

\begin{defn}
\label{defn:hanging}
Consider a space and map $X$ and $T$ as above. Let \[ E = \set{(x,s) \in X}{x\in \{a_n,b_n\}_n, s\in \{+,-\}} \] be the collection of one-tailed points of $X$ located at endpoints of intervals of monotonicity for $T$, and let $(x,s) \in E$. We say that $(x,s)$ is a \emph{hanging point} for $T$ when $T(x^s) \notin \{a_X,b_X\}$, or alternatively $\mathds{1}_{\{a_X,b_X\}}(T(x^s)) = 0$. Denote the set of all hanging points of $T$ by $H$.
\end{defn}

\begin{eg}
\label{eg:hanging-tent}
The paired tent map in Figure \ref{fig:tent-graph} has four hanging points: $(-0.5,-)$, $(-0.5,+)$, $(0.5,-)$, and $(0.5,+)$. Here, $X = [-1,1]$, and we have that \[ E = \left\{ (-1,+),(-0.5,-),(-0.5,+),(0,-),(0,+),(0.5,-),(0.5+),(1,-) \right\}. \] For the one-tailed points $(x,s) \in E$ where $x \in \{-1,0,1\}$, we have $T(x^s) \in \{-1,1\}$, and in the other four cases we have $T(x^s) \in \{-1,1\}$, again by simply looking at the picture. Note that $T$ is continuous at $\pm 0.5$ (where there are hanging points), whereas it is not continuous at $0$ (where there are no hanging points). Jumps or lack thereof do not affect whether or not a one-tailed point is hanging.
\end{eg}

In our given situation, recall that the variation of a function $f : X\to \R$ over a non-empty set $C\subset X$ is given by \[ \Var_C(f) = \sup\set{\sum_{i=1}^m \abs{f(x_i)-f(x_{i-1})}}{x_0 < x_1 < \dots < x_m, x_i\in C}. \] Write $\Var(f) = \Var_X(f)$. For a $\lambda$-equivalence class of functions $[f]$ and a set $C\subset X$ with positive measure, let \[ \overline{\Var}_C[f] = \inf\set{\Var_C(g)}{g = f\ \lambda-\text{a.e.}}. \] Let $BV(X,\lambda)$ be given by \[ BV(X,\lambda) = \set{[f]\in L^{\infty}(X,\lambda)}{\overline{\Var}[f] < \infty}. \] Then $BV(X,\lambda)$ is a Banach space, with the norm $\norm{\cdot}_{BV} = \norm{\cdot}_1 + \overline{\Var}(\cdot)$. (The $1$-norm could be replaced with either the essential supremum $\norm{\cdot}_{\infty}$ or the essential infimum of $[\abs{f}]$.)

We list all of the properties of $BV(X,\lambda)$ and $\overline{\Var}$ that will be used in the remainder of the work in the following lemma. In particular, all of the estimates we will make can be made $\lambda$-almost everywhere, where for some estimates it is very important that no singletons are assigned positive measure by $\lambda$. For the remainder of the article we will abuse notation and write $f\in BV(X,\lambda)$ and $\Var(f)$ instead of $[f]$ and $\overline{\Var}[f]$.

\begin{lem}
\label{lem:bv-properties}
Let $X$, $\lambda$, and $T$ be as described above ($\cM$).
\begin{enumerate}
\item If $f\in BV(X,\lambda)$ and $C\subset X$ with non-zero measure, then \[ \esssup_C(\abs{f}) \leq \essinf_C(\abs{f}) + \Var(f) \leq \frac{1}{\lambda(C)}\int_C\abs{f}\ d\lambda + \Var(f). \]
\item If $f,g\in BV(X,\lambda)$ and $C\subset X$, then $\Var_C(fg) \leq \norm{\restr{f}{C}}_{\infty}\Var_C(g) + \norm{\restr{g}{C}}_{\infty}\Var_C(f)$.
\item If $f\in BV(X,\lambda)$, $s\in \{+,-\}$, $C$ is a closed interval with non-zero measure, and $(x,s)\in C$, then $f(x^s)$ is well-defined (independent of version of $f$) and satisfies \[ \abs{f(x^s)} \leq \essinf_C(\abs{f}) + \Var_C(f) \leq \frac{1}{\lambda(C)}\int_C\abs{f}\ d\lambda + \Var_C(f). \]
\item If $f\in BV(X,\lambda)$ and $X = \bigcup_{m=1}^M [x_{m-1},x_m]$ with $x_m < x_{m+1}$, then we have \[ \Var(f) = \sum_{m=1}^M \Var_{(x_{m-1},x_m)}(f) + \sum_{m=1}^{M-1} \abs{f(x^+)-f(x^-)}. \] 
\item If $(f_n)_{n=1}^{\infty} \subset BV(X,\lambda)$ with $\sum_n f_n = f$ converging $\lambda$-almost everywhere, then $\Var(f) \leq \sum_n \Var(f_n).$ 
\end{enumerate}
Now, assume that $H = [c,d] \subset X$ is a closed interval with non-zero measure such that $(c,+), (d,-)\in H$. For $f\in BV(X,\lambda)$, define $F_H(f)$ by (the equivalence class of) \[ F_H(f) = \begin{cases}
f(x), & x\in (c,d), \\
0, & x\notin H, \\
f(c^{+}), & x = c, \\
f(d^{-}), & x = d.
\end{cases} \] Then $F_H(f) \in BV(X,\lambda)$, $F_H$ is multiplicative, and we have \[ \Var(F_H(f)) = \Var_{(c,d)}(f) + \abs{f(c^+)}\cdot (1-\mathds{1}_{\{a_X\}}(c)) + \abs{f(d^-)}\cdot (1-\mathds{1}_{\{b_X\}}(d)), \] where $X=[a_X,b_X]$. Finally, for all $f\in L^1(X,\lambda)$, the operator $L$ has the form \[ L(f) = \sum_{n} F_{K_n}\left( g\circ T_n^{-1} \right) F_{K_n}\left( f\circ T_n^{-1} \right). \]
\end{lem}

\subsection{Statement and Proof of the Inequality}
\label{subsect:proof-ly-ineq}

\begin{prop}[Balanced Lasota-Yorke Inequality]
\label{prop:ly-ineq}
Let $X$ and $T$ satisfy the assumptions in $(\cM)$. Suppose further that \[ \sup_{n\in N}\left\{ \frac{\Var_{I_n^o}(g)}{\lambda(I_n)} \right\} < \infty, \quad \text{and} \quad \sum_{(z,s)\in H} g(z^s) < \infty. \] Then for any $f\in BV(X,\lambda)$ and any finite collection of closed intervals $\cJ = \{ J_m \}_{m=1}^M$ with disjoint non-empty interiors such that $\bigcup_{m=1}^M J_m$ contains $H$, we have:
\begin{align*}
\Var(L(f)) & \leq \left( \sup_n\left\{ \norm{\restr{g}{I_n^o}}_\infty + \Var_{I_n^o}(g) \right\} + \max_m\{ h_{\cJ}(m) \} \right) \Var(f) \\
& \hspace{20pt} + \left( \sup_n\left\{ \frac{\Var_{I_n^o}(g)}{\lambda(I_n)} \right\} + \max_m\left\{ \frac{h_{\cJ}(m)}{\lambda(J_m)} \right\} \right) \norm{f}_1,
\end{align*}
where $\displaystyle h_{\cJ}(m) := \sum_{(z,s)\in H\cap J_m} g(z^{s})$.
\end{prop}

\begin{proof}
Recall that $K_n = T_n(I_n)$ and $K_n^o = T(I_n^o)$; let $K_n = [c_n,d_n]$. By Lemma \ref{lem:bv-properties}, we know that for $f\in BV(X,\lambda)$, we have: \[ \Var(L(f)) \leq \sum_{n} \Var\left( F_{K_n}\left( g\circ T_n^{-1}\cdot f\circ T_n^{-1} \right) \right). \] 
Each of the variation terms breaks into three parts: the variation for the interior and the two endpoint terms. Suppose that for $n$, $T_n$ is increasing. Then $T_n(a_n) = c_n$ and $T_n(b_n) = d_n$, and the $1-\mathds{1}_{\{a_X\}}(c_n)$, $1-\mathds{1}_{\{b_X\}}(d_n)$ terms are non-zero exactly when $(a_n,+)$, $(b_n,-)$ are hanging points for $T$. Thus we have 
\begin{equation*}
\begin{split}
g\circ T_n^{-1}(c_n^+)&\abs{f\circ T_n^{-1}(c_n^+)}\cdot (1-\mathds{1}_{\{a_X\}}(c_n)) \\
& + g\circ T_n^{-1}(d_n^-)\abs{f\circ T_n^{-1}(d_n^-)}\cdot (1-\mathds{1}_{\{b_X\}}(d_n)) \\
& = g(a_n^+)\abs{f(a_n^+)}\cdot \mathds{1}_H(a_n,+) + g(b_n^-)\abs{f(b_n^-)}\cdot \mathds{1}_H(b_n,-).
\end{split}
\end{equation*}
The case where $T_n$ is decreasing yields the same equation. For the interior term, since $T_n$ is either increasing or decreasing, we have that \[ \Var_{T(I_n^o)}\left( g\circ T_n^{-1}\cdot f\circ T_n^{-1} \right) = \Var_{I_n^o}(g\cdot f). \] Putting these estimates together, we have \[ \Var(L(f)) \leq \sum_n \Var_{I_n^o}(g\cdot f) + g(a_n^+)\abs{f(a_n^+)}\cdot \mathds{1}_H(a_n,+) + g(b_n^-)\abs{f(b_n^-)}\cdot \mathds{1}_H(b_n,-). \]

For the sum of the variation terms, we use the product inequality and the estimates on the essential supremum to see that \[ \Var_{I_n^o}(g\cdot f) \leq \left( \norm{\restr{g}{I_n^o}}_{\infty} + \Var_{I_n^o}(g) \right)\cdot \Var_{I_n^o}(f) + \frac{\Var_{I_n^o}(g)}{\lambda(I_n^o)}\int_{I_n^o} \abs{f}\ d\lambda. \] We have assumed that the ratio of the variation of $g$ on intervals $I_n^o$ over the measure of $I_n^o$ is bounded in $n$, so we obtain 
\begin{align*} 
\sum_{n} \Var_{I_n^o}(g\cdot f) & \leq \sup_{n\in N} \left\{ \norm{\restr{g}{I_n^o}}_{\infty} + \Var_{I_n^o}(g) \right\} \sum_n \Var_{I_n^o}(f) \\ 
& \hspace{50pt} + \sup_{n\in N}\left\{ \frac{\Var_{I_n^o}(g)}{\lambda(I_n^o)} \right\} \int_X \abs{f}\ d\lambda.
\end{align*}

For the sum of the hanging point terms, we see that it is simply a sum over all of the hanging points for $T$. Again, using Lemma \ref{lem:bv-properties}, for a hanging point $(x,s)$ in the interval $J_m$ we obtain \[ \abs{f(x^s)} \leq \frac{1}{\lambda(J_m)}\int_{J_m} \abs{f}\ d\lambda + \Var_{J_m}(f). \] For $H$ the set of hanging points, we have 
\begin{equation*}
\begin{split}
\sum_{n\in N} \left( g(b_n^+)\abs{f(b_n^-)}\cdot \mathds{1}_H(b_n,+) \right. & \left. + g(a_n^+)\abs{f(a_n^+)}\cdot \mathds{1}_H(a_n,+) \right) \\
& = \sum_{m=1}^M \sum_{(x,s)\in H\cap J_m} g(x^s)\abs{f(x^s)}.
\end{split}
\end{equation*}
We then bound each $\abs{f(x^s)}$ term above by the estimate using the containing interval, and obtain:
\begin{align*}
\sum_{m=1}^M & \sum_{(x,s)\in H\cap J_m} g(x^s)\abs{f(x^s)} \\
& \leq \sum_{m=1}^M \frac{1}{\lambda(J_m)}\left( \sum_{(x,s)\in H\cap J_m} g(x^s) \right)\int_{J_m} \abs{f}\ d\lambda \\
& \hspace{50pt} + \sum_{m=1}^M \left( \sum_{(x,s)\in H\cap J_m} g(x^s) \right) \Var_{J_m}(f) \\
& \leq \max_m\left\{ \frac{h_{\cJ}(m)}{\lambda(J_m)} \right\} \norm{f}_1 + \max_m\left\{ h_{\cJ}(m) \right\} \Var(f),
\end{align*}
where we use the notation $h_{\cJ}(m)$ for the sum of the $g(x^s)$ terms over $(x,s)$ in $J_m$. The proposition statement follows by utilizing the two upper bounds simultaneously.
\end{proof}

\section{Application to Cocycles of Perron-Frobenius Operators}
\label{sect:app-pf-ops}

Let $(\Omega,\cB,\mu,\sigma)$ be a fixed ergodic invertible probability-preserving system. The goal of this section is to use Theorem \ref{thm:pf} to examine random dynamical systems based on the following class of maps.

\begin{defn}
\label{defn:paired-tent}
Let $\epsilon_1,\epsilon_1 \in [0,1]$ be given. Let $T_{\epsilon_1,\epsilon_2} : [-1,1] \to [-1,1]$ be given by:
\[ T_{\epsilon_1,\epsilon_2}(x) = \begin{cases}
2(1+\epsilon_1)(x+1) - 1, & x\in [-1,-1/2], \\
-2(1+\epsilon_1)x - 1, & x\in [-1/2,0), \\
0, & x = 0, \\
-2(1+\epsilon_2)x + 1, & x\in (0,1/2], \\
2(1+\epsilon_2)(x-1) + 1, & x\in [1/2,1].
\end{cases} \]
The map $T_{\epsilon_1,\epsilon_2}$ will be called a \emph{paired tent map}.
\end{defn}

The name arises because there are two tent maps paired together to create a single map. See Figure \ref{fig:tent-graph} for an example. 
\begin{figure}[tb]
\includegraphics[width=0.75\textwidth]{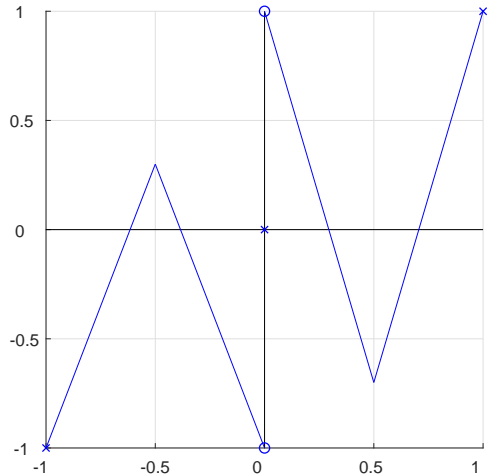}
\caption{The paired tent map, with parameters $\epsilon_1 = 0.3$ and $\epsilon_2 = 0.7$.}
\label{fig:tent-graph}
\end{figure}

Fix measurable functions $\epsilon_1, \epsilon_2 : \Omega \to [0,1]$. For each $\omega \in \Omega$ let $T_{\omega} := T_{\epsilon_1(\omega),\epsilon_2(\omega)}$, and let $L_{\omega}$ be the Perron-Frobenius operator corresponding to $T_\omega$. This generates a random dynamical system: a cocycle of maps $T_{\omega}^{(n)}$ and the associated cocycle of Perron-Frobenius operators $L_{\omega}^{(n)}$ over the base timing system $(\Omega,\cB,\mu,\sigma)$. For shorthand, we will write $T_{\omega}^{(n)}$ for the system, since the base timing system is given.

Our approach towards proving Theorem B involves utilizing our cocycle Perron-Frobenius theorem for bounded operators on a Banach space that preserve and occasionally contract a cone. This means that we need to find a Banach space and a cone, decide what our operators are, show that the operators actually preserve and sometimes contract the cone, and conclude the argument by applying the theorem with concrete estimates on the necessary quantities. This will show that the top Lyapunov space is one-dimensional and that the second Lyapunov exponent is bounded away from $0$ by some quantity computable from $\epsilon_1$ and $\epsilon_2$. To prove Corollary C, we will perform a similar but more intricate calculation making use of the scaling parameter.

For our Banach space, we will use functions of bounded variation, which is a non-separable space. Following Buzzi \cite{buzzi-doc} and as in Section \ref{sect:ly-ineq}, we view $BV[-1,1]$ as a subspace of $L^{\infty}([-1,1],\lambda)$, where the elements of $BV[-1,1]$ are equivalence classes of functions where there exists one member of the class with bounded variation and $\lambda$ is normalized Lebesgue measure. To apply our cocycle Perron-Frobenius theorem in the setting of bounded variation, we must ensure $\mu$-continuity of the cocycle of operators; to do this, we need to assume that $\epsilon_1$ and $\epsilon_2$ have countable range. This is because generally, Perron-Frobenius operators corresponding to different maps on the same space are uniformly far apart in the operator norm on $BV[-1,1]$; this is certainly the case in our situation.

For our cone, it turns out that the space $BV[-1,1]$ has a particularly nice cone, described by Liverani in \cite{liverani-annals-doc}. Given a positive real number $a$, we let the cone $\cC_a$ be given by \[ \cC_a = \set{f\in BV[-1,1]}{f \geq 0,\ \Var(f) \leq a\norm{f}_1}. \] It is easy to see that $\cC_a$ is convex, closed in the $BV[-1,1]$ norm topology, and closed under non-negative scalar multiples. The other properties of $\cC_a$ that are necessary to show for $\cC_a$ to be nice as in Definition \ref{defn:cones} can be found in Appendix \ref{append:d-adapt}.

The question of how to choose the correct cone $\cC_a$ may be approached by using a Lasota-Yorke inequality which holds for the situation. To apply our Cocycle Perron-Frobenius theorem, we want $L_{\omega}(\cC_a) \subset \cC_{\nu a}$ for some $\nu \in (0,1)$ and (almost) all $\omega$. If $L_{\omega}$ satisfies the Lasota-Yorke inequality \[ \Var(L_{\omega}(f)) \leq C_{\Var}\Var(f) + C_1\norm{f}_1 \] uniformly in $\omega$, then for $f\in \cC_a$, we have:
\begin{align*}
\Var(L_{\omega}(f)) & \leq C_{\Var}\Var(f) + C_1\norm{f}_1 \\
& \leq C_{\Var}a\norm{f}_1 + C_1\norm{f}_1 = \left( C_{\Var}a + C_1 \right) \norm{f}_1.
\end{align*}
Then, we find $a > 0$ and $\nu \in (0,1)$ that solve $C_{\Var}a + C_1 \leq \nu a$, which formally rearranges to \[ \frac{C_1}{\nu - C_{\Var}} \leq a. \] For this inequality to make sense with $\nu \in (0,1)$, we must have $C_{\Var} \in [0,1)$, and we must choose $\nu \in (C_{\Var},1)$. Then $a$ may be chosen to be the minimum value. The resulting cone $\cC_a$ would hence be preserved by (almost) every map $T_{\omega}$, and moreover would be mapped into $\cC_{\nu a}$. The cone distance on $\cC_a$ will be denoted by $\theta_a$.

Note that if the $L_{\omega}$ do not satisfy a uniform Lasota-Yorke-type inequality in $\omega$, then we may run into two separate problems. First, it could be that $C_{\Var}$ tends close to $1$, which means we cannot choose a $\nu$ as above, as every fixed $\nu$ would be smaller than some $C_{\Var}$. Second, it could be that $C_1$ tends to infinity, and so any fixed $a > 0$ would be smaller than the minimum possible value allowed by the map (using this method). These two issues are the basis for the requirement of uniformity in $\omega$ of the Lasota-Yorke-type inequality.

In addition, the first problem poses a challenge for our particular cocycle of maps: because we allow $T_{0,0}$ as a map and this map has derivatives with magnitude $2$ almost everywhere, we cannot only consider the first iterate $L_{\omega}$. All Lasota-Yorke-type inequalities will yield $C_{\Var} \geq 1$ for these instances of the map. Instead, we will generalize the idea of taking powers of a map to taking iterates of the cocycle, and use the second iterate cocycle. This allows use of our Lasota-Yorke inequality, though we have to deal with the added complexity of the maps. Moreover, the second iterate is a cocycle over $\sigma^2$, not $\sigma$; this turns out not to be an issue, thanks to standard ergodic theory techniques.

The remainder of this section is organized as follows: first, we briefly explain the reduction from the full cocycle to the second iterate cocycle, and then we use Proposition \ref{prop:ly-ineq} to obtain a uniform-in-$\omega$ Lasota-Yorke-type inequality, and choose our cone accordingly. We subsequently use the idea of covering, following both Liverani \cite{liverani-annals-doc} and Buzzi \cite{buzzi-doc}, to show that iterates of $L^{(2)}_{\omega}$ strictly contract the cone and thus obtain an explicit upper bound for the second-largest Lyapunov exponent for the cocycle. Finally, we investigate what happens when we scale the $\epsilon_1,\epsilon_2$ parameters by $\kappa$ and shrink $\kappa$ to zero, imitating a perturbation of the map $T_{0,0}$, and look at a specific collection of $\kappa$ to see that our bound is in some sense optimal.

\subsection{Reduction to Second-Iterate Cocycle}
\label{subsect:second-iterate}

We will state two well-known general lemmas: the first is a decomposition for maps $\sigma^k$ where $\sigma$ is ergodic, and the second indicates how the Lyapunov exponents of the first and second iterate cocycles are related. The application of the lemmas will occur in Section \ref{subsect:app-contr-cone}.

\begin{lem}[Cyclic Decomposition]
\label{lem:cyclic-decomp}
Let $k\in \Z_{\geq 1}$ and let $(\Omega,\mu,T)$ be an ergodic probability-preserving transformation. Then there exists an integer $l$ dividing $k$ and $A\subset \Omega$ measurable such that $\mu(A) = \frac{1}{l}$, the sets $\left\{ T^{-i}(A) \right\}_{i=0}^{l-1}$ are disjoint, $A = T^{-l}(A)$, and $(A,\restr{\mu}{A},T^k)$ is ergodic. Moreover, if $T$ is invertible, so is $T^k$ on $A$.
\end{lem}

The general point of this lemma is to illustrate the structure of powers of an ergodic map; the finest ergodic decompositions have a maximum number of ergodic components with positive measure, and the components map to each other cyclically via the original map. Moreover, the number of components must be a divisor of the power.

\begin{lem}
\label{lem:multiple-iter-LEs}
Let $L = L_{\omega}^{(n)} : X\to X$ be a cocycle of operators on $X$ over the ergodic, invertible base dynamical system $(\Omega,\mu,\sigma)$, where Lyapunov exponents are well-behaved. Let $P_{\omega} = L_{\omega}^{(k)}$ be the $k$-th iterate cocycle over $\sigma^k : \Omega_1 \to \Omega_1$, where $\Omega_1 \subset \Omega$ and $(\Omega_1,\restr{\mu}{\Omega_1},\sigma^k)$ is ergodic. Then the map $\cdot k : \mathrm{spec}(L) \to \mathrm{spec}(P)$ is an order-preserving bijection, i.e.\ the Lyapunov exponents of $L$ are just the Lyapunov exponents of $P$ scaled by $\frac{1}{k}$.
\end{lem}

In our situation, we have $k=2$, and so either $\sigma^2$ is ergodic on $\Omega$ or it is ergodic on a set $\Omega_1$ with $\mu(\Omega_1) = 1/2$, and its complement $\Omega_2$ satisfies $\sigma(\Omega_1) = \Omega_2$. This plays a role later, because as we will see, the expansion factors of the second iterate depend on both the $\epsilon_i(\omega)$ and $\epsilon_i(\sigma(\omega))$ terms. We will then be able to recover inequalities for the second-largest Lyapunov exponent for the original cocycle from those for the second iterate cocycle. From here on, let $\Omega_1$ be a set on which $\sigma^2$ is ergodic (potentially all of $\Omega$, but possibly only half of the space), and let $\Omega_2$ be its complement.

\subsection{Uniform Balanced Lasota-Yorke-type Inequality}
\label{subsect:app-ly-ineq}

For notation, let $S_{\omega}$ be given by $S_{\omega} := T_{\omega}^{(2)} = T_{\sigma(\omega)}\circ T_{\omega}$, and let the associated Perron-Frobenius operator be $P_{\omega} := L_{\omega}^{(2)} = L_{\sigma(\omega)}\circ L_{\omega}$. $S_\omega$ is the composition of two piecewise linear maps, and is therefore piecewise linear. Moreover, it has finitely many branches, because the two maps in the composition both have finitely many branches. Thus $S_{\omega}$ is piecewise $C^{1}$ on $[-1,1]$, and it is straightforward to verify that $S_{\omega}$ and the space $([-1,1],\lambda)$ satisfy the assumptions in $(\cM)$; one such $S_{\omega}$ is pictured in Figure \ref{fig:sec-tent-small-eps}. We prove the following proposition by applying Proposition \ref{prop:ly-ineq} with explicit numbers.

\begin{figure}[tbp]
\includegraphics[width=0.75\textwidth]{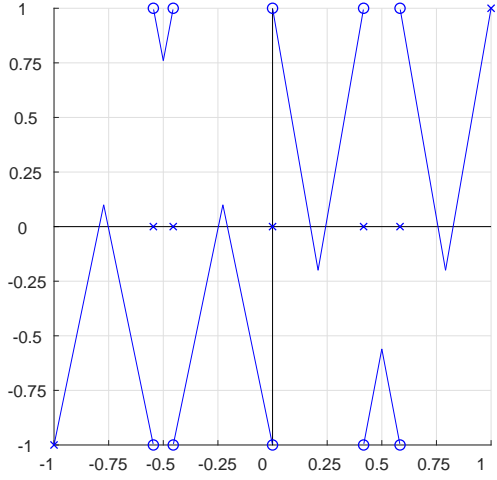}
\caption{The second iterate of the coupled tent map, $S_{\omega}$, with parameters $\epsilon_1(\omega) = 0.1$, $\epsilon_2(\omega) = 0.2$, $\epsilon_1(\sigma(\omega)) = 0.1$, and $\epsilon_2(\sigma(\omega)) = 0.2$.}
\label{fig:sec-tent-small-eps}
\end{figure}

\begin{figure}[tbp] 
\includegraphics[width=0.75\textwidth]{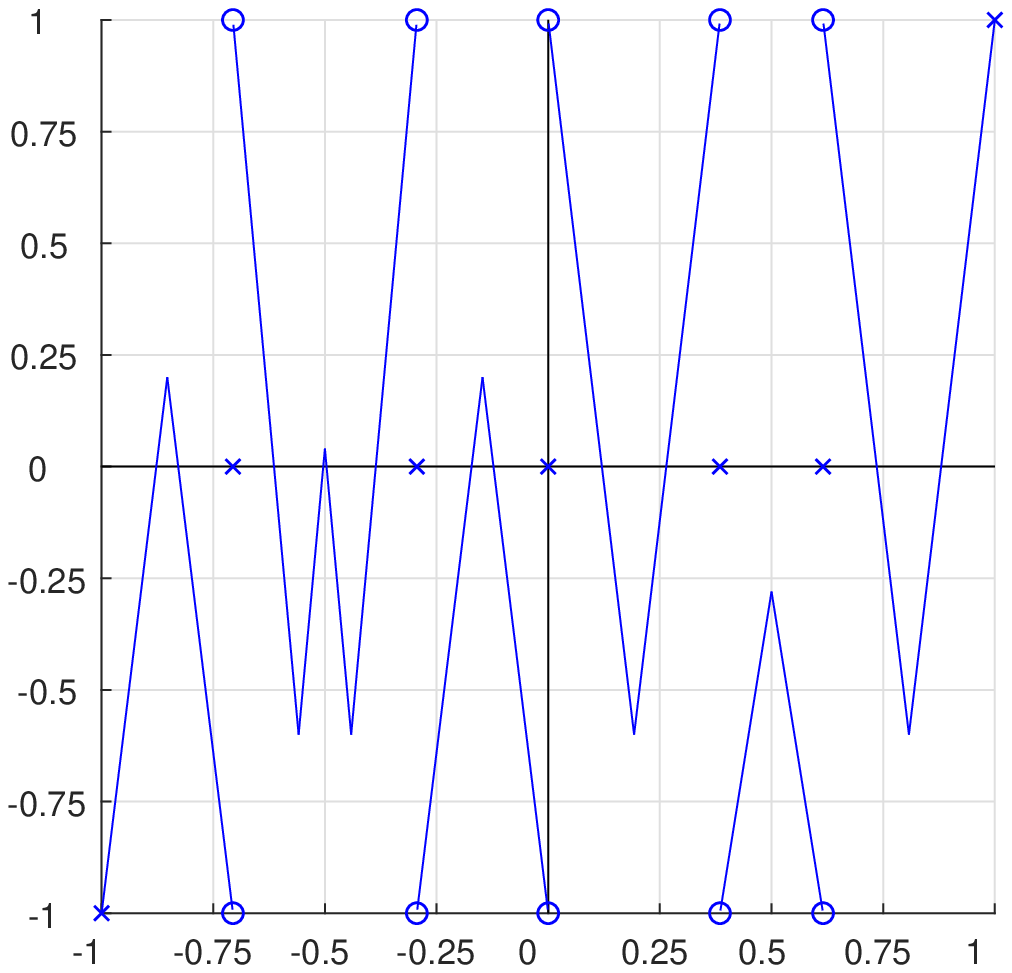}
\caption{The second iterate of the coupled tent map, $S_{\omega}$, with parameters $\epsilon_1(\omega) = 0.7$, $\epsilon_2(\omega) = 0.3$, $\epsilon_1(\sigma(\omega)) = 0.2$, and $\epsilon_2(\sigma(\omega)) = 0.6$.}
\label{fig:sec-tent-large-eps}
\end{figure}

\begin{prop}
\label{prop:tent-map-ly-ineq}
For any paired tent map cocyle $T_{\omega}$ over $\sigma$ and associated second-iterate P-F operator $P_{\omega}$, we have that for any $f\in BV[-1,1]$: \[ \Var(P_{\omega}(f)) \leq \frac{3}{4}\Var(f) + 6\norm{f}_1. \] If moreover $\epsilon_1,\epsilon_2 \leq 1/2$ on $\Omega$, then we have the sharper estimate \[ \Var(P_{\omega}(f)) \leq \frac{1}{2}\Var(f) + 4\norm{f}_1. \]
\end{prop}

\begin{proof}
We first compute the function $g_{\omega} = \abs{DS_{\omega}}^{-1}$. Set $D_i = \left[ \frac{i-3}{2},\frac{i-2}{2} \right]$ for $i=1,2,3,4$; these are the intervals of monotonicity for the maps $T_{\omega}$. Then let $I_{ij}(\omega) = D_i\cap T_{\omega}^{-1}(D_j)$; these are the intervals of monotonicity for $S_{\omega}$. Note that depending on $\epsilon_1(\omega)$ and $\epsilon_2(\omega)$, some of the $I_{ij}(\omega)$ may be empty (or single points, but single point intervals may as well be considered empty from the perspective of hanging points and integrals). When $\omega$ is understood, we will just write $I_{ij}$, and the index set is \[ N = N_{\omega} = \set{ij}{i,j\in \{1,2,3,4\},\ D_i \cap T_{\omega}^{-1}(D_j)\ \text{is non-trivial}}. \] The intervals $I_{ij}$ are ordered left-to-right along $[-1,1]$ as follows: \[ I_{11}, I_{12}, I_{13}, I_{14}, I_{24}, I_{23}, I_{22}, I_{21}, I_{34}, I_{33}, I_{32}, I_{31}, I_{41}, I_{42}, I_{43}, I_{44}. \] In Figure \ref{fig:sec-tent-large-eps}, the only empty intervals are $I_{31}$ and $I_{41}$, since the map does not make a ``W'' shape near $x=1/2$. In Figure \ref{fig:sec-tent-small-eps}, all of $I_{14}$, $I_{24}$, $I_{31}$, and $I_{41}$ are empty. When both $\epsilon_1$ and $\epsilon_2$ are zero, 
all of the intervals $I_{ij}$ surrounding $-1/2$ and $1/2$ (two on each side of each point) are empty; this is where exactly one of $i,j$ is at most $2$ are empty.

For notation, let $\eta_1 = \eta_2 = \epsilon_1$ and $\eta_3 = \eta_4 = \epsilon_2$. By the formula in Definition \ref{defn:paired-tent} and the Chain Rule, we see that for $x\in I_{ij}^o$, \[ g_{\omega}(x) = \frac{1}{4(1+\eta_j(\sigma(\omega))) (1+\eta_i(\omega))}. \] Of course, $g_{\omega} = 0$ at endpoints of the $I_{ij}$.

For different $\omega$, $S_{\omega}$ can look somewhat different, as shown in Figures 
\ref{fig:sec-tent-small-eps} and \ref{fig:sec-tent-large-eps}. These differences turn out to be very important for our analysis, due to the differing numbers of branches and hanging points, and the measures of the branches. This is why the balanced (and looser) Lasota-Yorke-type inequality in Proposition \ref{prop:ly-ineq} becomes useful: we may find a common inequality for all of these maps, regardless of the exact form of the maps, by appropriate choice of intervals $J_m$ for each subcollection of maps.

Note that the map $S_{\omega}$ depends on all four quantities $\epsilon_1(\omega)$, $\epsilon_1(\sigma(\omega))$, $\epsilon_2(\omega)$, and $\epsilon_2(\sigma(\omega))$, but only the two quantities $\epsilon_1(\omega)$ and $\epsilon_2(\omega)$ affect the complexity of the map. By this, we mean that $\epsilon_1(\omega)$ and $\epsilon_2(\omega)$ change the branch structure of the map, since the intervals of monotonicity are $I_{ij} = D_i \cap T_{\omega}^{-1}(D_j)$, and $T_{\omega}$ only depends on $\epsilon_1(\omega)$ and $\epsilon_2(\omega)$. The other two quantities affect the expansion of the branches (how much of the space the branches cover), but they do not change the branch structure. Moreover, the quantities $\epsilon_1(\omega)$ and $\epsilon_2(\omega)$ only affect the behaviour of $S_{\omega}$ on $[-1,0]$ and $[0,1]$, respectively. To simplify the exposition, we will investigate the bounds for different ranges of $\epsilon_1$, obtain those for $\epsilon_2$ by symmetry, and then obtain the complete bound by taking the maximum over the pairings.

So, fix $\omega$; we will consider the intervals in $[-1,0]$, which means we will look at the three cases of $\epsilon_1(\omega)$: $\epsilon_1 = 0$, $\epsilon_1 \in (0,1/2]$, and $\epsilon_1 \in (1/2,1]$. Note that in every case, $g$ is constant on the (interiors of the) intervals of monotonicity $I_{ij}^o$, so that each of the $\Var_{I_{ij}^o}(g)$ terms is zero.

\begin{enumerate}
\item First, assume that $\epsilon_1(\omega) = 0$; this is where in Figure \ref{fig:sec-tent-small-eps}, there would only be the two large tents on $[-1,0]$, and nothing special happening near $x=-1/2$. The hanging points in $[-1,0]$ are $(-3/4,-)$, $(-3/4,+)$, $(-1/4,-)$, and $(-1/4,+)$. Choose \[ J_1 = [-1,-3/4],\ J_2 = [-3/4,-1/2],\ J_3 = [-1/2,-1/4],\ J_4 = [-1/4,0]. \] These happen to be $I_{11}$, $I_{12}$, $I_{22}$, and $I_{21}$, respectively; it is easy to see that each interval contains exactly one of the hanging points. Each of these intervals has measure $\lambda(J_m) = \frac{1}{4}\cdot\frac{1}{2} = \frac{1}{8}$, and we may compute $h_{\cJ}(m)$ for each $m$ using the definition: 
\begin{gather*}
h_{\cJ}(1) = g((-3/4)^-) = \frac{1}{4(1+\epsilon_1(\sigma(\omega)))} = g((-3/4)^+) = h_{\cJ}(2), \\
h_{\cJ}(3) = g((-1/4)^-) = \frac{1}{4(1+\epsilon_1(\sigma(\omega)))} = g((-1/4)^+) = \quad h_{\cJ}(4).
\end{gather*}

\item Next, assume that $\epsilon_1(\omega) \in (0,1/2]$, as in Figure \ref{fig:sec-tent-small-eps}. Here, there are two new hanging points at $x = -1/2$, because the map $T_{\omega}$ leaks mass near $-1/2$ into $[0,1]$ (which one can see in Figure \ref{fig:tent-graph}). The hanging points in $[-1,0]$ are $(t_1,\pm)$, $(1/2,\pm)$, and $(t_2,\pm)$, where \[ t_1(\omega) = -\frac{(3/4+\epsilon_1(\omega))}{1+\epsilon_1(\omega)}, \qquad t_2(\omega) = -\frac{1}{4(1+\epsilon_1(\omega))}; \] the points $t_1$, $1/2$, and $t_2$ are in order, left-to-right. We need to cover the six hanging points with intervals $J_m$, but we have to do this in such a way that the measures of the intervals cannot potentially vanish with shrinking $\epsilon_1$. Doing some simple analysis of $t_1$ and $t_2$ with respect to $\epsilon_1$ shows us that $t_1 \in \left[-\frac{5}{3}, -\frac{3}{4}\right]$ and $t_2 \in \left[-\frac{1}{4}, -\frac{1}{6}\right]$. Using this information, we can choose the following intervals:
\begin{alignat*}{6}
J_1 & {}={} [-1,t_1], & \quad J_2 & {}={} [t_1,-5/8], & \quad J_3 & {}={} [-5/8,-1/2], & \\
J_4 & {}={} [-1/2,-3/8], & \quad J_5 & {}={} [-3/8,t_2], & \quad J_6 & {}={} [t_2,0]. &
\end{alignat*}
Each of these intervals has length at least $\frac{1}{16}$, and contains exactly one hanging point. Moreover, at each hanging point $(x,s)$, \[ g_{\omega}(x^s) = \frac{1}{4(1+\eta_i(\omega))(1+\eta_j(\sigma(\omega)))} \leq \frac{1}{4(1+\epsilon_1(\omega))} \leq \frac{1}{4}, \] which is the best uniform bound in this case, because $\epsilon_1(\omega)$ can approach $0$. This means that each $h_{\cJ}(m)$ is bounded above by $\frac{1}{4}$.

\item Finally, assume that $\epsilon_1(\omega) \in (1/2,1]$, as in Figure \ref{fig:sec-tent-large-eps}. The complexity of the map has increased again, in the sense that there are more branches and there are now ten hanging points. Moreover, this case poses a new difficulty, over the previous cases: depending on $\epsilon_1$, the hanging points in the middle branches can become arbitrarily close together. This means that we have two choices: either we change the intervals based on $\epsilon_1$ and end up with arbitrarily small intervals (as Rychlik does for a single map in \cite{rychlik}), or we restrict the size of our intervals and let there be multiple hanging points in some of the intervals. We cannot use the former option, because we need uniformity in the inequality, and choosing ever smaller intervals means the $\norm{\cdot}_1$ coefficient explodes. Thus, we choose intervals in such a way as to attempt to minimize the contributions from the $h_{\cJ}(m)$, while keeping the intervals from being too small. 

We choose the intervals \[ J_1 = [-1,j_1], \quad J_2 = [j_1,-1/2], \quad J_3 = [-1/2,j_2], \quad J_4 = [j_2,0], \] where $j_1$ and $j_2$ are the two jumps that the map $S_{\omega}$ takes in $[-1,0]$. The intervals $J_1$ and $J_4$ have measures at least $\frac{1}{8}$; the intervals $J_2$ and $J_3$ have measures at least $\frac{1}{12}$. Each of the intervals $J_1$ and $J_4$ contain two hanging points, and we have $h_{\cJ}(1),\ h_{\cJ}(4) \leq \frac{1}{3}.$ For the intervals $J_2$ and $J_3$, each of them contain three hanging points, but we obtain $h_{\cJ}(2),\ h_{\cJ}(3) \leq \frac{1}{2}$. The contribution to $h_{\cJ}$ remains controlled because the larger expansion rate for this $\omega$ counteracts the increase in complexity (in terms of number of branches). This is a key feature of the approach.
\end{enumerate}

As mentioned earlier, all of these computations would be analogous when performed on the other interval, $[0,1]$. So fix $\epsilon_1(\omega)$ and $\epsilon_2(\omega)$, and consider all of the bounds on the measures of the intervals and $h_{\cJ}$ together. We have:
\begin{gather*}
\sup_n\left\{ \norm{\restr{g}{I_n^o}}_\infty + \Var_{I_n^o}(g) \right\} \leq \max\left\{ \frac{1}{6}+0,\frac{1}{4}+0\right\} = \frac{1}{4}, \\
\max_m\left\{ h_{\cJ}(m) \right\} \leq \max\left\{ \frac{1}{4},\frac{1}{3},\frac{1}{2} \right\} = \frac{1}{2}, \\
\sup_{n\in N}\left\{ \frac{\Var_{I_n^o}(g)}{\lambda(I_n^o)} \right\} = 0, \\
\max_m\left\{ \frac{h_{\cJ}(m)}{\lambda(J_m)} \right\} \leq \max\left\{ \frac{1/4}{1/8}, \frac{1/4}{1/16},\frac{1/3}{1/8},\frac{1/2}{1/12} \right\} = 6,
\end{gather*}
allowing for every possible pairing of $\epsilon_1$ and $\epsilon_2$. This means that by Proposition \ref{prop:ly-ineq}, for any $f\in BV[-1,1]$ we have: \[ \Var(P_{\omega}(f)) \leq \left(\frac{1}{4}+\frac{1}{2}\right)\Var(f) + (0+6)\norm{f}_1 = \frac{3}{4}\Var(f) + 6\norm{f}_1. \] This is uniform in $\omega$. For the specific situation of $\epsilon_1,\epsilon_2 \leq 1/2$ on $\Omega$, doing the same for only cases 1 and 2 yields the inequality \[ \Var(P_{\omega}(f)) \leq \frac{1}{2}\Var(f) + 4\norm{f}_1, \] as desired.
\end{proof}

\subsection{Covering Properties}
\label{subsect:covering}

We briefly describe the notions of covering to be used, following Buzzi \cite{buzzi-doc}, and supply the main lemmas. We work in the $\lambda$-almost everywhere setting as Buzzi does, which differs from Liverani in \cite{liverani-bv-doc}.

\begin{defn}
\label{defn:covering}
Let $T_{\omega} : [-1,1]\to [-1,1]$ generate a cocycle of maps over $(\Omega,\mu,\sigma)$, each map satisfying the assumptions $(\cM)$, and suppose that for almost every $\omega$, $\lambda(T_{\omega}([-1,1])) = 1$. Let $L_{\omega}$ be the associated Perron-Frobenius operator. We say that the cocycle has the \emph{dynamical covering} property when for $\mu$-almost every $\omega$ and every interval $I\subset [-1,1]$ with positive measure, there exists an $m = m(\omega,I)$ such that for all $m' \geq m$, $\lambda(T^{m'}_{\omega}(I)) = 1$. We say that the cocycle has the \emph{functional covering} property when for $\mu$-almost every $\omega$ and every interval $I\subset [-1,1]$ with positive measure, there exists an $m = m(\omega,I)$ such that for all $m' \geq m$, \[ \essinf(L^{(m')}_{\omega}(\mathds{1}_{I})) > 0. \]
\end{defn}

In the two definitions of covering, we required an almost-onto condition: for almost every $\omega$, $\lambda(T_{\omega}([-1,1])) = 1$. Without this condition, neither covering property can be true. The next lemma indicates the relative strengths of the conditions and provides a computational estimate.

\begin{lem}
\label{lem:covering}
Let $T_{\omega} : [-1,1]\to [-1,1]$ generate a cocycle of maps over $(\Omega,\mu,\sigma)$, where each $T_{\omega}$ satisfies $(\cM)$ and has $\lambda(T_{\omega}([-1,1])) = 1$. If the cocycle has the functional covering property, then it also has the dynamical covering property. If the cocycle has the dynamical covering property and for $\mu$-almost every $\omega$, $\essinf(g_{\omega}) > 0$, then it has the functional covering property. In either case, if $m_{DC}$ and $m_{FC}$ are the integers in the definitions of dynamical and functional covering, respectively, then $m_{DC}(\omega,I) = m_{FC}(\omega,I)$; moreover, if in this case we have that for $\mu$-almost every $\omega$, $\essinf(g_{\omega}) > 0$, then for all $m' \geq m_{FC}(\omega,I)$, we have \[ \essinf\left(L^{(m')}_{\omega}(\mathds{1}_{I})\right) \geq \prod_{j=0}^{m'-1} \essinf(g_{\sigma^{j}(\omega)}) > 0. \]
\end{lem}

Alone, the functional covering property does not allow us to perform explicit computations with the Perron-Frobenius operator acting on the cone $\cC_a$, because it only tells us about push-forwards of characteristic functions. The following lemma allows us to say something about functions in the cone $\cC_a$. It is similar to Lemma 3.2 in \cite{liverani-bv-doc}.

\begin{lem}
\label{lem:strong-func-cover}
Suppose that $T_{\omega} : [-1,1] \to [-1,1]$ generates a cocycle of almost-onto maps satisfying $(\cM)$, and suppose that it satisfies the functional covering property. Let $a > 0$, and let $\cQ = \{Q_k\}_{k=1}^K$ be a partition of $[-1,1]$ into closed intervals with $\lambda(Q_k) \leq (2a)^{-1}$ for each $k$. Let $m_C(\omega,\cQ) = \max_k\{m_{DC}(\omega,Q_k)\}.$ Then for $\mu$-almost every $\omega$, for all $h\in \cC_a$, and for $m \geq m_C(\omega,\cQ)$, we have \[ \essinf(L^{(m)}_{\omega}(h)) \geq \frac{1}{2} \min_k \{ \essinf(L^{(m)}_{\omega}(\mathds{1}_{Q_k})) \} \norm{h}_1. \] If moreover, $\essinf(g_\omega) > 0$ for almost every $\omega$, then \[ \essinf(L^{(m)}_{\omega}(h)) \geq \frac{1}{2} \left( \prod_{i=0}^{m-1} \essinf(g_{\sigma^i(\omega)}) \right) \norm{h}_1. \]
\end{lem}

Thanks to this lemma, we have an almost-everywhere positive lower bound for the images of elements of the cone $\cC_a$ with unit integral after $m \geq m_C(\omega,\cQ)$ steps: \[ \frac{1}{2}\prod_{j=0}^{m-1} \essinf(g_{\sigma^{j}(\omega)}). \] We call $m_C(\omega,\cQ)$ the \emph{covering time} (with respect to $\cQ$).

\subsection{Preservation and Contraction of the Cone}
\label{subsect:app-contr-cone}

We now state and prove the precise version of Theorem \ref{thm:spec-gap}, quoted in the Introduction as Theorem B. 

\begin{thm}
\label{thm:spec-gap}
Let $(\Omega,\mu,\sigma)$ be an ergodic, invertible, probability-preserving transformation, and let $\epsilon_1, \epsilon_2 : \Omega \to [0,1]$ be measurable functions which are both not $\mu$-a.e.\ equal to $0$ and which both have countable range. Let
\begin{gather*} 
M_{\epsilon_1,\epsilon_2} = \frac{1}{2}\min\left\{ \int_{\Omega}\epsilon_1\ d\mu, \int_{\Omega}\epsilon_2\ d\mu \right\}, \\
D_{\epsilon_1,\epsilon_2} = 4(1+\max\{\esssup(\epsilon_1),\esssup(\epsilon_2)\})^2.
\end{gather*} 
Let $T_{\omega} = T_{\epsilon_1(\omega),\epsilon_2(\omega)}$ be generate a cocycle $T_{\omega}^{(n)}$ of paired tent maps, and suppose that $\sigma^2$ is ergodic on $\Omega_1 \subset \Omega$ with the restriction of $\mu$ to $\Omega_1$, with $\mu(\Omega_1) > 0$. Then there exists an explicit set $G_P\subset \Omega_1$ with positive measure and explicit numbers $a > 0$, $\nu \in (0,1)$, and $d\in \Z_{\geq 1}$ such that upon setting \[ k_P = \left\lceil\frac{\log(a)}{\log(1.5)}\right\rceil + 1 + d + \left\lceil -\frac{\log(M_{\epsilon_1,\epsilon_2})}{\log(4)} \right\rceil \] and $C = C(\epsilon_1,\epsilon_2,a,\nu,d)$ equal to \[ \frac{\restr{\mu}{\Omega_1}(G_P)}{2k_P}\log\left( \tanh\left( -\frac{1}{4}\log\left( \frac{2(1+\nu)}{1-\nu}(1 + \nu{}a)\right) + \frac{1}{4}k_P\log(D_{\epsilon_1,\epsilon_2})\right) \right) \] we have: \[ \lambda_2 \leq C(\epsilon_1,\epsilon_2,a,\nu,d) < 0 = \lambda_1, \] where $\lambda_1$ and $\lambda_2$ are the largest and second-largest Lyapunov exponents for the cocycle of Perron-Frobenius operators for $T_{\omega}^{(n)}$.
\end{thm}

Note that this upper bound for $\lambda_2$ is not unique; it is simply the result of our particular method of proof. In particular, in the next section we use a smaller $G_P$ and a larger $k_P$ to obtain a better diameter bound $D_P$. The result is an upper bound which happens to have a much nicer asymptotic property as the $\epsilon_1$ and $\epsilon_2$ parameters are scaled towards $0$, but which requires a usage of Birkhoff's ergodic theorem to obtain the asymptotic relationship, holding only for sufficiently small scaling parameters. Here, we have chosen to use the smallest $k_P$, and the set $G$ is computable simply from the maps $\epsilon_1$ and $\epsilon_2$.

The proof will proceed via Lemmas \ref{lem:interval-exp} and \ref{lem:leakage-time} and Corollary \ref{cor:tent-dyn-cover}, specific to these maps, that bound the $m_{DC}(\omega,I)$ terms, which will be combined with Lemmas \ref{lem:proj-met-bound} and \ref{lem:strong-func-cover} to bound the diameter of the image of the cone $\cC_a$. The computation involved allows us to explicitly control the covering time $m_C(\omega,\cQ)$ for a fixed choice of $\cQ$ and on a positive measure set of $\omega$, which is enough to apply our general theory to obtain explicit bounds on the second Lyapunov exponent; the covering time turns out to be our contraction time.

The following lemma describes an upper bound for the distance from an element $f$ in the subcone $\cC_{\nu a} \subset \cC_a$ to the constant $\mathds{1}$, in terms of the essential infimum and supremum of $f$. (This inequality is why we need estimates on both of these quantities, as well as requiring the cone to be mapped into the subcone $\cC_{\nu a}$.) It is essentially Lemma 3.1 in Liverani \cite{liverani-bv-doc}. 

\begin{lem}
\label{lem:proj-met-bound}
Let $f\in \cC_{\nu a} \subset BV(X)$, with $a > 0$ and $\nu \in (0,1)$, such that $\norm{f}_1 = 1$. Then we have: \[ \theta_{a}(f,\mathds{1}) \leq \log\left( \frac{1+\nu}{1-\nu}\cdot\frac{\esssup(f)}{\essinf(f)} \right). \]
\end{lem}

The next lemma describes how quickly intervals expand under the action of $S_{\omega}$ to cover one of $[-1,0]$ or $[0,1]$.

\begin{lem}
\label{lem:interval-exp}
Let $T_{\omega}$ generate a cocycle of paired tent maps over $(\Omega,\mu,\sigma)$, and let $S_{\omega}$ generate the cocycle of second iterate maps over $(\Omega_1,\restr{\mu}{\Omega_1},\sigma^2)$.
\begin{enumerate}
\item If $I \subset [-1,1]$ is an interval with positive measure of the form $[-1,b]$, $[b,0]$, $[0,b]$, or $[b,1]$, then $S_{\omega}^{(m)}(I)$ contains $[-1,0]$ (if $I$ is $[-1,b]$ or $[b,0]$) or $[0,1]$ (if $I$ is $[0,b]$ or $[b,1]$) in at most \[ m = \left\lceil \frac{-\log(2\lambda(I))}{\log(4)} \right\rceil \] steps.
\item If $I \subset [-1,1]$ is an interval with positive measure, then $S_{\omega}^{(m)}(I)$ contains one of $[-1,0]$ or $[0,1]$ in at most \[ m = \left\lceil \frac{-\log(2\lambda(I))}{\log(1.5)} \right\rceil + 1 \] steps.
\end{enumerate}
\end{lem}

\begin{proof}
First, let $I \subset [-1,1]$ be an interval with positive measure of the form $[-1,b]$; the other cases are analogous. If $I$ contains $I_{11}$, then we have \[ S_{\omega}(I) \supset S_{\omega}(I_{11}) = T_{\sigma(\omega)}\big([-1,0]\big) \supset [-1,0]. \] If $I$ is contained in $I_{11}$, then because $S_{\omega}$ restricted to $I_{11}$ is an affine map with expansion at least $4$, we have that $\lambda(S_{\omega}(I)) \geq 4\lambda(I).$ If $S_{\omega}(I)$ covers $I_{11}$, then we apply the next iterate and are finished, but if not we continue to iterate, to obtain \[ \lambda(S_{\omega}^{(n)}(I)) \geq 4^n\lambda(I). \] We are looking for $S_{\omega}^{(n)}(I)$ to cover $[-1,0]$, which has measure $1/2$, so we solve $4^n\lambda(I) \geq 1/2$ to obtain \[ n \geq \left\lceil \frac{-\log(2\lambda(I))}{\log(4)} \right\rceil. \] Potentially, it could take fewer steps than that because the expansion rate could be larger. The other cases are analogous.

For the second part, we begin by restricting $I$ to lie inside of one of the intervals of continuity for $S_{\omega}$, which in the index notation are:
\begin{gather*} 
I_{11}\cup I_{12}, \quad I_{13}\cup I_{14} \cup I_{24} \cup I_{23}, \quad I_{22}\cup I_{21}, \\
I_{34}\cup I_{33}, \quad I_{32}\cup I_{31} \cup I_{41} \cup I_{42}, \quad I_{43}\cup I_{44}.
\end{gather*}
When we push $I$ forward by $S_{\omega}$, $I$ expands according to the slope of $S_{\omega}$, but there might be overlap in the image due to the number of monotonicity branches. Where there are two branches (for example, in $I_{11}\cup I_{12}$), we know that the slope is at least $4$, and there is at worst the possibility that the interval image exactly overlaps on each branch, so we observe that \[ \lambda(S_{\omega}(I)) \geq \frac{4}{2}\lambda(I) = 2\lambda(I). \] On the other hand, where there are four branches (for example, in $I_{13}\cup I_{14} \cup I_{24} \cup I_{23}$ where the middle branches are non-trivial), the scale factor is at least $6$ (by looking at the formula for the derivative of $S_{\omega}$), and at worst there are four overlapping sections of the image. This means that \[ \lambda(S_{\omega}(I)) \geq \frac{6}{4}\lambda(I) = 1.5\lambda(I). \] Thus in all restricted situations, $\lambda(S_{\omega}(I)) \geq 1.5\lambda(I)$.

Continuing to work with the restricted intervals $I$, write $S_{\omega}(I)$ as \[ S_{\omega}(I) = \big( S_{\omega}(I)\cap [-1,0] \big) \cup \big( S_{\omega}(I)\cap [0,1] \big). \] We have two cases: in one case, as we apply the cocycle to $I$ the resulting set does not split into two pieces of positive measure (one in $[-1,0]$, the other in $[0,1]$). The scale factor is at least $1.5$ at each step, and so we solve $(1.5)^{n}\lambda(I) \geq 1/2$ to find a bound on the number of steps it takes to cover one of $[-1,0]$ or $[1,0]$: \[ n \leq \left\lceil \frac{-\log(2\lambda(I))}{\log(1.5)} \right\rceil. \] In the other case, at some point the image of $I$ splits into two intervals, one contained in $[-1,0]$ and the other in $[0,1]$. By the first part of the lemma, the resulting intervals scale in length by a factor of $4$. If we only look at the larger of the two intervals, the size of the interval is at least $\frac{1.5}{2}\lambda(S_{\omega}^{(k)}(I)) = 0.75\lambda(S_{\omega}^{(k)}(I))$, and in the next step the length scales by $4$, and $4\cdot 0.75 > (1.5)^2$. We therefore see that the number of steps it takes the image of $I$ under $S^{(m)}_{\omega}$ to cover one of $[-1,0]$ or $[0,1]$ is no more than in the case where the interval does not split, except if there is a split in the interval instead of covering all of $[-1,0]$ or $[0,1]$. In this case, the next iterate produces the covering, and so we simply add $1$ to our previous bound.

Finally, suppose that $I$ is an interval that overlaps at least two adjacent intervals of continuity, listed above. If it contains at least two adjacent intervals of continuity, one of them must have an image that covers $[-1,0]$ or $[0,1]$ in one step. So, suppose that $I$ overlaps at most three adjacent intervals of continuity. Then $I$ intersects one of the intervals with more than one-third of its measure; call the intersection $I'$. If $I'$ has more than two branches of monotonicity of $S_{\omega}$, then $I$ contains a full branch of monotonicity over one of $I_{13}$, $I_{23}$, $I_{32}$, or $I_{42}$, and the images of these intervals are $[0,1]$ or $[-1,0]$, respectively, because at least one of $I_{14}$, $I_{24}$, $I_{31}$, or $I_{41}$ is non-trivial. Finally, suppose that $I'$ has at most two branches of monotonicity of $S_{\omega}$. Then we have \[ \lambda(S_{\omega}(I')) \geq \frac{1}{2}\cdot 4 \cdot \frac{1}{3}\lambda(I) = \frac{2}{3}\lambda(I). \] This image interval is of the form $[-1,b]$ or $[b',1]$, so in the next step its measure scales with factor $4$, and $4\cdot \frac{2}{3} > (1.5)^2$, so in two steps $I'$ expands faster than expanding under the $1.5$ scale factor. The argument from the case where $I$ was contained in an interval of continuity now concludes the proof.
\end{proof}

A key aspect of the paired tent maps, when $\epsilon_1$ and $\epsilon_2$ take non-zero values, is that mass can ``leak'' from $[-1,0]$ to $[0,1]$ and vice versa. The next lemma describes how long it can take for leakage to occur in both direction, for certain $\omega$.

\begin{lem}
\label{lem:leakage-time}
Let $T_{\omega}$ generate a cocycle of paired tent maps over $(\Omega,\mu,\sigma)$ with $\epsilon_1,\epsilon_2 : \Omega \to [0,1]$ having positive integrals, and let $S_{\omega}$ generate the cocycle of second iterate maps over $(\Omega_1,\restr{\mu}{\Omega_1},\sigma^2)$. Let \[ M_{\epsilon_1,\epsilon_2} = \frac{1}{2}\min\left\{\int_{\Omega} \epsilon_1\ d\mu,\int_{\Omega} \epsilon_2\ d\mu \right\}. \] Then there exists an explicit set $\cG \subset \Omega_1$ with positive measure and an explicit positive integer $d$ such that for every $\omega \in \cG$, \[ 2\lambda\big(S_{\omega}^{(d)}([-1,0]) \cap [0,1]\big) \geq M_{\epsilon_1,\epsilon_2} \quad \text{and}\quad 2\lambda\big(S_{\omega}^{(d)}([0,1]) \cap [-1,0]\big) \geq M_{\epsilon_1,\epsilon_2}. \] In particular, $S_{\omega}^{(d)}([-1,0]) \cap [0,1]$ and $S_{\omega}^{(d)}([0,1]) \cap [-1,0]$ each contain intervals of the form $[0,a]$, $[a,1]$, $[-1,a]$, and $[a,0]$ with total measure at least $\frac{1}{2}M_{\epsilon_1,\epsilon_2}$.
\end{lem}

\begin{proof}
Recall that $\Omega_2 = \Omega\setminus \Omega_1$. We define $G_1$ and $G_2$ by \[ G_k = \left( \Omega_1\cap \epsilon_k^{-1}\left[ M_{\epsilon_1,\epsilon_2}, 1 \right] \right) \cup \left( \sigma^{-1}\left( \Omega_2\cap \epsilon_k^{-1}\left[ M_{\epsilon_1,\epsilon_2}, 1 \right] \right) \right), \] for $k=1,2$. By ergodicity of $\sigma^{2}$ on $\Omega_1$, there exists a smallest $d_{12}\in \Z_{\geq 1}$ such that $G_1 \cap \sigma^{-2d_{12}}(G_2)$ has positive measure, and there exists a smallest $d_{21}$ such that $G_2 \cap \sigma^{-2d_{21}}(G_1)$ has positive measure. Set $d = \max\{d_{12},d_{21}\}$, and define \[ \cG = \left( G_1 \cap \bigcup_{i=0}^{d-1} \sigma^{-2i}(G_2) \right) \cup \left( G_2 \cap \bigcup_{i=0}^{d-1} \sigma^{-2i}(G_1) \right). \] This set has positive measure, as it contains both $G_1 \cap \sigma^{-2d_{12}}(G_2)$ and $G_2 \cap \sigma^{-2d_{21}}(G_1)$. 

To show that $\cG$ has the required property, observe from the form of $S_{\omega}$ that:
\begin{gather*}
S_{\omega}([-1,0])\cap [0,1] = [0,\epsilon_1(\sigma(\omega))] \cup [1-2(1+\epsilon_2(\sigma(\omega)))\epsilon_1(\omega),1], \\
S_{\omega}([0,1])\cap [-1,0] = [-1,-1+2(1+\epsilon_1(\sigma(\omega)))\epsilon_2(\omega)] \cup [-\epsilon_2(\sigma(\omega)),0].
\end{gather*}
Let $\omega \in \cG$. Then we have two cases. Suppose first that $\omega \in G_1 \cap \bigcup_{i=0}^{d-1} \sigma^{-2i}(G_2)$. By definition of $G_1$, we have $\max\{\epsilon_1(\omega),\epsilon_1(\sigma(\omega)) \} \geq M_{\epsilon_1,\epsilon_2}$, and since $2(1+\epsilon_2(\sigma(\omega)))\epsilon_1(\omega) \geq \epsilon_1(\omega)$, we see that \[ 2\lambda(S_{\omega}([-1,0])) \geq M_{\epsilon_1,\epsilon_2}. \] Then, we know that $\sigma^{d_{12}}(\omega) \in G_2$, so that \[ 2\lambda(S_{\omega}^{(2d_{12})}([0,1])\cap [-1,0]) \geq \max\{ \epsilon_2(\sigma^{2d_{12}}(\omega)), \epsilon_2(\sigma^{2d_{12}+1}(\omega)) \} \geq M_{\epsilon_1,\epsilon_2}. \] All intervals making up $S_{\omega}([-1,0])\cap [0,1]$ and $S_{\omega}^{(d_{12})}([0,1])\cap [-1,0]$ continue to expand under future iterates of the cocycle, and so we have \[ \min\{2 \lambda(S_{\omega}^{(d)}([-1,0])\cap [0,1]), 2\lambda(S_{\omega}^{(d)}([0,1])\cap [-1,0]) \} \geq M_{\epsilon_1,\epsilon_2}. \] The analogous argument holds for $\omega \in G_2 \cap \bigcup_{i=0}^{d-1} \sigma^{-2i}(G_1)$. This completes the proof.
\end{proof}

The previous two lemmas allow us to show that intervals with positive measure cover the whole space by first covering one of $[-1,0]$ or $[0,1]$ and then leaking to the other side and expanding. The following corollary gives a precise statement with a quantitative estimate on the covering time.

\begin{cor}
\label{cor:tent-dyn-cover}
Let $T_{\omega}$ generate a cocycle of paired tent maps over $(\Omega,\mu,\sigma)$ with $\epsilon_1,\epsilon_2 : \Omega \to [0,1]$ having positive integrals, and let $S_{\omega}$ generate the cocycle of second iterate maps over $(\Omega_1,\restr{\mu}{\Omega_1},\sigma^2)$. Let $M_{\epsilon_1,\epsilon_2}$, $\cG$, and $d$ be as in Lemma \ref{lem:leakage-time}. For all $\tau \in (0,1)$, let \[ m_1(\tau) = \left\lceil \frac{-\log(2\tau)}{\log(1.5)} \right\rceil + 1, \] let $m_2(\omega,\tau) = \min\set{m\geq 0}{\sigma^{2(m_1(\tau)+m)}(\omega) \in \cG}$, and let \[ m_3 = \left\lceil \frac{-\log(M_{\epsilon_1,\epsilon_2})}{\log(4)} \right\rceil. \] Then $m_2(\omega,\tau)$ is finite for almost every $\omega$, and if $I \subset [-1,1]$ is an interval with $\lambda(I) \geq \tau$, we have \[ m_{DC}(\omega,I) \leq m_1(\tau) + m_2(\omega,\tau) + d + m_3. \] Thus $S_{\omega}^{(n)}$ has the dynamical covering property.
\end{cor}

\begin{proof}
Fix $\tau \in (0,1)$, and let $I\subset [-1,1]$ be an interval with $\lambda(I) \geq \tau$. By Lemma \ref{lem:interval-exp}(b), we see that $S_{\omega}^{(m)}(I)$ contains one of $[-1,0]$ or $[0,1]$ in at most \[ \left\lceil \frac{-\log(2\lambda(I))}{\log(1.5)} \right\rceil + 1 \leq \left\lceil \frac{-\log(2\tau)}{\log(1.5)} \right\rceil + 1 = m_1(\tau) \] steps, which can be rephrased as saying $S_{\omega}^{(m_1(\tau))}(I)$ covers one of $[-1,0]$ or $[0,1]$ for all $k$. By Lemma \ref{lem:leakage-time}, there exists an explicit set $\cG \subset \Omega_1$ with positive measure and a positive integer $d$ such that for every $\omega \in \cG$, \[ 2\lambda(S_{\omega}^{(d)}([-1,0]) \cap [0,1]) \geq M_{\epsilon_1,\epsilon_2} \quad \text{and}\quad 2\lambda(S_{\omega}^{(d)}([0,1]) \cap [-1,0]) \geq M_{\epsilon_1,\epsilon_2}. \] By ergodicity of $\sigma^2$ on $\Omega_1$, we have that $\bigcup_{m=m_1(\tau)}^{\infty} \sigma^{-2m}(\cG)$ is a set of full measure, and thus $m_2(\omega,\tau) = \min\set{m\geq 0}{\sigma^{2(m_1(\tau)+m)}(\omega) \in \cG}$ is finite for almost every $\omega$. Since $S_{\omega}^{(m_1(\tau))}(I)$ contains one of $[-1,0]$ or $[0,1]$ for all $\omega$ and $S_{\omega}$ is onto for all $\omega$, we have that $S_{\omega}^{(m_1(\tau)+m_2(\omega,\tau))}(I)$ also contains one of $[-1,0]$ or $[0,1]$, for almost every $\omega$. Then in $d$ more iterates, this set is guaranteed to leak into the other half of the space $[-1,1]$, with minimum measure $\frac{1}{2}M_{\epsilon_1,\epsilon_2}$. This leakage takes the form of one or two intervals with one endpoint being $-1$, $0$, or $1$, as appropriate. By Lemma \ref{lem:interval-exp}(a), we see that the leaked mass expands to cover the remainder of the space $[-1,1]$ in at most \[ m_3 = \left\lceil \frac{-\log\left(2\cdot\frac{1}{2}M_{\epsilon_1,\epsilon_2}\right)}{\log(4)} \right\rceil = \left\lceil \frac{-\log(M_{\epsilon_1,\epsilon_2})}{\log(4)} \right\rceil \] steps. Putting all of this together, we get that for almost every $\omega \in \Omega_1$, \[ S_{\omega}^{(m_1(\tau) + m_2(\omega,\tau) + d + m_3)}(I) = [-1,1]. \] Hence $m_{DC}(\omega,I) \leq m_1(\tau) + m_2(\omega,\tau) + d + m_3 < \infty$, and so $S_{\omega}^{(n)}$ has the dynamical covering property.
\end{proof}

\begin{proof}[Proof of Theorem \ref{thm:spec-gap}.]
We are now equipped to prove Theorem \ref{thm:spec-gap}. Choose $\nu \in (3/4,1)$ and $a = \left\lceil \frac{6}{\nu-3/4} \right\rceil$, so that $S_{\omega}(\cC_a) \subset \cC_{\nu a}$ for all $\omega$. Since $a$ is an integer, let $\cQ = \{Q_k\}_{k=1}^{2a}$ be a uniform partition of $[-1,1]$ into closed intervals, so that $\lambda(Q_k) = (2a)^{-1}$ for all $k$. In Corollary \ref{cor:tent-dyn-cover}, set $\tau = (2a)^{-1}$, and let $G_P = \sigma^{-2m_1(\tau)}(\cG)$, where $\cG$ is the set from Lemma \ref{lem:leakage-time}. For $\omega \in G_P$, we have $m_2(\omega,\tau) = 0$. Thus for each $k$, we have
\begin{align*} 
m_{DC}(\omega,Q_k) & \leq m_1((2a)^{-1}) + m_2(\omega,\tau) + d + m_3 \\
& = \left\lceil\frac{\log(a)}{\log(1.5)}\right\rceil + 1 + d + \left\lceil -\frac{\log(M_{\epsilon_1,\epsilon_2})}{\log(4)} \right\rceil = k_P.
\end{align*} 

In the proof of Proposition \ref{prop:tent-map-ly-ineq}, we computed $g_{\omega}$, the weight function for the Perron-Frobenius operator $P_{\omega}$ associated to $S_{\omega}$. In particular, we have \[ \essinf(g_{\omega}) \geq \frac{1}{4(1+\max\{\esssup(\epsilon_1),\esssup(\epsilon_2)\})^2} = D_{\epsilon_1,\epsilon_2} > 0 \] for all $\omega$. Therefore, by Lemma \ref{lem:covering} and Corollary \ref{cor:tent-dyn-cover} the cocycle $S_{\omega}^{(n)}$ has the functional covering property. Let $m_{C}(\omega,\cQ) = \max_k\{m_{DC}(\omega,Q_k)\}$ as in Lemma \ref{lem:strong-func-cover}, and observe that $k_P \geq m_C(\omega,\cQ)$ for every $\omega \in G_P$. Applying Lemma \ref{lem:strong-func-cover} gives us that for every $h\in \cC_a$ with $\norm{h}_1 = 1$ and $\omega \in G_P$, we have \[ \essinf(P^{(k_P)}_{\omega}(h)) \geq \frac{1}{2} \prod_{i=0}^{k_P-1} \essinf(g_{\sigma^i(\omega)}) \geq \frac{1}{2}D_{\epsilon_1,\epsilon_2}^{k_P}. \] By Lemma \ref{lem:bv-properties}, we have that \[ \esssup(P_{\omega}^{(k_P)}(h)) \leq \essinf(P_{\omega}^{(k_P)}(h)) + \Var(P_{\omega}^{(k_P)}(h)). \] In general, $\essinf(\cdot) \leq \norm{\cdot}_1$, and we know $P_{\omega}^{(k_P)}(h) \in \cC_{\nu{}a}$, so that \[ \esssup(P_{\omega}^{(k_P)}(h)) \leq \norm{P_{\omega}^{(k_P)}(h)}_1 + a\norm{P_{\omega}^{(k_P)}(h)}_1 = (1+a)\norm{h}_1 = 1+a, \] since $P_{\omega}$ preserves $\norm{\cdot}_1$. Inserting these bounds into Lemma \ref{lem:proj-met-bound} gives us
\begin{align*}
\theta_a(P_{\omega}^{(k_P)}(h),\mathds{1}) & \leq \log\left( \frac{1+\nu}{1-\nu}\cdot\frac{\esssup(P_{\omega}^{(k_P)}(h))}{\essinf(P_{\omega}^{(k_P)}(h))} \right) \\
& \leq \log\left( \frac{2(1+\nu)}{1-\nu}(1 + \nu{}a)\right) + k_P\log(D_{\epsilon_1,\epsilon_2}).
\end{align*} Using the triangle inequality and the scale-invariance of the projective metric $\theta_a$, we obtain that for all $\omega \in G_P$,
\begin{align*}
\diam_{\theta_a}(P_{\omega}^{k_P}(\cC_a)) & = \sup\set{\theta_a(f_1,f_2)}{f_1,f_2 \in P_{\omega}^{m_C}(\cC_a)} \\
& \leq 2\sup\set{\theta_a(P_{\omega}^{(k_P)}(h),\mathds{1})}{h \in \cC_a,\ \norm{h}_1 = 1} \\
& \leq 2\log\left( \frac{2(1+\nu)}{1-\nu}(1 + \nu{}a)\right) + 2k_P\log(D_{\epsilon_1,\epsilon_2}) =: D_P.
\end{align*}

We now apply the cocycle Perron-Frobenius theorem, Theorem \ref{thm:pf}, to $S_{\omega}^{(n)}$ with parameters $G_P$, $k_P$, and $D_P$. The result is that if $\lambda_1'$ and $\lambda_2'$ are the first and second largest Lyapunov exponents for the cocycle $S_{\omega}^{(n)}$, we have \[ \lambda_2' \leq \frac{\restr{\mu}{\Omega_1}(G_P)}{k_P}\log\left( \tanh\left( -\frac{1}{4}\log\left( \frac{2(1+\nu)}{1-\nu}(1 + \nu{}a)\right) + \frac{1}{4}k_P\log(D_{\epsilon_1,\epsilon_2})\right) \right) + \lambda_1'. \] We know that $\lambda_1' = 0$, because $\norm{h}_1 \leq \norm{h}_{BV} \leq (1+a)\norm{h}_1$ for all $h\in \cC_a$ and $P_{\omega}$ preserves the integral of non-negative functions; every function in $\cC_a$ has Lyapunov exponent $\lambda_1'$ for $P_{\omega}$, so we get $0 \leq \lambda_1' \leq 0$. Finally, to convert this to a statement about the Lyapunov exponents for $T_{\omega}^{(n)}$, $\lambda_1 = 0$, and $\lambda_2$, we apply Lemma \ref{lem:multiple-iter-LEs} in the case of $k=2$, to obtain: \[ \lambda_2 \leq \frac{\restr{\mu}{\Omega_1}(G_P)}{2k_P}\log\left( \tanh\left( -\frac{1}{4}\log\left( \frac{2(1+\nu)}{1-\nu}(1 + \nu{}a)\right) + \frac{1}{4}k_P\log(D_{\epsilon_1,\epsilon_2})\right) \right) < 0. \] All of the terms that make up the bound here are explicitly computable for specific examples. This concludes the proof.
\end{proof}

\subsection{Perturbation Asymptotics}
\label{subsect:app-pert-asympt}

The map $T_{0,0}$ that generates a cocycle over the one-point space by taking powers of the map has a two-dimensional eigenspace corresponding to the invariant densities $\mathds{1}_{[-1,0]}$ and $\mathds{1}_{[0,1]}$. If instead, we allow $\epsilon_1$ and $\epsilon_2$ to be larger than zero, then $T_{\epsilon_1,\epsilon_2}$ only has a single invariant probability density, rather than two, and there is a spectral gap between the eigenvalue $1$ and the next largest eigenvalues. If $\epsilon_1$ and $\epsilon_2$ are both non-zero but tend to zero, the spectral gap shrinks towards zero (see \cite{keller-liverani}). We can say something similar in the situation of our cocycle of paired tent maps, with the Lyapunov exponents, where we can identify the order of the convergence rate. For notation, we say that $f(x) \sim_{0} g(x)$ when $\lim\limits_{x \to 0^+} \frac{f(x)}{g(x)} = 1$, and we say that $a(x) \lesssim_{0} b(x)$ when there exist $c(x)$ and $d(x)$ such that $a(x) \sim_{0} c(x)$, $b(x) \sim_{0} d(x)$, and $c(x) \leq d(x)$.

\begin{thm}
\label{thm:LE-asympt}
Let $(\Omega,\mu,\sigma)$, $\epsilon_1$, and $\epsilon_2$ be as in Theorem \ref{thm:spec-gap}. Let $\kappa \in (0,1]$, and define $T_{\omega} := T_{\kappa\epsilon_1(\omega),\kappa\epsilon_2(\omega)}$. Then there exists $c > 0$ such that the second-largest Lyapunov exponent $\lambda_2$ of the cocycle of Perron-Frobenius operators satisfies $\lambda_2(\kappa) \lesssim_{0} -c\kappa$. This estimate is sharp, in the following sense: there is a sequence $(\kappa_n)_{n=1}^{\infty} \subset (0,1/2)$ such that $\kappa_n \to 0$, each $T_{\kappa_n}$ is Markov, and if we take $(\Omega,\mu,\sigma) = (\{*\},\delta_{*},\id)$ with $T_{*} = T_{\kappa_n}$, then $\lambda_2(\kappa_n) \sim_{0} -2\kappa_n$.
\end{thm} 

The corollary says that there is an upper bound for $\lambda_2(\kappa)$ that is asymptotically linear in $\kappa$, and there are examples of these cocycles where $\lambda_2(\kappa)$ is exactly asymptotically linear in $\kappa$. The statement of Theorem C in the Introduction is obtained from Theorem \ref{thm:LE-asympt} by unravelling the limits involved in the asymptotic calculations and making $c$ closer to $0$ if necessary.

It turns out that the $C(\kappa)$ from Theorem \ref{thm:spec-gap} is asymptotically proportional to $\kappa/\log(\kappa)$, using standard asymptotic equivalences. The improvement in Theorem \ref{thm:LE-asympt} comes from allowing the wait time $k_P$ to increase and meanwhile being much more careful about the mass that leaks between the two intervals. As mentioned before Theorem \ref{thm:spec-gap}, we choose to present both upper bounds to illustrate multiple outcomes of the same technique that serve somewhat different purposes.

The proof of the theorem is spread over the following two sections; the asymptotic upper bound is proven in the remainder of the current section, and the example providing an exact rate for $\lambda_2(\kappa)$ is provided by Proposition \ref{prop:markov-tent-LE} in Section \ref{subsect:markov-paired-tent}. 

We begin with simple computational equalities for the $M$ and $D$ terms.

\begin{lem}
\label{lem:D-M-eps-props}
We have, for $\kappa \in (0,1]$:
\begin{gather*}
M_{\kappa\epsilon_1,\kappa\epsilon_2} = \kappa M_{\epsilon_1,\epsilon_2}, \\
D_{\kappa\epsilon_1,\kappa\epsilon_2} = 4(1+\kappa B_{\epsilon_1,\epsilon_2})^2,
\end{gather*}
where $B_{\epsilon_1,\epsilon_2} = \max\{\esssup(\epsilon_1),\esssup(\epsilon_2)\}$.
\end{lem}

In general, one would expect multiple leakage times and mixing within the halves as the system runs. Choosing $k_P$ to be larger and accounting for the times at which leakage occurs means that we can account for more leakage terms. This requires a better choice of the set $G_P$, as well; we will choose $G_P$ using Birkhoff's pointwise ergodic theorem to find a set where the frequencies of leaking (in both directions) is high.

The last ingredient in the proof is a firm understanding of how to account for preimages and leaking in the lower bound. Set $I^- = [-1,0]$ and $I^+ = [0,1]$, and for each $\omega$, define two operators on $L^1[-1,1]$: \[ Q_{\omega}^- = \mathds{1}_{I^-}\cdot P_{\omega}, \quad Q_{\omega}^+ = \mathds{1}_{I^+}\cdot P_{\omega}. \] If $f$ is a density on $[-1,1]$ and $s\in \{-,+\}$, then $Q_{\omega}^s(f)$ is the restriction of the density $P_{\omega}(f)$ to the interval $I^s$; $Q_{\omega}^s$ captures information about mass staying in the same interval or leaking over. The next lemma collects the required properties of $Q_{\omega}^{\pm}$.

\begin{lem}
\label{lem:Q-omega-props}
Let $\omega \in \Omega$.
\begin{enumerate}
\item The operators $Q_{\omega}^-$ and $Q_{\omega}^+$ are linear and satisfy $P_{\omega} = Q_{\omega}^- + Q_{\omega}^+$.
\item For non-negative $f\in L^{1}[-1,1]$ and $s\in \{-,+\}$, $Q_{\omega}^s(f) \geq 0$.
\item For $n\geq 1$, let $\Gamma \subset \{-,+\}^n$ be a collection of strings of the form $b = (b_0b_1\dots b_{n-1})$, and for such a string $b$ let $Q_{\omega}^{(b)}$ be defined by \[ Q_{\omega}^{(b)} = Q_{\sigma^{2(n-1)}(\omega)}^{b_{n-1}} \circ \cdots \circ Q_{\omega}^{b_0}. \] Then for any non-negative $f \in L^1[-1,1]$, $P_{\omega}^{(n)}(f) \geq \sum_{b\in \Gamma} Q_{\omega}^{(b)}(f).$
\item For $n \geq 1$, $b = (b_0b_1\dots b_{n-1})\in \{-,+\}^n$, and $x\in [-1,1]$, define \[ \Phi_{\omega}^{b}(x) = \left(S_{\omega}^{(n)}\right)^{-1}\{x\} \cap \bigcap_{i=1}^{n} \left(S_{\omega}^{(i)}\right)^{-1}(I^{b_{i-1}}). \] This is the set of points $z$ that in $n$ steps of the cocycle $S$ are mapped to $x$ and that at each intermediate step $i$ are found in $I^{b_i}$. Then for any $f\in L^{1}[-1,1]$ and almost every $x\in [-1,1]$, we have \[ Q_{\omega}^{(b)}(f)(x) = \sum_{z\in \Phi_{\omega}^{b}(x)} \left( \prod_{i=0}^{n-1} g_{\sigma^{2i}(\omega)}(S_{\omega}^{(i)}(z)) \right) f(z), \] and in particular for any non-trivial interval $I \subset [-1,1]$, we have: \[ Q_{\omega}^{(b)}(\mathds{1}_I)(x) \geq \abs{I\cap \Phi_{\omega}^{b}(x)}\cdot \prod_{i=0}^{n-1} \essinf(g_{\sigma^{2i}(\omega)}), \] where $\abs{I\cap \Phi_{\omega}^{b}(x)}$ is the (finite) cardinality of $I\cap \Phi_{\omega}^{b}(x)$.
\end{enumerate}
\end{lem}

\begin{proof}
The first two parts follow directly from the definition. For the third part, observe that for $n = 1$ this is clear, by the first two parts. Suppose that it holds for some $n \geq 1$. Let $\Gamma \subset \{-,+\}^{n+1}$, and let \[ \Gamma' = \set{b' \in \{-,+\}^n}{ (b-) \text{ or } (b+) \in \Gamma}. \] Then we have, for $f\in L^1[-1,1]$:
\begin{align*}
P_{\omega}^{(n+1)}(f) & = Q_{\sigma^{2n}(\omega)}^-(P_{\omega}^{(n)}(f)) + Q_{\sigma^{2n}(\omega)}^+(P_{\omega}^{(n)}(f)) \\
& \geq Q_{\sigma^{2n}(\omega)}^- \left( \sum_{b'\in \Gamma'} Q_{\omega}^{(b)}(f) \right) + Q_{\sigma^{2n}(\omega)}^+ \left( \sum_{b'\in \Gamma'} Q_{\omega}^{(b)}(f) \right) \\
& = \sum_{b\in (\Gamma'-)\cup(\Gamma'+)} Q_{\omega}^{(b)}(f) \geq \sum_{b\in \Gamma} Q_{\omega}^{(b)}(f). 
\end{align*}
By induction, the inequality holds for all $n \geq 1$.

Lastly, let $f\in L^1[-1,1]$, $x\in [-1,1]$, and consider $s \in \{-,+\}$. Then \[ Q_{\omega}(f)(x) = \mathds{1}_{I^s}(x) \sum_{z\in S_{\omega}^{-1}\{x\}} g_{\omega}(z) f(z) = \sum_{z\in \Phi_{\omega}^{s}(x)} g_{\omega}(z) f(z), \] where the equality holds because $\Phi_{\omega}^{s}(x) = S_{\omega}^{-1}\{x\}\cap S_{\omega}^{-1}(I^s)$ and $x\in I^s$ if and only if $z\in S_{\omega}^{-1}(I^s)$. Suppose that the equality holds for all $b' \in \{-,+\}^n$, and consider $b = (b'b_n) \in \{-,+\}^{n+1}$. We then have: 
\begin{align*} 
Q_{\omega}^{(b)}(f)(x) & = Q_{\sigma^{2n}(\omega)}^{b_n}(Q_{\omega}^{(b')}(f))(x) \\
& = \mathds{1}_{I^{b_n}}(x) \sum_{y\in S_{\sigma^{2n}(\omega)}^{-1}\{x\}} g_{\sigma^{2n}(\omega)}(y) \sum_{z\in \Phi_{\omega}^{b'}(y)} \left( \prod_{i=0}^{n-1} g_{\sigma^{2i}(\omega)}(S_{\omega}^{(i)}(z)) \right) f(z) \\
& = \sum_{z\in \Phi_{\omega}^{b}(x)} \left( \prod_{i=0}^{n} g_{\sigma^{2i}(\omega)}(S_{\omega}^{(i)}(z)) \right) f(z),
\end{align*}
where we see that the intermediate $y$ had to be in $I^{b_{n-1}}$ and $x\in I^{b_n}$, so $z\in \Phi_{\omega}^{b}(x)$. This proves the form of $Q_{\omega}^{b}(f)$. Specifying $f = \mathds{1}_I$, we have that the sum is over preimages $z \in \Phi_{\omega}^{b}(x)\cap I$, so we easily obtain: \[ Q_{\omega}^{(b)}(\mathds{1}_I)(x) = \sum_{z\in I\cap \Phi_{\omega}^{b}(x)} \prod_{i=0}^{n-1} g_{\sigma^{2i}(\omega)}(S_{\omega}^{(i)}(z)) \geq \abs{I\cap \Phi_{\omega}^{b}(x)}\cdot \prod_{i=0}^{n-1} \essinf(g_{\sigma^{2i}(\omega)}), \] where the absolute value term denotes cardinality (in this case, the number of $n$-th preimages of $x$ that are in $I$ and are in $I^{b_i}$ at the $i+1$-st time). 
\end{proof}

\begin{proof}[Proof of the upper bound in Theorem \ref{thm:LE-asympt}]
As in the proof of Lemma \ref{lem:leakage-time}, denote $\Omega_2 = \Omega\setminus \Omega_1$, and define subsets $G_1$ and $G_2$ of $\Omega_1$ by \[ G_k = \left( \Omega_1\cap \epsilon_k^{-1}\left[ M_{\epsilon_1,\epsilon_2}, 1 \right] \right) \cup \left( \sigma^{-1}\left( \Omega_2\cap \epsilon_k^{-1}\left[ M_{\epsilon_1,\epsilon_2}, 1 \right] \right) \right), \] for $k=1,2$. These two sets are the sets where leakage happens from $[-1,0]$ to $[0,1]$ and vice versa, respectively, with some minimum amount. Note that by Lemma \ref{lem:D-M-eps-props}, we have \[ \kappa \epsilon_i(\omega) \geq M_{\kappa\epsilon_1,\kappa\epsilon_2} = \kappa M_{\epsilon_1,\epsilon_2} \quad \iff \quad \epsilon_i(\omega) \geq M_{\epsilon_1,\epsilon_2}. \] Because of this, $G_1$ and $G_2$ are independent of $\kappa \in (0,1]$, even though the amount of leakage scales with $\kappa$. 

By Birkhoff's pointwise ergodic theorem, we know that \[ \frac{1}{n} \sum_{i=0}^{n-1} \mathds{1}_{G_k}\circ \sigma^{2i} \conv{n\to\infty} \restr{\mu}{\Omega_1}(G_k) \] $\mu$-almost everywhere for $k = 1,2$. We can find $N_1,\N_2 \in \Z_{\geq 1}$ and sets $C_1,C_2 \subset \Omega_1$ such that $\restr{\mu}{\Omega_1}(C_1\cap C_2) \geq 1/2$ and when $n\geq N_k$ and $\omega \in C_k$, then for $f = \frac{1}{2}\min\{\restr{\mu}{\Omega_1}(G_1),\restr{\mu}{\Omega_1}(G_2)\}$, the orbit $(\sigma^{2i}(\omega))_{i=0}^{n-1}$ is in $G_k$ at least $fn$ times. (See Lemma \ref{lem:ae-conv}.) Set $N_0 = \max\{N_1,N_2\}$ and $\cG' = C_1 \cap C_2$ and observe that for all $n \geq N_0$ and $\omega \in \cG'$, the orbit $(\sigma^{2i}(\omega))_{i=0}^{n-1}$ is in $G_k$ at least $fn$ times for each $k=1,2$.

Let $\kappa_0 \in (0,1)$ be a solution of $\displaystyle N_0 = \left\lceil -\frac{\log(M_{\kappa\epsilon_1,\kappa\epsilon_2})}{\log(4)} \right\rceil$, and for $\kappa \leq \kappa_0$, set \[ k_P = m_1 + 2\left\lceil -\frac{\log(M_{\kappa\epsilon_1,\kappa\epsilon_2})}{\log(4)} \right\rceil = m_1 + 2m_3(\kappa), \] where $m_1 = m_1((2a)^{-1})$ is the waiting time until any interval of measure $(2a)^{-1}$ covers one of $I^-$ or $I^+$ and $m_3(\kappa) = \left\lceil -\frac{\log(M_{\kappa\epsilon_1,\kappa\epsilon_2})}{\log(4)} \right\rceil$ is the number of steps it takes for a leaked interval of size $M_{\kappa\epsilon_1,\kappa\epsilon_2}$ to expand to cover one of $I^+$ or $I^-$, both as in Corollary \ref{cor:tent-dyn-cover}. Let $G'_P = \sigma^{-2m_1}(\cG')$ and let $\cQ = \{Q_k\}_{k=1}^{2a}$ be the same partition of $[-1,1]$ into closed intervals, each with measure $\lambda(Q_k) = (2a)^{-1}$.

Observe that each $Q_k$ is contained in one of $I^-$ and $I^+$; we claim that $S_{\omega}^{(m_1)}(Q_k)$ is contained in the same half of $[-1,1]$. To see this, note that the computation for $m_1((2a)^{-1})$ involved following an interval along its orbit under $S_{\omega}^{(n)}$, and if the interval ever intersected both sets, we only followed the larger piece. For small $\kappa$, we have $\kappa \ll (2a)^{-1}$, which means that the amount of leakage at any step of the orbit is small compared to the measure of the larger piece. Hence the bulk of the set $S_{\omega}^{(i)}(Q_k)$ lies in the same $I^s$ in which it starts, and so indeed $S_{\omega}^{(m_1)}(Q_k)$ covers that $I^s$.

For fixed $\kappa \in (0,\kappa_0]$, we will be using the lower bound \[ \essinf(g_{\omega}) \geq \frac{1}{D_{\kappa\epsilon_1,\kappa\epsilon_2}} = \frac{1}{4(1+\kappa B_{\epsilon_1,\epsilon_2})^2}. \]
We will now find a lower bound for $P_{\omega}^{(k_P(\kappa))}(\mathds{1}_{Q_k})$; the primary tool is Lemma \ref{lem:Q-omega-props}.

\begin{claim*}
Suppose that $Q_k \subset I^s$, for $s\in \{-,+\}$.
\begin{itemize}
\item If $x\in I^s$, then \[ P_{\omega}^{(k_P(\kappa))}(\mathds{1}_{Q_k})(x) \geq (1+\kappa B_{\epsilon_1,\epsilon_2})^{-2k_P(\kappa)} \cdot 4^{-m_1}.\]
\item If $x\in I^{-s}$, then \[ P_{\omega}^{(k_P(\kappa))}(\mathds{1}_{Q_k})(x) \geq \frac{1}{(1+\kappa B_{\epsilon_1,\epsilon_2})^{2k_P(\kappa)}} \cdot \frac{1}{2\cdot 4^{m_1}} \cdot \frac{f\cdot m_3(\kappa)}{4^{m_3(\kappa)}}. \]
\end{itemize}
\end{claim*}

\begin{proof}[Proof of claim]
Fix $Q_k \subset I^s$, for some $s\in \{-,+\}$, and fix $\omega \in G'_P$; we will use the fact that for such $\omega$, $\sigma^{2m_1}(\omega)\in \cG'$. By the above comment, we see that for any $x\in I^s$, there is at least one $m_1$-th preimage $z\in Q_k$ for $x$ that stays in $I^s$ along its orbit; that is, if $b = (s\dots s) \in \{-,+\}^{m_1}$, then $\Phi_{\omega}^{b}(x)$ is non-empty. By part (4) of Lemma \ref{lem:Q-omega-props}, we have \[ Q_{\omega}^{(b)}(\mathds{1}_{Q_k})(x) \geq \abs{I\cap \Phi_{\omega}^{b}(x)}\cdot \prod_{i=0}^{n-1} \essinf(g_{\sigma^{2i}(\omega)}) \geq \left(\frac{1}{4(1+\kappa B_{\epsilon_1,\epsilon_2})^2}\right)^{m_1}. \] We then obtain \[ P_{\omega}^{(k_P(\kappa))}(\mathds{1}_{Q_k}) \geq \left(\frac{1}{4(1+\kappa B_{\epsilon_1,\epsilon_2})^2}\right)^{m_1} P_{\sigma^{2m_1}(\omega)}^{(2m_3(\kappa))}(\mathds{1}_{I^s}). \]

We now consider the two cases. First, consider the case where $x\in I^s$. Then $x$ has exactly four preimages in $I^s$, and each of those preimages has four preimages in $I^s$, and so on. If we let $b \in \{-,+\}^{2m_3(\kappa)}$ be the string whose components are all $s$, then we see that $\abs{\Phi_{\omega}^{b}(x)} = 4^{2m_3(\kappa)}$. By parts (3) and (4) of Lemma \ref{lem:Q-omega-props}, we have:
\begin{align*}
P_{\omega}^{(k_P(\kappa))}(\mathds{1}_{Q_k})(x) & \geq \left(\frac{1}{4(1+\kappa B_{\epsilon_1,\epsilon_2})^2}\right)^{m_1} P_{\sigma^{2m_1}(\omega)}^{(2m_3(\kappa))}(\mathds{1}_{I^s})(x) \\
& \geq \left(\frac{1}{4(1+\kappa B_{\epsilon_1,\epsilon_2})^2}\right)^{m_1} Q_{\sigma^{2m_1}(\omega)}^{(b)}(\mathds{1}_{I^s})(x) \\
& \geq \frac{4^{2m_3(\kappa)}}{\left(4(1+\kappa B_{\epsilon_1,\epsilon_2})^2\right)^{m_1+2m_3(\kappa)}} = (1+\kappa B_{\epsilon_1,\epsilon_2})^{-2k_P(\kappa)} \cdot 4^{-m_1}.
\end{align*}

Second, consider the case where $x\in I^{-s}$; this is where we use $\sigma^{2m_1}(\omega) \in \cG'$. We will split up the remaining iterates of the cocycle into two stages: in the first stage, we will wait $m_3(\kappa)$ steps and count the number of times the map leaks $I^s$ into the other half of the space (which we will denote $I^{-s}$), and in the second stage, we will let the map mix for another $m_3(\kappa)$ steps.

Since $\sigma^{2m_1}(\omega) \in \cG'$ and $m_3(\kappa) \geq N_0$ (by choice of $\kappa$), we see that in the next $m_3(\kappa)$ steps, the cocycle leaks $[-1,0]$ to $[0,1]$ and $[0,1]$ to $[-1,0]$ at least $f\cdot m_3(\kappa)$ times each. Let $0 \leq n_1^- < \cdots < n_{l(-)}^- \leq m_3(\kappa)-1$ and $0 \leq n_1^+ < \cdots < n_{l(+)}^+ \leq m_3(\kappa)-1$ be the times when leakage happens from $[-1,0]$ to $[0,1]$ and from $[0,1]$ to $[-1,0]$, respectively, along the cocycle $S_{\sigma^{2m_1}(\omega)}^{(m_3(\kappa))}$; we know that $l(-),l(+) \geq f\cdot m_3(\kappa)$.

For $1 \leq i \leq l(s)$, let $b^i \in \{-,+\}^{2m_3(\kappa)}$ be the string whose first $n_i^s-1$ components are $s$, and the rest of the components are $-s$. This string follows points that stay in $I^s$ for a while and then switch to $I^{-s}$ and remain there. Let $\Gamma = \{b^i\}_{i=1}^{l(s)}$ be the collection of these strings. By parts (3) and (4) of Lemma \ref{lem:Q-omega-props}, we have:
\begin{align*}
P_{\sigma^{2m_1}(\omega)}^{(2m_3(\kappa))}(\mathds{1}_{I^s})(x) & \geq \sum_{i=1}^{l(s)} Q_{\sigma^{2m_1}(\omega)}^{(b^i)}(\mathds{1}_{I^s})(x) \\
& \geq \left( \frac{1}{4(1+\kappa B_{\epsilon_1,\epsilon_2})^2}\right)^{2m_3(\kappa)}\cdot \sum_{i=1}^{l(s)} \abs{I^s\cap \Phi_{\sigma^{2m_1}(\omega)}^{b^i}(x)}.
\end{align*}
We need to compute $\abs{I^s \cap \Phi_{\sigma^{2m_1}(\omega)}^{b^i}(x)}$. Here, we must identify how the map acts over the $2m_3(\kappa)$ steps. Over the first $n_i^s-1$ steps, the interval $I^s$ undergoes mixing, and at the end any point in $I^s$ will have $4^{n_i^s-1}$ preimages in $I^s$ that stay in $I^s$ for all $n_1^s-1$ steps. At the $n_i^s$-th step, we follow the leakage to $I^{-s}$; for $y \in I^{-s}\cap S_{\sigma^{2(m_1+n_1^s)}(\omega)}(I^s)$, there are at least $2$ preimages of $y$ in $I^s$ (the map is symmetric). Over the next $m_3(\kappa)$ steps, the leakage expands to cover all of $I^{-s}$, and this could potentially happen in an invertible way. Over the remaining steps, $I^{-s}$ mixes with itself in the same way that $I^s$ mixes. Putting all of this together tells us that:
\begin{align*}
\abs{I^s \cap \Phi_{\sigma^{2m_1}(\omega)}^{b^i}(x)} & \geq 4^{2m_3(\kappa)-m_3(\kappa)-n_i^s}\cdot 1\cdot 2\cdot 4^{n_i^s-1} \\
& = 4^{m_3(\kappa)}\cdot 2^{-1}. 
\end{align*}
Putting all of this information together, we obtain, for $x\in I^{-s}$ and $Q_k \subset I^{s}$:
\begin{align*}
P_{\omega}^{(k_P(\kappa))}(\mathds{1}_{Q_k})(x) & \geq \left(\frac{1}{4(1+\kappa B_{\epsilon_1,\epsilon_2})^2}\right)^{m_1} P_{\sigma^{2m_1}(\omega)}^{(2m_3(\kappa))}(\mathds{1}_{I^s})(x) \\
& \geq \frac{1}{(1+\kappa B_{\epsilon_1,\epsilon_2})^{2k_P(\kappa)}} \cdot \frac{1}{4^{m_1+2m_3(\kappa)}} \sum_{i=1}^{l(s)} 4^{m_3(\kappa)}\cdot 2^{-1} \\
& = \frac{1}{(1+\kappa B_{\epsilon_1,\epsilon_2})^{2k_P(\kappa)}} \cdot \frac{1}{2\cdot 4^{m_1}} \cdot \frac{l(s)}{4^{m_3(\kappa)}} \\
& \geq \frac{1}{(1+\kappa B_{\epsilon_1,\epsilon_2})^{2k_P(\kappa)}} \cdot \frac{1}{2\cdot 4^{m_1}} \cdot \frac{f\cdot m_3(\kappa)}{4^{m_3(\kappa)}}. \qedhere
\end{align*}
\end{proof}

The last term is clearly smaller than $1$, so this lower bound is the lower bound for all $x$. Hence, for $\omega \in G'_P$ we have: \[ \essinf(P_{\omega}^{k_P(\kappa)}(\mathds{1}_{Q_k})) \geq \frac{1}{(1+\kappa B_{\epsilon_1,\epsilon_2})^{2k_P(\kappa)}} \cdot \frac{1}{2\cdot 4^{m_1}} \cdot \frac{f\cdot m_3(\kappa)}{4^{m_3(\kappa)}} =: \gamma(\kappa). \] Note that clearly, $\gamma(\kappa) \conv{\kappa \to 0^+} 0$.

Finally, we obtain our diameter bound, apply Theorem \ref{thm:pf}, and perform the asymptotic calculations. As at the end of the proof of Theorem \ref{thm:spec-gap}, we now have that \[ \diam_{\theta_a}(P_{\omega}^{k_P(\kappa)}\cC_a) \leq 2\log\left( \frac{2(1+\nu)}{1-\nu}(1 + \nu{}a)\right) - 2\log(\gamma(\kappa)), \] using Lemma \ref{lem:proj-met-bound}. Call this diameter bound $D'_P = c_1 - 2\log(\gamma(\kappa))$, writing $c_1$ for the constant involving $a$ and $\nu$. Write \[ c_2 = \frac{e^{-c_1/2}\restr{\mu}{\Omega_1}(G'_P)fM_{\epsilon_1,\epsilon_2}}{2\cdot 4^{m_1}}. \] For $\kappa \in (0,\kappa_0]$, we apply the cocycle Perron-Frobenius theorem, Theorem \ref{thm:pf}, to the cocycle $S_{\omega}^{(n)}$ with parameters $G'_P,k_P,D'_P$, to get an upper bound $C_1(\kappa)$ for $\lambda_2(\kappa)$. We have, using standard asymptotic estimates and the definition of $k_P$,
\begin{align*}
C_1(\kappa) & = \frac{\restr{\mu}{\Omega_1}(G'_P)}{k_P(\kappa)}\log\left( \tanh\left( \frac{1}{4}\left( c_1 - 2\log(\gamma(\kappa)) \right) \right) \right) \\
& \sim_{0} \frac{-2e^{-c_1/2}\restr{\mu}{\Omega_1}(G'_P)}{m_1 + 2m_3(\kappa)}\cdot \gamma(\kappa). 
\end{align*}
Moreover, we have \[ m_3(\kappa) = \left\lceil -\frac{\log(M_{\kappa\epsilon_1,\kappa\epsilon_2})}{\log(4)} \right\rceil \sim_{0} -\frac{\log(\kappa)}{\log(4)}. \] This allows us to show that \[ \frac{1}{(1+\kappa B_{\epsilon_1,\epsilon_2})^{2k_P(\kappa)}} = \exp \big( -2(m_1+2m_3(\kappa))\log(1+\kappa B_{\epsilon_1,\epsilon_2}) \big) \sim_{0} 1, \] that $4^{-m_3(\kappa)} \sim_{0} \kappa M_{\epsilon_1,\epsilon_2}$, and that $\displaystyle \gamma(\kappa) \sim_{0} \frac{f\cdot m_3(\kappa)}{2\cdot 4^{m_1}}\cdot \kappa M_{\epsilon_1,\epsilon_2}.$ Thus we obtain: 
\begin{align*}
C_1(\kappa) & \sim_{0} \frac{-2e^{-c_1/2}\restr{\mu}{\Omega_1}(G')}{m_1 + 2m_3(\kappa)}\cdot \gamma(\kappa) \\
& \sim_{0} \frac{-2e^{-c_1/2}\restr{\mu}{\Omega_1}(G')}{m_1 + 2m_3(\kappa)}\cdot \frac{f\cdot m_3(\kappa)}{2\cdot 4^{m_1}}\cdot \kappa M_{\epsilon_1,\epsilon_2} \\
& \sim_{0} -c_2 \kappa.
\end{align*}
Thus we have $\lambda_2(\kappa) \leq C_1(\kappa) \sim_{0} -c_2\kappa$, which is the desired result in Theorem \ref{thm:LE-asympt}.
\end{proof}

\subsection{Markov Paired Tent Maps}
\label{subsect:markov-paired-tent}

It remains to show the sharpness of the upper bound in Theorem \ref{thm:LE-asympt}, in the sense of asking whether or not the second largest Lyapunov exponent for the Perron-Frobenius cocycle is ever linear in $\kappa$. \emph{A priori}, it could never be the case that $\lambda_2(\kappa)$ is linear in $\kappa$, even though there is a linear upper bound. The following example eliminates this possibility by providing a sequence of paired tent maps $T_n := T_{\kappa_n,\kappa_n}$ which, when considered as cocycles over the trivial one-point dynamics, have second-largest Lyapunov exponent of the P-F cocycle $P_{T_n}$ asymptotically equivalent to $-2\kappa_n$. These $T_n$ will be chosen to be Markov maps, allowing for explicit bounds on the spectrum of $P_{T_n}$ and hence the Lyapunov spectrum.

\begin{prop}
\label{prop:markov-tent-LE}
There exists a sequence of paired tent maps $T_n := T_{\kappa_n,\kappa_n}$ with $\kappa_n \to 0$ such that the second-largest Lyapunov exponent for each Perron-Frobenius operator $P_{T_n}$ (generating the trivial cocycle over the one-point space), denoted $\lambda_2(n)$, satisfies $\lambda_2(n) \sim_{0} -2\kappa_n$.
\end{prop}

\begin{proof}[Proof outline]
The first thing to do is to find $\kappa_n$ where $T_{\kappa_n,\kappa_n}$ is Markov. This happens when the images of the intervals of monotonicity generate a finite partition of $[-1,1]$, and it is relatively easy to show by pushing forward the points $\pm 1/2$ that this occurs when $\kappa_n$ solves $(2+2\kappa)^{n}\kappa = 1$. The explicit Markov partition is found to be, for $n=1$, \[ \Big\{ \left[ -1,-\tfrac{1}{2} \right], \left[ -\tfrac{1}{2}, -\kappa_1 \right], \left[ -\kappa_1, 0 \right], \left[ 0, \kappa_1 \right], \left[ \kappa_1, \tfrac{1}{2} \right], \left[ \tfrac{1}{2}, 1 \right] \Big\} \] and for $n \geq 2$,
\begin{align*}
\Big\{ \left[ -1, T_n(-\kappa_n) \right] \Big\} & \cup \Big\{ \left[ T_n^{i}(-\kappa_n), T_n^{i+1}(-\kappa_n) \right] \Big\}_{i=1}^{n-2} \\
& \cup \Big\{ \left[ T_n^{n-1}(-\kappa_n),-\tfrac{1}{2} \right], \left[ -\tfrac{1}{2}, -\kappa_n \right], \left[ -\kappa_n, 0 \right] \Big\} \\
& \cup \Big\{ \left[ 0, \kappa_n \right], \left[ \kappa_n, \tfrac{1}{2} \right], \left[ \tfrac{1}{2}, T_n^{n-1}(\kappa_n) \right] \Big\} \\
& \cup \Big\{ \left[ T_n^{i+1}(\kappa_n), T_n^i(\kappa) \right] \Big\}_{i=1}^{n-2} \cup \Big\{ \left[ T_n(\kappa_n), 1 \right] \Big\}.
\end{align*}
The partition and the graph of $T_{\kappa_n,\kappa_n}$ for $n=4$ is pictured in Figure \ref{fig:markov-tent-4-grid-square-centered}.

\begin{figure}[tbp]
\includegraphics[width=\textwidth]{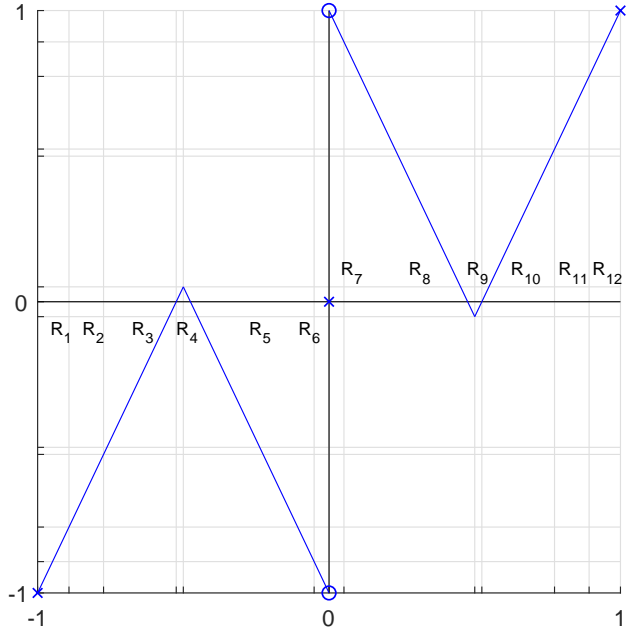}
\caption{Markov partition for $T_4$.}
\label{fig:markov-tent-4-grid-square-centered}
\end{figure}

\begin{figure}[tbp]
\[ \left[ \begin{array}{@{}c|c@{}}
\begin{matrix}
1 & 0 & 0 & & 0 & 0 & 0 & 1  \\
1 & 0 & 0 & & 0 & 0 & 1 & 0 \\
0 & 1 & 0 & & 0 & 0 & 1 & 0 \\
0 & 0 & 1 & & 0 & 0 & 1 & 0 \\
& \vdots & & \ddots & & & \vdots & \\
0 & 0 & 0 & & 1 & 0 & 1 & 0 \\
0 & 0 & 0 & & 1 & 0 & 1 & 0 \\
0 & 0 & 0 & & 1 & 0 & 1 & 0
\end{matrix}
&
\begin{matrix}
0 & 0 & 0 & 0 & & 0 & 0 & 0 \\
& & & & & & & \\
& & & & & & & \\
& & & & & & & \\
& & & & \vdots & & & \\
& & & & \hphantom{\ddots} & & & \\
0 & 0 & 0 & 0 & & 0 & 0 & 0 \\
0 & 1 & 1 & 0 & & 0 & 0 & 0
\end{matrix} \\
\hline
\begin{matrix}
0 & 0 & 0 & & 0 & 1 & 1 & 0 \\
0 & 0 & 0 & & 0 & 0 & 0 & 0 \\
& & & & & & &  \\
& & & \vdots & & & & \\
& & & & & & & \\
& & & \hphantom{\ddots} & & & & \\
& & & & & & & \\
0 & 0 & 0 & & 0 & 0 & 0 & 0
\end{matrix}
&
\begin{matrix}
0 & 1 & 0 & 1 & & 0 & 0 & 0 \\
0 & 1 & 0 & 1 & & 0 & 0 & 0 \\
0 & 1 & 0 & 1 & & 0 & 0 & 0 \\
& \vdots & & & \ddots & & \vdots & \\
0 & 1 & 0 & 0 & & 1 & 0 & 0 \\
0 & 1 & 0 & 0 & & 0 & 1 & 0 \\
0 & 1 & 0 & 0 & & 0 & 0 & 1 \\
1 & 0 & 0 & 0 & & 0 & 0 & 1
\end{matrix}
\end{array} \right] \]
\caption{General form of the $(2n+4)$-by-$(2n+4)$ adjacency matrix $A_n$.}
\label{fig:adj-matrix}
\end{figure}

Now, the Perron-Frobenius operator $P_n := P_{T_n}$ preserves the finite-dimensional subspace of characteristic functions on the Markov partition. For each $n$, let $V_n = \subspan_{\C}\set{\mathds{1}_{R_i}}{ 1 \leq i \leq 2n+4}$, where $\{R_i\}$ is the Markov partition for $T_n$ (written left-to-right). The restriction of $P_n$ to $V_n$ has representation $M_n = (2+2\kappa_n)^{-1}A_n$, where $A_n$ is the adjacency matrix for $T_n$ with respect to $\{R_i\}_{i=1}^{2n+4}$, and so the spectrum of $M_n$ is $\sigma(M_n) = (2+2\kappa_n)^{-1}\sigma(A_n)$. The general form of the adjacency matrix $A_n$ is given in Figure \ref{fig:adj-matrix}. 

The spectral radius of $A_n$ is $2+2\kappa_n$, and the characteristic polynomial of $A_n$ is
\begin{align*} 
\mathrm{char}_{A_n}(x) & = x^{2n+4}-4x^{2n+3}+4x^{2n+2}-4x^2 \\
& = x^2(x^{n+1}-2x^n - 2)(x^{n+1}-2x^n + 2) \\
& = x^2(x^n(x-2) - 2)(x^n(x-2) + 2).
\end{align*}
From here, some calculus and complex analysis determines that $2+2\kappa_n$ is the largest root, and for $n\geq 5$ there is a second-largest real root $2-2r_n$, where $r_n$ is asymptotically $\kappa_n$; there are two roots at zero, and all other roots are close to the unit circle. Relating these eigenvalues of $A_n$ to the Lyapunov exponents of $P_n$ indicates that the largest Lyapunov exponent is $0$ and that the second-largest Lyapunov exponent is asymptotically $-2\kappa_n$.
\end{proof}

\appendix

\section{Miscellaneous ergodic theory}
\label{append:misc-erg}

\begin{lem}
\label{lem:first-entry-conv}
Let $(\Omega,\cB,\mu,\sigma)$ be an ergodic probability-preserving transformation, let $G\in \cB$ have positive measure, and let $N : \Omega \to \Z_{\geq 1}\cup \{\infty\}$ be the first entry time of $\omega$ into $G$: $N(\omega) = \inf\set{n\geq 1}{\sigma^n(\omega) \in G}$. Then $N$ is measurable and for $\mu$-almost every $\omega$, $N(\sigma^n(\omega))/n \conv{n\to\infty} 0$.
\end{lem}

\begin{lem}
\label{lem:weird-birk-sum}
Let $(\Omega,\cB,\mu,\sigma)$ be an ergodic probability-preserving transformation, let $N : \Omega \to \Z_{\geq 1}$ be such that $(N\circ \sigma^n)/n$ converges to $0$ $\mu$-almost everywhere, let $h : \Omega \to \R$ be measurable, and suppose that $n^{-1}\sum_{i=0}^{n-1} h(\sigma^{i}(\omega))$ converges to $\lambda \in \R$ for $\mu$-almost every $\omega$. Then \[ \frac{1}{n} \sum_{i=0}^{N(\sigma^n(\omega))-1} h(\sigma^{i+n}(\omega)) \conv{n\to\infty} 0. \]
\end{lem}

\begin{proof}
Observe that for any $\omega$ and $n$, we have: 
\begin{equation*}
\begin{split}
\frac{1}{n}\sum_{i=0}^{n-1} h(\sigma^{i}(\omega)) & + \frac{1}{n}\sum_{i=0}^{N(\sigma^n(\omega))-1} h(\sigma^{i+n}(\omega)) \\
& = \frac{n+N(\sigma^n(\omega))}{n}\cdot \frac{1}{n+N(\sigma^n(\omega))}\sum_{i=0}^{n+N(\sigma^n(\omega))} h(\sigma^i(\omega)).
\end{split}
\end{equation*}
Since $N(\sigma^n(\omega))/n$ tends to $0$, we see that the desired term converges to $0$, because the two classical Birkhoff sums both converge to $\lambda$ and $(n+N(\sigma^n(\omega)))/n$ converges to $1$, by Lemma \ref{lem:first-entry-conv}.
\end{proof}

\begin{lem}
\label{lem:birk-sums-conv-L1}
Let $(\Omega,\cB,\mu,\sigma)$ be an ergodic probability-preserving transformation, and let $f : \Omega \to \R$ be a measurable function. Suppose that the Birkhoff sums $\displaystyle \frac{1}{n}\sum_{i=0}^{n-1} f(\sigma^i(\omega))$ converge to $\lambda \in \R$ for $\mu$-almost every $\omega$.
\begin{enumerate}
\item If $f \geq 0$, then $f$ is integrable.
\item If $f^+$ is integrable, then $f$ is integrable.
\end{enumerate}
\end{lem}

\begin{lem}
\label{lem:ae-conv}
Suppose that $(\cM,\cB,\mu)$ is a probability space, and $f_n : \cM \to \R_{\geq 0}$ is a sequence of measurable functions converging $\mu$-almost everywhere to the constant $D > 0$. Then for each $\delta \in (0,1]$ and $D'\in [0,D)$, there exists an $N \in \Z_{\geq 1}$ such that the set \[ \bigcap_{n=N}^{\infty} f_n^{-1}\left[ D', \infty \right) \] has measure at least $1-\delta$.  
\end{lem}

\section{$D$-adaptedness of particular cones}
\label{append:d-adapt}

\begin{prop}
\label{prop:bv-cone}
Let $BV[0,1]$ be the Banach space of bounded variation functions under the norm $\norm{f}_{BV} = \max\{\norm{f}_1,\Var(f)\}$, and for $a > 0$ let \[ \cC_a = \set{f\in BV[0,1]}{f(x) \geq 0,\ \Var(f) \leq a \int_0^1 f\ dx} \setminus \{0\}. \] Then $\cC_a$ is a nice cone, with $D$-adapted constant $D = 2a + 1$.
\end{prop}

\begin{proof}
If $f\in \cC_a$ and $c\geq 0$, then \[ \Var(cf) = c\Var(f) \leq ca\int_0^1 f\ dx = a\int_0^1 cf\ dx, \] so $cf\in \cC_a$. If $f,g\in \cC_a$ and $\lambda \in [0,1]$, then
\begin{align*} 
\Var(\lambda f + (1-\lambda) g) & \leq \Var(\lambda f) + \Var((1-\lambda)g) \\
& = \lambda\Var(f) + (1-\lambda)\Var(g) \\
& \leq \lambda a\int_0^1 f\ dx + (1-\lambda) a\int_0^1 g\ dx = a\int_0^1 (\lambda f + (1-\lambda)g)\ dx,
\end{align*}
so $\lambda f + (1-\lambda)g \in \cC_a$. If $f_n\in \cC_a$ for each $n\geq 1$ and $f_n \conv{n\to\infty} f$ in $\norm{\cdot}_{BV}$, then \[ \abs{\int_0^1 f_n\ dx - \int_0^1 f\ dx} \leq \int_0^1 \abs{f_n-f}\ dx \conv{n\to\infty} 0, \] and \[ \abs{\Var(f_n)-\Var(f)} \leq \Var(f_n-f) \conv{n\to\infty} 0, \] so we obtain: \[ \Var(f) = \lim_{n\to\infty} \Var(f_n) \leq \lim_{n\to\infty} a\int_0^1 f_n\ dx = a\int_0^1 f\ dx. \] This shows that $\cC_a\cup\{0\}$ is closed, which is what we want. Clearly $\cC_a$ is blunt and salient. To see that $\cC_a$ is generating, we simply observe that $\mathds{1}$ is an interior point of $\cC_a$.

Now, suppose that $f,g \in BV[0,1]$, and $-f\preceq g \preceq f$. Then we have $f(x)-g(x) \geq 0$ and $g(x)+f(x)\geq 0$ for all $x$, and hence $f(x) \geq 0$ and $\abs{g(x)} \leq f(x)$. This forces $\norm{g}_1 \leq \norm{f}_1$. We need the following lemma:

\begin{lem}
\label{lem:var-ineq}
Let $f,g\in BV[0,1]$. Then \[ \abs{\Var(f)-\Var(g)} \leq \Var(f-g). \]
\end{lem}

\begin{proof}
For a fixed partition $0 = x_0 < \dots < x_n = 1$, we have:
\begin{align*}
\Var(f-g) & \geq \sum_{k=1}^n \abs{f(x_k)-g(x_k)-(f(x_{k-1})-g(x_{k-1}))} \\
& \geq \sum_{k=1}^n \abs{\abs{f(x_k)-f(x_{k-1})}-\abs{g(x_k)-g(x_{k-1})}} \\
& \geq \abs{\sum_{k=1}^n \abs{f(x_k)-f(x_{k-1})} - \sum_{k=1}^n \abs{g(x_k)-g(x_{k-1})}}.
\end{align*}
where the inequalities are justified by the definition, the reverse triangle inequality, and the triangle inequality again, respectively. Then, we have 
\begin{align*} 
\sum_{k=1}^n \abs{f(x_k)-f(x_{k-1})} - \Var(g) & \leq \sum_{k=1}^n \abs{f(x_k)-f(x_{k-1})} - \sum_{k=1}^n \abs{g(x_k)-g(x_{k-1})} \\
& \leq \Var(f-g),\
\end{align*} 
using the above inequalities, so that upon taking the supremum over all finite partitions of $[0,1]$ we get $\Var(f)-\Var(g) \leq \Var(f-g)$. Similarly, \[ \sum_{k=1}^n \abs{g(x_k)-g(x_{k-1}) - \Var(f)} \leq \Var(f-g), \] which implies $\Var(g)-\Var(f) \leq \Var(f-g)$, and so \[ \abs{\Var(f)-\Var(g)} \leq \Var(f-g). \qedhere \]
\end{proof}

Using Lemma \ref{lem:var-ineq}, we get: 
\begin{equation*}
\begin{split} 
\Var(g) - \Var(f) & \leq \Var(f-g) \leq a\int_0^1 f-g\ dx \\
& = a\int_0^1 f\ dx - a\int_0^1 g\ dx \\
& \leq a\int_0^1 f\ dx + a\int_0^1 f\ dx = 2a\norm{f}_1. 
\end{split}
\end{equation*}
Hence, we have: 
\begin{align*} 
\norm{g}_{BV} & = \max\{\norm{g}_1,\Var(g)\} \\
& \leq \max\{\norm{f}_1,2a\norm{f}_1+\Var(f)\} \\
& \leq \max\{\norm{f}_{BV},(2a+1)\norm{f}_{BV}\} \leq (2a+1)\norm{f}_{BV},
\end{align*} 
which is the desired inequality.
\end{proof}

\bibliographystyle{amsplain}
\bibliography{cocycle-pf-refs}

\end{document}